\documentclass[10.5pt]{article}
\usepackage{}
\usepackage{amsmath}
\usepackage{amsfonts}
\usepackage{graphicx}
\usepackage{amssymb}
\usepackage{bm}
\usepackage{color}

\newtheorem{condition**}{A*}
\newtheorem{condition***}{C*}
\newtheorem{condition*}{C}

\newtheorem{example}{Example}[section]
\newtheorem{proposition}{Proposition}[section]
\newtheorem{corollary}{Corollary}[section]
\newtheorem{definition}{Definition}[section]
\newtheorem{theorem}{Theorem}[section]
\newtheorem{lemma}{Lemma}[section]
\newtheorem{remark}{Remark}[section]
\newcommand{\ignore}[1]{}

\def\proof{\noindent{\bf Proof} }

\def\w{\widetilde}
\def\b{\mathbb}
\def\c{\check}
\def\bE{\mathbb E}
\newenvironment{keywords}{{\bf Key words: }}{}

\begin{document}

\title{Backward Stackelberg Differential Game with Constraints: a Mixed Terminal-Perturbation and Linear-Quadratic Approach}

\author{Xinwei Feng$\;^{a}$, Ying Hu$\;^{b}$, Jianhui Huang$\;^{c}$\bigskip\\{\small ~$^{a}$Zhongtai Securities Institute for Financial Studies, Shandong University, Jinan, Shandong
250100, China}\\{\small ~$^{b}$Univ Rennes, CNRS, IRMAR-UMR 6625, F-35000 Rennes, France}\\{\small $^{c}$Department of Applied Mathematics, The Hong Kong Polytechnic University, Hong Kong, China}}
\maketitle

\begin{abstract}
We discuss an open-loop backward Stackelberg differential game involving single leader and single follower. Unlike most Stackelberg game literature, the state to be controlled is characterized by a \emph{backward} stochastic differential equation (BSDE) for which the terminal- instead initial-condition is specified as a priori; the decisions of leader consist of a static \emph{terminal-perturbation} and a dynamic \emph{linear-quadratic} control. In addition, the terminal control is subject to (convex-closed) \emph{pointwise} and (affine) \emph{expectation} constraints. Both constraints are arising from real applications such as mathematical finance. For information pattern: the leader announces both terminal and open-loop dynamic decisions at the initial time while takes account the best response of follower. Then, two interrelated optimization problems are sequentially solved by the follower (a backward linear-quadratic (BLQ) problem) and the leader (a mixed terminal-perturbation and backward-forward LQ (BFLQ) problem). Our open-loop Stackelberg equilibrium is represented by some coupled backward-forward stochastic differential equations (BFSDEs) with mixed initial-terminal conditions. Our BFSDEs also involve nonlinear projection operator (due to pointwise constraint) combining with a Karush-Kuhn-Tucker (KKT) system (due to expectation constraint) via Lagrange multiplier. The global solvability of such BFSDEs is also discussed in some nontrivial cases. Our results are applied to one financial example.
\end{abstract}

\begin{keywords}Backward stochastic differential equation, Karush-Kuhn-Tucker
(KKT) system, pointwise and affine constraints, Stackelberg game, backward linear-quadratic control, terminal perturbation.
\end{keywords}

%\footnotetext[1]{{\scriptsize Corresponding author:}}
\renewcommand{\thefootnote}{\arabic{footnote}}
\footnotetext[1]{{\scriptsize
       The work of Ying Hu is partially supported  by  Lebesgue Center of Mathematics ``Investissementsd'avenir''program-ANR-11-LABX-0020-01, by ANR CAESARS (Grant No. 15-CE05-0024) and by ANR MFG (Grant No. 16-CE40-0015-01).}}
\renewcommand{\thefootnote}{\fnsymbol{footnote}}
\footnotetext{\textit{{\scriptsize E-mail:}}
  {\scriptsize xwfeng@sdu.edu.cn (Xinwei\ Feng); ying.hu@univ-rennes1.fr (Ying\ Hu); majhuang@polyu.edu.hk (Jianhui\ Huang).}}

\section{Introduction}
Let $(\Omega, \mathcal F, \mathbb F, \mathbb P)$ be a complete filtered probability space on which a standard one-dimensional Brownian motion $W=\{W(t),0 \leq t <\infty\}$ is defined, where $\mathbb F=\{\mathcal F_t\}_{t\geq 0}$ is the natural filtration of $W$ augmented by all the $\mathbb P$-null sets in $\mathcal F$. Consider the following controlled linear backward stochastic differential equation (BSDE) on a finite time horizon $[0,T]$:
\begin{equation}\label{e state equation}
\begin{aligned}
dX(s)=\Big[A(s)X(s)+B_1(s)u_1(s)+B_2(s)u_2(s)+C(s)Z(s)\Big]ds+Z(s)dW(s),\quad
X(T)=\xi,
\end{aligned}
\end{equation}
where $A(\cdot),B_1(\cdot),B_2(\cdot),C(\cdot)$ are $\mathbb F$-progressively measurable processes defined on $\Omega\times[0,T]$ with proper dimensions. Unlike forward stochastic differential equation (SDE), solution of BSDE \eqref{e state equation} consists of a pair of adapted processes $(X(\cdot),Z(\cdot))\in \mathbb R^n\times\mathbb R^n$ where the second component $Z(\cdot)$ is necessary to ensure the adaptiveness of $X(\cdot)$ when propagating from terminal- backward to initial-time. In \eqref{e state equation}, $u_1(\cdot)$ and $u_2(\cdot)$ are dynamic decision processes employed by Player 1 (the leader, denoted by $\mathcal{A}_{L}$) and Player 2 (the follower, denoted by $\mathcal{A}_{F}$) in the game with values in $\mathbb R^{m_1}$ and $\mathbb R^{m_2}$ respectively. Moreover, unlike SDE, the terminal condition $\xi$ is specified in BSDE \eqref{e state equation} by the leader $\mathcal{A}_{L}$ at the initial time, and committed to be steered together with the follower by dynamic $u_2(\cdot)$. For some illustrating example, $\xi$ acts as some terminal hedging payoff on $T$, while $u_1(\cdot), u_2(\cdot)$ represent the possible dynamic portfolio selection or consumption process on $[0, T]$. The terminal $\xi$ to be steered may capture some appropriate approximation for quadratic deviation $K|X_{T}-\xi|^{2}$ with penalty index $K \longrightarrow +\infty$ (see \cite{LZ2001}).

Furthermore, let $\mathcal K$ be a nonempty closed convex subset in $\mathbb R^n$. Then, for a deterministic scalar $\beta$ and vector $\alpha \in \mathbb R^n$, we can define the following two constraints on admissible terminal payoff $\xi$:
 \begin{equation}\label{constraints}\left\{\begin{aligned}
\text{Pointwise constraint: }&\mathcal U_{\mathcal K}=L^2_{\mathcal F_T}(\Omega;\mathcal K);\\
\text{Affine expectation constraint: }&\mathcal U_{\alpha,\beta}=\Big\{\xi\Big| \xi \in L^2_{\mathcal F_T}(\Omega;\mathbb R^n),\langle\alpha, \mathbb E\xi\rangle\geq \beta \Big\}.
\end{aligned}\right.\end{equation}
Constraints of such kinds arise naturally in financial applications (e.g., see \cite{BJPZ2005} for expectation constraint, \cite{ET2015,HZ2005,LZ18} for pointwise one). In particular, the mean-variance portfolio selection with no-shorting yield such constraints both. Now, we define $\mathcal U(\mathcal K,\alpha,\beta)\triangleq\mathcal U_{\mathcal K}\bigcap\mathcal U_{\alpha,\beta}$ for the admissible terminal control set. Detailed discussion on feasibility of $\mathcal U(\mathcal K,\alpha,\beta)$ is deferred in Section  \ref{feasibility of constraint}. In addition, the following Hilbert spaces are introduced for dynamic admissible controls:\begin{equation}\nonumber
\mathcal U_i[0,T]\triangleq\Big\{u_i:[0,T]\times \Omega\rightarrow\mathbb R^{m_i}\Big| u_i(\cdot) \text{ is $\mathbb F$-progressively measurable, }\mathbb E\int^T_0|u_i(s)|^2ds<\infty\Big\},\quad i=1,2.
\end{equation}
Any element $(\xi,u_1(\cdot))\in\mathcal U(\mathcal K,\alpha,\beta)\times\mathcal U_1[0,T]$ is called an admissible control of $\mathcal{A}_{L}$, and any element $u_2(\cdot)\in\mathcal U_2[0,T]$ is called an admissible (dynamic) control of $\mathcal{A}_{F}$. Under some mild conditions on coefficients, for any $(\xi,u_1(\cdot),u_2(\cdot))\in \mathcal U(\mathcal K,\alpha,\beta)\times\mathcal U_1[0,T]\times \mathcal U_2[0,T]$, state equation \eqref{e state equation} admits a unique square-integrable adapted solution $(X(\cdot),Z(\cdot))\equiv( X(\cdot;\xi,u_1(\cdot),u_2(\cdot)),Z(\cdot;\xi,u_1(\cdot),u_2(\cdot))).$
To evaluate the performance of decisions $\xi,u_1(\cdot)$ and $u_2(\cdot)$, we introduce the following cost functionals:
\begin{equation}\label{cost functionals}\left\{\begin{aligned}
J_1(\xi,u_1(\cdot),u_2(\cdot))
\triangleq&\frac{1}{2}\mathbb E\Big\{\int^T_0\Big[\langle Q_1(s)X(s),X(s)\rangle+\langle S_1(s)Z(s),Z(s)\rangle+\left\langle R^1_{11}(s) u_{1}(s),u_{1}(s)\right\rangle\Big]ds\\&\qquad+\langle G_1\xi,\xi\rangle+\langle H_1X(0),X(0)\rangle\Big\},\\
J_2(\xi,u_1(\cdot),u_2(\cdot))
\triangleq&\frac{1}{2}\mathbb E\Big\{\int^T_0\Big[\langle Q_2(s)X(s),X(s)\rangle+\langle S_2(s)Z(s),Z(s)\rangle+\left\langle R^2_{22}(s)u_{2}(s),u_{2}(s)\right\rangle\Big]ds\\&\qquad+\langle H_2X(0),X(0)\rangle\Big\},
\end{aligned}\right.\end{equation}
where $Q_1(\cdot),Q_2(\cdot),S_1(\cdot),S_2(\cdot),R^1_{11}(\cdot)$, and $R^2_{22}(\cdot)$ are all $\mathbb F$-progressively measurable symmetric matrix valued processes, defined on $\Omega\times[0,T]$, of proper dimensions, $G_1$ is $\mathcal F_T$-measurable symmetric matrix valued random variable of proper dimension and $H_1,H_2$ are deterministic symmetric matrices of proper dimensions.  For $i=1,2$, $J_i(\xi,u_1(\cdot),u_2(\cdot))$ is the cost functional for agent $i$.

Let us now explain the Stackelberg differential game in some mixed backward linear quadratic (BLQ) and terminal-perturbation pattern.

At initial time, leader $\mathcal{A}_{L}$ announces some terminal (random) target $\xi\in\mathcal U(\mathcal K,\alpha,\beta)$ (to be reachable at terminal time $T$) and his planned dynamic strategy $u_1(\cdot)\in\mathcal U_1[0,T]$ over entire horizon $[0,T]$. $\xi$ is treated in a hard-constraint case, or in a limiting soft-constraint case (see \cite{BO1999}) when the soft-penalty on quadratic deviation $K|X_{T}-\xi|^{2}$ is endowed with sufficiently large attenuation level $K>0$. In both cases, the state dynamics becomes \eqref{e state equation} (see \cite{LZ2001}). Actually, $\xi$ may be interpreted as specific requirement of contractual or regulatory nature to reflect some risky position concern at terminal time $T$. Then, given the knowledge of leader's strategy, the follower $\mathcal{A}_{F}$ determines his best response strategy $\bar u_2(\cdot)\in\mathcal U_2[0,T]$ over entire horizon to minimize $J_2(\xi,u_1(\cdot),u_2(\cdot))$. Noticing state $X$ is steered imperatively towards the predetermined random target $\xi$ at maturity $T$. Since the follower's optimal response depends on the leader's strategy, the leader can take it into account as a priori before announcing his committed strategy to minimize $J_1(\xi, u_1(\cdot),\bar u_2(\cdot))$ over $(\xi ,u_1(\cdot))\in\mathcal U(\mathcal K,\alpha,\beta)\times \mathcal U_1[0,T]$.

\textbf{A principal-agent framework}. The above procedure might fit into some \emph{principal-agent} problem (see \cite{CZ2012}) but in a \emph{backward} framework:
$\mathcal{A}_{L}$ is the principal (owner of given firm) who specifies, at initial contract concluding time, some terminal achievement target $\xi$ to be realized by the agent in contractual manner together with his decision process $u_1(\cdot)$. Noticing $u_1$ may be interpreted as his committed consumption/capital withdraw process, an outflow on state dynamics $X$ as firm's wealth process. Meanwhile, $\mathcal{A}_F$ acts as the agent (manager) who is stimulated to reach such target by utilizing his investment/management/wage process $u_2(\cdot)$. When setting contract, $\mathcal{A}_L$ may set some constraints on $\xi$ with business concerns, while $\mathcal{A}_F$ is pushed to realize the terminal level $\xi$ once contract is executed due to some guarantee or breach clause. Thus, a BSDE state with $\xi$ follows through the contractual force.

Rigorously speaking, $\mathcal{A}_{F}$ aims to find a map $\bar \alpha:\mathcal U(\mathcal K,\alpha,\beta)\times\mathcal U_1[0,T]\rightarrow\mathcal U_2[0,T]$ and $\mathcal{A}_{L}$ aims to find a control $(\xi,u_1(\cdot))\in\mathcal U(\mathcal K,\alpha,\beta)\times\mathcal U_1[0,T]$ such that
\begin{equation*}
\left\{\begin{aligned}
&J_2(\xi,u_1(\cdot), \bar\alpha[\xi,u_1(\cdot)](\cdot))=\min_{u_2(\cdot)\in\mathcal U_2[0,T] }J_2(\xi,u_1(\cdot),u_2(\cdot)),\quad\forall (\xi,u_1(\cdot))\in\mathcal U(\mathcal K,\alpha,\beta)\times\mathcal U_1[0,T],\\
&J_1(\bar \xi,\bar u_1(\cdot),\bar \alpha[\bar \xi,\bar u_1](\cdot))=\min\limits_{\xi\in\mathcal U(\mathcal K,\alpha,\beta),u_1\in\mathcal U_1[0,T]}J_1(\xi,u_1(\cdot),\bar\alpha[\xi,u_1(\cdot)](\cdot)).
\end{aligned}\right.
\end{equation*}If the above pair $(\bar \xi,\bar u_1(\cdot),\bar \alpha(\bar \xi,\bar u_1(\cdot)))$ exists, we refer to it as an \emph{open-loop} Stackelberg equilibrium.

The setup in (1)-(3) above is especially motivated by optimal trading and quadratic hedging problem in financial mathematics when combining with terminal payoff subject to pointwise and integral constraints (see example in Section \ref{motivation}). Accordingly, the main novelties of our contribution are triple: (i) introduction of a new class of backward Stackelberg differential games with (pointwise and expectation affine) constraints and a mixed combination of terminal-perturbation and linear quadratic (LQ) control (both in backward sense); (ii) the characterization of open-loop Stackelberg equilibrium via new class of backward-forward stochastic differential equations (BFSDEs) with Karush-Kuhn-Tucker
(KKT) qualification condition; (iii) global solvability for above BFSDEs and some related Riccati equations.

To highlight above novelties, it is helpful to have some literature review comparing to some relevant existing works, especially to BLQ control, (forward) Stackelberg differential games, and various control problems with constraints imposed.

{\textbf{LQ control and game of backward state dynamics.}}
Nonlinear BSDE was initially introduced in \cite{PP1990} and is a well-formulated stochastic system hence it has been found various applications, for example, on stochastic recursive utility in economics by \cite{DE1992}. Interested readers may refer \cite{EPQ1997} for more BSDE applications in financial mathematics. Moreover, the relationship between BSDE and forward LQ optimal control is studied in \cite{KZ2000}. Based on it, \cite{LZ2001} discussed a BLQ optimal control problem motivated by quadratic hedging. \cite{LSX17} studied the BLQ optimal control problem with mean-field type. \cite{HWX2009} studied BLQ optimal control with partial information and gave some applications in pension fund optimization problems. Furthermore, some recent literature on games of BSDE can be found in \cite{WY2012,HWW2016}.

{\textbf{Stackelberg game.}}
The Stackelberg game (also termed as leader-follower game) was first introduced by \cite{S1934}. It differs from Nash game in its decision hierarchy of involved agents. Stackelberg games have been extensively explored from various settings. We list few works more relevant to ours: for deterministic Stackelberg game, see \cite{BO1999,L2010}, etc. For stochastic cases, \cite{BB1981} studied LQ Stackelberg differential game, but the state and control variables do not enter the diffusion coefficient. \cite{Y02} studied a more general Stackelberg game with random coefficients, control enters diffusion terms and control weight may be indefinite. \cite{BCS2015} investigated Stackelberg differential game in various different information structures, whereas the diffusion coefficient does not contain the control variables. \cite{OSU2013} studied stochastic Stackelberg differential game with time-delayed information. Notice that all above Stackelberg game works are framed in \emph{forward }sense with underlying state as a forward SDE that differs substantially from our backward one here.

\textbf{Constrained control and game.}
Naturally, control or game problems are always subject to possible constraints during its decision making. Such constraints may be posed on underlying state indirectly or decision input directly, or both in some mixed sense. From another viewpoint, these constraints may be structured as soft- or hard-constraint. In soft-constraint, a penalization depending on the deviation from constraints should be implemented in cost functional with some attenuation parameter indicating the softness. Hard-constraint might be viewed as limiting case of soft-constraint with \emph{attenuation index} tends to infinity. Thus, hard-constraint should be strictly followed in decision process to avoid any cost blow-up.
There exist considerable works on constrained stochastic control or games and we name a few more relevant. For example, \cite{HZ2005} studied stochastic LQ control constrained in general convex-closed cone, and some extended Riccati method is proposed; \cite{CZ2006} extends \cite{HZ2005} to infinite time horizon case. \cite{HZ2005,CZ2006} are both structured as hard constraint and include no-shorting of mean-variance problem as their special case. Moreover, \cite{LZL2002} studied LQ control problems with general input
constraint and its applications in financial portfolio selection with no-shorting
constraints. Some linear constraints are also treated therein. \cite{LZ1999,LM1997} studied various classes of integral affine and quadratic constraints.

\textbf{Terminal-perturbation with constraints.}
There arise various scenarios from mathematical finance with constraints on terminal payoffs that are static, e.g., the Markowitz mean-variance model poses some expectation constraint on terminal return. Thereby, it can convert to a family of indefinite stochastic LQ optimal controls with terminal constraints (\cite{ZL2000,LZL2002}). \cite{BJPZ2005} first employed backward approach to solve mean-variance problem by Lagrange method and obtained the optimal replicating portfolio strategy by solving some BSDE. To deal with state constraints of dynamic optimization problem, \cite{EPQ2001} (see also \cite{Q1993}) introduced the backward perturbation method and terminal variable of BSDE is regarded as some ``control variable". The terminal-perturbation method is well studied in financial mathematics and stochastic control (see e.g. \cite{JP2008,JZ2006,JZ2010}).

Compared with the above literature reviewed, main contributions of the present paper maybe summarized along the following lines:
\begin{itemize}
\item We introduce a new class of backward stochastic Stackelberg differential games featured by a mixed terminal-perturbation and BLQ control pattern. Other technical features include: backward-forward state system, random coefficients and Riccati equations, indefinite control weights.
\item Terminal-perturbation is subject to two (pointwise and affine expectation) constraints, some duality approach is invoked to tackle such constraints.
\item The open-loop Stackelberg equilibrium is represented by a coupled BFSDEs with mixed initial-terminal conditions, projection operator and constraint qualification conditions. To our knowledge, it is the first time to derive such \emph{constrained forward-backward systems}. Related global wellposedness is also studied in some special but nontrivial cases.
\end{itemize}

The rest of the paper is organized as follows.  In Section 2, we give some preliminaries and formulate the Stackelberg game in backward sense. The BLQ problem for follower is studied in Section 3, the mixed terminal-perturbation/backward-forward linear-quadratic (BFLQ) problem for leader is discussed in Section 4. In particular, Stackelberg equilibrium strategy is represented by some coupled BFSDEs with mixed initial-terminal conditions and constrained Karush-Kuhn-Tucker
(KKT) system. The global solvability of such BFSDEs is further discussed in Sections 5 in nontrivial cases. As the application, one example is discussed in Section 6.

\section{Preliminary and BLQ Stackelberg game formulation}\label{BLQ Stackelberg game}
The following notations will be used throughout the paper. Let $\mathbb{R}^n$ denote the $n$-dimensional Euclidean space with standard Euclidean norm $|\cdot|$ and standard Euclidean inner product $\langle\cdot,\cdot\rangle$. The transpose of a vector (or matrix) $x$ is denoted by $x^{\top}$. $\textrm{Tr}(A)$ denotes the trace of a square matrix $A$. Let $\mathbb{R}^{n\times m}$ be the Hilbert space consisting of all ($n\times m$)-matrices with the inner product $\langle A,B\rangle\triangleq\textrm{Tr}(AB^{\top})$ and the norm $||A||\triangleq\langle A,A\rangle^{\frac{1}{2}}$. Denote the set of symmetric $n\times n$ matrices with real elements by $\mathbb S^n$ and $n\times n$  identity matrices by $I_n$. If $M\in \mathbb S^n$ is positive (semi-)definite, we write $M>$ ($\geq$) $0$. If there exists a constant $\delta>0$ such that $M\geq\delta I$, we write $M\gg0$. Let $\mathbb S^n_+$ be the space of all positive semi-definite matrices of $\mathbb S^n$ and $\hat{\mathbb S}^n_+$ be the space of all positive definite matrices of $\mathbb S^n$.

Consider a finite time horizon $[0,T]$ for a fixed $T>0$. Let $H$ be a given Hilbert space. The set of $H$-valued continuous functions is denoted by $C([0,T];H)$. If $N(\cdot)\in C([0,T];\mathbb S^n)$ and $N(t)>$ ($\geq$) $0$ for every $t\in[0,T]$, we say that $N(\cdot)$ is positive (semi-)definite, which is denoted by $N(\cdot)>$ ($\geq$) $0$. For any $t\in[0,T)$ and Euclidean space $\mathbb H$, let(for the deterministic process, the subscripts $\mathcal F_t$ or $\mathbb{F}$ will be omitted)
\begin{equation}\nonumber
\begin{aligned}
&L^2_{\mathcal F_t}(\Omega;\mathbb H)=\{\xi:\Omega\rightarrow\mathbb H| \xi \text{ is  $\mathcal F_t$-measurable, } \mathbb{E}|\xi|^{2}<\infty\},\\
&L^\infty_{\mathcal F_t}(\Omega;\mathbb H)=\{\xi:\Omega\rightarrow\mathbb H| \xi \text{ is  $\mathcal F_t$-measurable, } \mbox{esssup}_{\omega\in\Omega}|\xi(\omega)|<\infty\},\\
&L^2_{\mathbb{F}}(0,T;\mathbb H)=\{\phi:[0,T]\times\Omega\rightarrow\mathbb H\Big|\text{$\phi$ is $\mathbb F$-progressively measurable, $\mathbb E \int^T_0|\phi(s)|^2ds<\infty$}\},\\
&L^\infty_{\mathbb{F}}(0,T;\mathbb H)=\{\phi:[0,T]\times\Omega\rightarrow\mathbb H|\text{$\phi$ is $\mathbb F$-progressively measurable, $\mbox{esssup}_{s\in[0,T]}\mbox{esssup}_{\omega\in\Omega}|\phi(s)|<\infty$}\},\\
&L^2_{\mathbb{F}}(\Omega;C([0,T];\mathbb H))=\{\phi:[0,T]\times\Omega\rightarrow\mathbb H|\text{$\phi$ is $\mathbb F$-adapted, continuous,  $\mathbb E [\sup\limits_{s\in[0,T]}|\phi(s)|^2]<\infty$}\}.\\
%&L^2_{\mathbb{F}}(\Omega;L^1(t,T;\mathbb H))=\Big\{\phi:[t,T]\times\Omega\rightarrow\mathbb H\Big|\text{$\phi$ is $\mathbb F$-progressively measurable, $\mathbb E \Big(\int^T_0|\phi(s)|ds\Big)^2<\infty$}\Big\}.
\end{aligned}
\end{equation}
 Recall the sets $\mathcal U_i[0,T]=L^2_{\mathbb{F}}(0,T;\mathbb R^{m_i})$. For notational simplicity, let $m = m_1 + m_2$ and denote
\begin{equation}\nonumber
B(\cdot) = (B_1(\cdot), B_2(\cdot)),\quad R_1(\cdot)=\left(\begin{matrix}R^1_{11}(\cdot)&0\\0 & 0\end{matrix}\right),\quad R_2(\cdot)=\left(\begin{matrix}0&0\\0 & R^2_{22}(\cdot)\end{matrix}\right).
\end{equation}
Naturally, we identify $u(\cdot)=(u_1(\cdot)^\top,u_2(\cdot)^\top)^\top\in\mathcal U[0, T ] = \mathcal U_1[0, T ] \times \mathcal U_2[0, T ]$. With such notations, the state equation \eqref{e state equation} becomes
\begin{equation}\label{e state equation 1}
\begin{aligned}
dX(s)=\Big[A(s)X(s)+B(s)u(s)+C(s)Z(s)\Big]ds+Z(s)dW(s),\quad
X(T)=\xi,
\end{aligned}
\end{equation}
where the terminal condition $\xi$ is a control variable with the constraints \eqref{constraints}.
The cost functionals become
\begin{equation}\nonumber\left\{\begin{aligned}
&J_1(\xi,u(\cdot))=\frac{1}{2}\mathbb E\Big\{\int^T_0\Big[\langle Q_1(s)X(s),X(s)\rangle+\langle S_1(s)Z(s),Z(s)\rangle+\left\langle R_1(s)u(s),u(s)\right\rangle\Big]ds\\
&\quad+\langle G_1\xi,\xi\rangle+\langle H_1X(0),X(0)\rangle\Big\},\\
&J_2(\xi,u(\cdot))=\frac{1}{2}\mathbb E\Big\{\int^T_0\Big[\langle Q_2(s)X(s),X(s)\rangle+\langle S_2(s)Z(s),Z(s)\rangle+\left\langle R_2(s)u(s),u(s)\right\rangle\Big]ds+\langle H_2X(0),X(0)\rangle\Big\}.
\end{aligned}\right.\end{equation}
Let us introduce the following assumptions, which will be used later.
\begin{description}
	\item[(H1)] The coefficients of the state equation satisfy the following:
\begin{equation}\nonumber
\begin{aligned}
&A(\cdot)\in L^\infty_{\mathbb F}(0,T;\mathbb R^{n\times n}),\quad B(\cdot)\in L^\infty_{\mathbb F}(0,T;\mathbb R^{n\times m}),\quad C(\cdot)\in L^\infty_{\mathbb F}(0,T;\mathbb R^{n\times n}).
\end{aligned}
\end{equation}
	\item[(H2)] The weighting coefficients of cost functional satisfy the following:
\begin{equation}\nonumber
\begin{aligned}
&G_1\in L^\infty_{\mathcal F_T}(\Omega; \mathbb S^n), H_1,H_2\in \mathbb S^n,  Q_1(\cdot),Q_2(\cdot),S_1(\cdot),S_2(\cdot)\in L^\infty_{\mathbb F}(0,T;\mathbb S^{n}), R_1(\cdot),R_2(\cdot)\in L^\infty_{\mathbb F}(0,T;\mathbb S^{m}).
\end{aligned}
\end{equation}
\end{description}Under (H1),  by \cite[Theorem 3.1]{PP1990}, for any $\xi\in\mathcal U(\mathcal K,\alpha,\beta)$ and $u(\cdot)\in\mathcal U[0,T]$, \eqref{e state equation 1} admits a unique strong solution $(X(\cdot),Z(\cdot))\in L^2_{\mathbb F}(\Omega;C([0,T];\mathbb R^n))\times L^2_{\mathbb F}(0,T;\mathbb R^n).$
Moreover, the following estimation holds:
\begin{equation}\label{estimation of BSDE}
\mathbb E\Big[\sup\limits_{s\in[0,T]}|X(s)|^2+\int^T_0|Z(s)|^2ds\Big]\leq L\mathbb E\Big[|\xi|^2+\int^T_0|u(s)|^2ds\Big],
\end{equation}
where $L>0$ is a constant which depends on the coefficients of \eqref{e state equation 1}. Therefore, under (H1)-(H2), the functionals $J_i(\xi,u(\cdot))=J_i(\xi,u_1(\cdot),u_2(\cdot))$ are well-defined for all $\xi\in\mathcal U(\mathcal K,\alpha,\beta)$ and $u_i(\cdot)\in\mathcal U_i[0,T]$, $i=1,2$.
If the coefficients in \eqref{e state equation 1} are deterministic, by \cite[Proposition 2.1]{SY2014}, \eqref{e state equation 1} admits a unique strong solution under the following relaxed assumption:
\begin{description}
	\item[(H1$^\prime$)] The coefficients of the state equation satisfy the following:
\begin{equation}\nonumber
\begin{aligned}
&A(\cdot)\in L^1(0,T;\mathbb R^{n\times n}),\quad B(\cdot)\in L^\infty(0,T;\mathbb R^{n\times m}),\quad C(\cdot)\in L^2(0,T;\mathbb R^{n\times n}).
\end{aligned}
\end{equation}
\end{description}
Moreover, \eqref{estimation of BSDE} still holds.
Hereafter, time variable $s$ will often be suppressed to simplify notations. We briefly state the procedure of finding an open-loop Stackelberg equilibrium: first, for any given $(\xi, u_1(\cdot))$,  $\mathcal A_F$ should solve a BLQ control problem with $\bar \alpha(\xi, u_1(\cdot))$ as the best response; second, given best response, $\mathcal A_L$ then solves a BFLQ control and terminal-perturbation with optimal $\bar \xi$ and $\bar u_1(\cdot).$ The Stackelberg equilibrium follows by $(\bar \xi, \bar u_1(\cdot),\bar \alpha(\bar \xi,\bar  u_1(\cdot)))$.

\section{Backward LQ problem for $\mathcal{A}_{F}$}\label{follower}
For given $(\xi,u_1(\cdot))\in\mathcal U(\mathcal K,\alpha,\beta)\times\mathcal U_1[0,T]$,  the follower $\mathcal A_F$ should solve the following
BLQ Problem: \\
\begin{equation*}\textbf{(BLQ):} \text{ Minimize}\quad J_2(\xi,u_1(\cdot), u_2(\cdot))\text{\qquad subject to}\quad\eqref{e state equation 1}, \quad u_2(\cdot)\in\mathcal U_2[0,T].
\end{equation*}
\begin{definition}
\emph{(a)} For given $(\xi,u_1(\cdot))\in\mathcal U(\mathcal K,\alpha,\beta)\times\mathcal U_1[0,T]$,
problem \emph{\textbf{(BLQ)}} is said to be finite if cost functional $J_2(\xi,u_1(\cdot), u_2(\cdot))$ is bounded from below, that is, $\inf_{u_2(\cdot)\in\mathcal U_2[0,T]} J_2(\xi,u_1(\cdot),u_2(\cdot))>-\infty$\emph{;}

\emph{(b)} Problem \emph{\textbf{(BLQ)}} is said to be \emph{(}uniquely\emph{)} solvable if there exists a \emph{(}unique\emph{)} $ u_2^{*}(\cdot) \in \mathcal U_2[0,T]$ such that $J_2(\xi,u_1(\cdot),u_2^{*}(\cdot))=\inf_{u_2(\cdot)\in\mathcal U_2[0,T]}J_2(\xi,u_1(\cdot),u_2(\cdot)).$ In this case, $u_2^{*}(\cdot)$ is called minimizer of \emph{\textbf{(BLQ)}}.
\end{definition}We now give a representation of cost functional for {(BLQ)} which helps us to study its solvability. Its proof is straightforward based on duality theory thus we omit details here.
\begin{proposition}\label{presentation of cost-follower}
Let \textbf{\emph{(H1)}}-\textbf{\emph{(H2)}} hold. There exist two bounded self-adjoint linear operators $M_2:\mathcal U_2[0,T]\rightarrow \mathcal U_2[0,T]$, $M_1: L_{\mathcal F_T}^2(\Omega;\mathbb R^n)\times\mathcal U_1[0,T]\rightarrow\mathcal U_2[0,T]$ and some $M_0\in\mathbb R$ depending on $(\xi,u_1(\cdot))$ such that
\begin{equation*}\label{quadratic form of cost-follower}
J_2(\xi,u_1(\cdot),u_2(\cdot))=\frac{1}{2}\Big[\mathbb E\langle M_2(u_2)(\cdot),u_2(\cdot)\rangle+2\mathbb E\langle M_1(\xi,u_1)(\cdot),u_2(\cdot)\rangle+M_0\Big],
\end{equation*}with
\begin{equation}\label{coefficient of quadratic functional-follower}\begin{aligned}
&M_2(u_2)(\cdot)=R_{22}^2(\cdot)u_2(\cdot)- B_2^\top(\cdot) Y_1(\cdot),\qquad M_1(\xi,u_1)(\cdot)=-B_2^\top(\cdot) Y_2(\cdot)-B_2^\top(\cdot) Y_3(\cdot),\\
&M_0=-\mathbb E\int_0^T\langle B_1^\top(s) Y_3(s),u_1(s)\rangle ds+\mathbb E\langle Y_2(T),\xi\rangle+2\mathbb E\langle Y_3(T),\xi\rangle,
\end{aligned}\end{equation}
where $Y_1, Y_2, Y_3$ satisfy the following backward-forward systems:\begin{equation}\label{eq:8}\left\{\begin{aligned}
&dY_1(s)=\Big[-A^\top(s) Y_1(s)+Q_2(s)X_1(s)\Big]ds+\Big[-C^\top(s) Y_1(s)+S_2(s)Z_1(s)\Big]dW(s),\\
&dX_1(s)=\Big[A(s)X_1(s)+B_2(s)u_2(s)+C(s)Z_1(s)\Big]ds+Z_1(s)dW(s),\\
&X_1(T)=0,\qquad Y_1(0)=H_2X_1(0),
\end{aligned}\right.\end{equation}
\begin{equation*}\left\{\begin{aligned}
&dY_2(s)=\Big[-A^\top(s) Y_2(s)+Q_2(s)X_2(s)\Big]ds+\Big[-C^\top (s)Y_2(s)+S_2(s)Z_2(s)\Big]dW(s),\\
&dX_2(s)=\Big[A(s)X_2(s)+C(s)Z_2(s)\Big]ds+Z_2(s)dW(s),\\
&X_2(T)=\xi,\qquad Y_2(0)=H_2X_2(0),
\end{aligned}\right.\end{equation*}
\begin{equation*}\left\{\begin{aligned}
&dY_3(s)=\Big[-A^\top(s) Y_3(s)+Q_2(s)X_3(s)\Big]ds+\Big[-C^\top(s) Y_3(s)+S_2(s)Z_3(s)\Big]dW(s),\\
&dX_3(s)=\Big[A(s)X_3(s)+B_1(s)u_1(s)+C(s)Z_3(s)\Big]ds+Z_3(s)dW(s),\\
&X_3(T)=0,\qquad Y_3(0)=H_2X_3(0).
\end{aligned}\right.\end{equation*}
\end{proposition}In the above, we use $\langle\cdot,\cdot\rangle$ to denote inner products in different Hilbert spaces, which can be identified from the context. Based on Proposition \ref{presentation of cost-follower}, we have the following result for the solvability of problem (BLQ), whose proof is similar to that of \cite[Theorem 6.2.2]{YZ1999}.
\begin{proposition}\label{solvability of follower}
Let \textbf{\emph{(H1)}}-\textbf{\emph{(H2)}} hold.
 \begin{description}
  \item[(a)] Problem \textbf{\emph{(BLQ)}} is finite only if \textbf{\emph{(BLQ)}} is convex \emph{(}i.e., $M_2\geq0$\emph{)};

  \item[(b)] Problem \textbf{\emph{(BLQ)}} is \emph{(}uniquely\emph{)} solvable if and only if \emph{(}iff\emph{)} \textbf{\emph{(BLQ)}} is convex \emph{(}$M_2\geq0$\emph{) } and the following stationary condition holds true: there exists a \emph{(}unique\emph{)} $\bar u_2(\cdot)\in\mathcal U_2[0,T]$ such that
\begin{equation}\label{stationary condition-1}
 M_2(\bar u_2)(\cdot)+M_1(\xi,u_1)(\cdot)=0.
\end{equation}
Moreover, \eqref{stationary condition-1} implies that $\mathcal R(M_1(\xi,u_1))\subset\mathcal R(M_2(\bar u_2)),$
where $\mathcal R (\mathcal{S})$ stands for the range of operator \emph{(}matrix\emph{)} $\mathcal{S}$.
   \item[(c)]  If \textbf{\emph{(BLQ)}} is uniformly convex \emph{(}i.e., $M_2\gg0$\emph{)}, then problem \textbf{\emph{(BLQ)}} admits a unique optimal control given by $$\bar u_2(\cdot)=-M_2^{-1}(M_1(\xi,u_1))(\cdot).$$
 \end{description}\end{proposition}
 (a)-(c) in Proposition \ref{solvability of follower} can be summarized by the following inclusion relation diagram:
 \begin{equation*}\begin{aligned}
  &\text{uniform convexity} \Longrightarrow \text{unique solvability} \Longrightarrow \text{solvability}(\Longleftrightarrow \text{convexity, stationary condition}) \\
  &\Longrightarrow \text{finiteness} \Longrightarrow \text{convexity}.
  \end{aligned}\end{equation*}
  Given representation \eqref{coefficient of quadratic functional-follower}, \eqref{stationary condition-1} takes the following form:
$$R_{22}^2(\cdot)\bar u_2(\cdot)- B_2^\top(\cdot) Y_1(\cdot)-B_2^\top(\cdot) Y_2(\cdot)-B_2^\top(\cdot) Y_3(\cdot)=0.$$
Therefore, if we define $\bar Y=Y_1+Y_2+Y_3, \bar X=X_1+X_2+X_3, \bar Z=Z_1+Z_2+Z_3$, we have the following solvability result in terms of BFSDEs.
 %which is a variant of Proposition \ref{solvability of follower} in which the conditions are given in terms of a new FBSDE.
\begin{theorem}\label{th1}
Under \textbf{\emph{(H1)}}-\textbf{\emph{(H2)}}, for any $u_2(\cdot)\in\mathcal U_2[0,T]$, suppose that
\begin{equation}\label{convexity-follower}
\mathbb E\langle M_2(u_2)(\cdot),u_2(\cdot)\rangle=\mathbb E\int_0^T\langle R_{22}^2(s)u_2(s)- B_2^\top(s) Y_1(s),u_2(s)\rangle ds\geq0,
\end{equation}where $(Y_1,X_1,Z_1)$ is the solution of \eqref{eq:8} with respect to $u_2(\cdot)$. Then problem \textbf{\emph{(BLQ)}} is \emph{(}uniquely\emph{)} solvable with an \emph{(}the\emph{)} optimal pair $(\bar X(\cdot),\bar Z(\cdot),\bar u_2(\cdot))$ iff there \emph{(}uniquely\emph{)} exists a 4-tuple $(\bar Y(\cdot),\bar X(\cdot),\bar Z(\cdot),\bar u_2(\cdot))$ satisfying BFSDEs
\begin{equation}\label{optimal system}
\left\{\begin{aligned}
&d\bar Y(s)=\Big[-A^\top(s) \bar Y(s)+Q_2(s) \bar X(s)\Big]ds+\Big[-C^\top(s) \bar Y(s)+S_2(s) \bar Z(s)\Big]dW(s),\\
&d\bar X(s)=\Big[A(s)\bar X(s)+B_1(s)u_1(s)+B_2(s)\bar u_2(s)+C(s)\bar  Z(s)\Big]ds+\bar Z(s)dW(s),\\
& \bar Y(0)=H_2\bar  X(0),\quad \bar  X(T)=\xi,
\end{aligned}\right.
\end{equation}such that
\begin{equation}\label{stationary condition-2}
R_{22}^2(s)\bar u_2(s)-B_2^\top(s) \bar Y(s)=0,\qquad s\in[0,T],\quad \mathbb P-a.s.
\end{equation}
\end{theorem}Let us give the following inverse assumption.
\begin{description}
	\item[(H3)] $R_{22}^2(\cdot)$ is invertible and $(R_{22}^2(\cdot))^{-1}\in L_{\mathbb F}^\infty(0,T;\mathbb R^{m_2})$.
\end{description}Clearly, under (H3), optimal control $\bar u_2(\cdot)$ can be further represented as\begin{equation}\label{e1F optimal control of the follower}
\bar u_2(s)=(R^2_{22}(s))^{-1}B_2(s)^\top\bar Y(s),
\end{equation}
and \eqref{optimal system}-\eqref{stationary condition-2} are equivalent to the following BFSDEs:
\begin{equation}\label{Hamiltonian system-follower}
\left\{\begin{aligned}
&d\bar Y(s)=\Big[-A^\top(s) \bar Y(s)+Q_2(s) \bar X(s)\Big]ds+\Big[-C^\top(s) \bar Y(s)+S_2(s) \bar Z(s)\Big]dW(s),\\
&d\bar X(s)=\Big[A(s)\bar X(s)+B_1(s)u_1(s)+B_2(s)(R^2_{22}(s))^{-1}B_2^\top(s)\bar Y(s)+C(s)\bar  Z(s)\Big]ds+\bar Z(s)dW(s),\\
& \bar Y(0)=H_2\bar  X(0),\quad \bar  X(T)=\xi.
\end{aligned}\right.
\end{equation}
BFSDEs \eqref{Hamiltonian system-follower} differs from classical forward-backward stochastic differential equations (FBSDEs) because forward state $\bar Y(\cdot)$ depends on backward state $\bar X(\cdot)$ via initial $\bar X(0)$ instead terminal $\bar X(T)$. Unlike Yong \cite{Y02}, the state \eqref{Hamiltonian system-follower} is not \emph{decoupled} thus its global solvability is not straightforward.
In Section \ref{solva-Hamil-follwer}, we will establish global solvability under some suitable conditions on the coefficients. Moreover, regarding the relation between \eqref{Hamiltonian system-follower} and \textbf{(BLQ)}, we have the following statement:
\begin{corollary}\label{coro1}
Under \textbf{\emph{(H1)}}-\textbf{\emph{(H3)}}, let \eqref{convexity-follower} hold. Then Problem \textbf{\emph{(BLQ)}} is \emph{(}pathwise uniquely\emph{)} solvable iff BFSDEs \eqref{Hamiltonian system-follower} admits a \emph{(}unique\emph{)} strong solution $(\bar Y(\cdot),\bar X(\cdot),\bar Z(\cdot))\in L^2_{\mathbb F}(\Omega;C([0,T];\mathbb R^n))\times L^2_{\mathbb F}(\Omega;C([0,T];\mathbb R^n))\times L^2_{\mathbb F}(0,T;\mathbb R^n)$.
\end{corollary}If uniformly convexity holds, i.e., there exists a constant $\gamma>0$ such that for any $u_2(\cdot)\in\mathcal U_2[0,T]$,
\begin{equation}\label{uniformly convexity-follower}
\mathbb E\langle M_2(u_2)(\cdot),u_2(\cdot)\rangle=\mathbb E\int_0^T\langle R_{22}^2(s)u_2(s)- B_2^\top(s) Y_1(s),u_2(s)\rangle ds\geq\gamma\mathbb E\int_0^T|u_2(s)|^2ds,
\end{equation}then (BLQ) is uniquely solvable. Therefore, it follows from Corollary \ref{coro1} that BFSDEs \eqref{Hamiltonian system-follower} admits a unique strong solution $(\bar Y(\cdot),\bar X(\cdot),\bar Z(\cdot))$.
Next, we will study the uniformly convex condition \eqref{uniformly convexity-follower} of (BLQ). First,
introduce the following auxiliary BLQ problem \textbf{(ABLQ)}:
\begin{equation*}\label{state equation of (ABLQ)}\left\{
\begin{aligned}
&\text{ Minimize }
\begin{aligned}
\mathcal J(u_2(\cdot))=\mathbb{E}
      \Big\{&\int_{0}^{T}\Big[\langle Q_2x(s),x(s)\rangle+\langle S_2z(s),z(s)\rangle+\langle R^2_{22} u_2(s),u_2(s)\rangle\Big] ds+\langle H_2 x(0),x(0)\rangle\Big\},
      \end{aligned}\\
&\text{ subject to }
dx(s)=\Big[A(s)x(s)+B_2(s)u_2(s)+C(s)z(s)\Big]ds+z(s)dW(s),\quad x(T)=0,\quad s\in[0,T].
\end{aligned}\right.
\end{equation*}Note that for {(ABLQ)}, its functional $\mathcal J(u_2(\cdot))=\mathbb E\langle M_2(u_2)(\cdot),u_2(\cdot)\rangle,$  which is the left hand side of \eqref{convexity-follower}. Therefore, convexity condition \eqref{convexity-follower} holds iff {(ABLQ)} is well-posed with a necessarily nonnegative minimal cost. Moreover, if there exists a constant $\gamma>0$ such that $\mathcal J(u_2(\cdot))>\gamma\mathbb E\int_0^T|u_2(s)|^2ds$ for any $u_2(\cdot)\in\mathcal U_2[0,T]$, the uniformly convexity condition \eqref{uniformly convexity-follower} holds. Now we introduce the following standard assumptions
\begin{equation*}\begin{aligned}
&\textbf{(SA-1):}\quad H_2\geq0,\quad Q_2(\cdot)\geq0,\quad S_2(\cdot)\geq0,\quad R^2_{22}(\cdot)\gg0.
\end{aligned}\end{equation*}
For any given nonsingular symmetric matrix $M$, we introduce the following Riccati equation (denoted by \textbf{(SRE-1)}):
\begin{equation*}\label{Riccati-follower}\left\{\begin{aligned}
&dP=-\Big[Q_2+PA+A^\top P-PB_2(R_{22}^2)^{-1}B_2^\top P-(PC +K)(P+S_2)^{-1}(C^\top P+K)\Big] ds+K dW(s),\\
&P(T)=M,\\
&P(s)+S_2(s)>0,\quad 0\leq s\leq T.
\end{aligned}\right.\end{equation*}
\begin{proposition}\label{convexity of follower}
Under \textbf{\emph{(H1)}}-\textbf{\emph{(H3)}}, if $R_{22}^2(\cdot)>0$ and Riccati equation \textbf{\emph{(SRE-1)}} has a solution $(P(\cdot),K(\cdot))\in L_{\mathbb F}^\infty(0,T;\mathbb S^n)\times L_{\mathbb F}^2(0,T;\mathbb S^n)$ such that
$ P(0)+H_2\geq0.$
%and
%$$ P^{-1}(\cdot)K(\cdot)(P(\cdot)+S_2(\cdot))^{-1}\in L_{\mathbb F}^\infty(0,T;\mathbb R^n).$$
Then for any $u_2(\cdot)\in\mathcal U_2[0,T]$,
$$\mathcal J(u_2(\cdot))\geq0,$$
%and $$\mathcal J(u_2(\cdot))=0\Longleftrightarrow u_2(\cdot)=0.$$Therefore, the strictly convexity condition \eqref{strictly convexity-follower} holds.
and in this case,  \textbf{\emph{(BLQ)}} is convex on $u_2(\cdot).$ Moreover, if there exists a constant $\delta>0$ and $R_{22}^2(\cdot)\geq\delta I$, then there exists a constant $\gamma>0$ such that $$\mathcal J(u_2(\cdot))\geq\delta\gamma\mathbb E\int_0^T|u_2(s)|^2ds,$$and in this case,  \textbf{\emph{(BLQ)}} is uniformly convex on $u_2(\cdot).$ In particular, under \textbf{\emph{(SA-1)}}, \textbf{\emph{(BLQ)}} is uniformly convex on $u_2(\cdot).$
%i.e., $ u_2(\cdot)\rightarrow J(\xi,u_1(\cdot),u_2(\cdot))$ is uniformly convex.
\end{proposition}
The proof of Proposition \ref{convexity of follower} is given in the Appendix, Section \ref{sec7.1}.

\section{Terminal-perturbation and BFLQ problem of $\mathcal A_L$}\label{leader}

Considering \eqref{e1F optimal control of the follower}, the corresponding state process for $\mathcal A_L$ becomes the following BFSDEs:
\begin{equation}\label{state equation-leader}
\left\{\begin{aligned}
&d Y(s)=\Big[-A^\top(s)  Y(s)+Q_2(s)  X(s)\Big]ds+\Big[-C^\top(s)  Y(s)+S_2(s)  Z(s)\Big]dW(s),\\
&d X(s)=\Big[A(s) X(s)+B_1(s)u_1(s)+B_2(s)(R^2_{22}(s))^{-1}B_2^\top(s) Y(s)+C(s)  Z(s)\Big]ds+ Z(s)dW(s),\\
&  Y(0)=H_2  X(0),\quad   X(T)=\xi,
\end{aligned}\right.
\end{equation}which is controlled by $\xi$ (terminal-perturbation) and $u_1(\cdot)$ with the following cost functional
\begin{equation}\nonumber
\begin{aligned}
&J_1(\xi,u_1(\cdot))=\frac{1}{2}\mathbb E\Big\{\int^T_0\Big[\langle Q_1(s) X(s), X(s)\rangle+\langle S_1(s) Z(s), Z(s)\rangle+\left\langle R^1_{11}(s)u_{1}(s),u_{1}(s)\right\rangle\Big]ds\\
&+\langle G_1\xi,\xi\rangle+\langle H_1X(0),X(0)\rangle\Big\}.
\end{aligned}\end{equation}
The existence and uniqueness of BFSDEs \eqref{state equation-leader} is established in Corollary \ref{coro1}. Now, $\mathcal{A}_{L}$ should solve the following \emph{mixed} \emph{terminal-perturbation} and BFLQ problem for above system:
\begin{equation*}\label{leader-BFLQ}\textbf{(P)}:
\text{ Minimize } \qquad J_1(\xi,u_1(\cdot))
\text{\qquad subject to } \qquad \eqref{state equation-leader},\ (\xi,u_1(\cdot))\in\mathcal U(\mathcal K,\alpha,\beta)\times\mathcal U_1[0,T].
\end{equation*}We denote above problem as (P) for \emph{primal} problem, to be compared with the \emph{dual} problem that will be introduced later. Now, it is necessary to set some definitions pertinent to its solvability.
\begin{definition} \emph{(a)} Problem \emph{\textbf{(P)}} is said to be finite if cost functional $J_{1}$ is bounded from below, that is, $\mu_{p}\triangleq\inf_{(\xi,u_1(\cdot))\in\mathcal U(\mathcal K,\alpha,\beta)\times\mathcal U_1[0,T]} J_1(\xi,u_1(\cdot))>-\infty.$ $\mu_{p}$ is called the value of \emph{(}primal\emph{)} problem \emph{\textbf{(P)}}\emph{;}

\emph{(b)} Problem \emph{\textbf{(P)}} is said to be \emph{(}uniquely\emph{)} solvable if there exists a \emph{(}unique\emph{)} $(\xi^{*}, u_1^{*}(\cdot)) \in \mathcal U(\mathcal K,\alpha,\beta)\times\mathcal U_1[0,T]$ such that $\mu_{p}=J_1(\xi^{*},u_1^{*}(\cdot)).$ In this case, $(\xi^{*}, u_1^{*}(\cdot))$ is called minimizer of problem \emph{\textbf{(P)}}.\end{definition}For solvability, a related definition is the convexity. Considering $\mathcal{U}(\mathcal K,\alpha,\beta)$ is closed-convex, we formulate the following trivial definition.
\begin{definition}Problem \emph{\textbf{(P)}} is said to be convex if its cost functional $J_{1}$ is convex on $(\xi,u_1(\cdot)).$ Its strictly- and uniformly-convexity can be defined similarly.\end{definition}

\subsection{Convexity and solvability of primal problem}\label{Convexity and solvability of primal problem}
For primal problem (P), the following representation of $J_{1}$ may help to characterize its solvability and convexity in a direct manner.
\begin{proposition}\label{presentation of cost-leader}
Let \textbf{\emph{(H1)}}-\textbf{\emph{(H3)}} hold. There exist two bounded self-adjoint linear operators $\mathcal M_2:\mathcal U_1[0,T]\rightarrow \mathcal U_1[0,T]$, $\mathcal M_1:  L_{\mathcal F_T}^2(\Omega;\mathbb R^n) \rightarrow L_{\mathcal F_T}^2(\Omega;\mathbb R^n)$ and a bounded linear operator $\mathcal M_0: L_{\mathcal F_T}^2(\Omega;\mathbb R^n) \rightarrow \mathcal U_1[0,T]$ such that
\begin{equation}\label{quadratic form of cost-leader}
J_1(\xi,u_1(\cdot))=\frac{1}{2}\mathbb E\Big[\langle\mathcal M_2(u_1)(\cdot),u_1(\cdot)\rangle+\langle\mathcal M_1(\xi),\xi\rangle+2\langle\mathcal M_0(\xi)(\cdot),u_1(\cdot)\rangle\Big],
\end{equation}with\begin{equation*}\begin{aligned}
&\mathcal M_2(u_1)(\cdot)=R_{11}^1(\cdot)u_1(\cdot)-B_1^\top(\cdot)g_1(u_1)(\cdot),\ \mathcal M_1(\xi)=G_1\xi+g_2(T),\ \mathcal M_0(\xi)(\cdot)=-B_1^\top(\cdot)g_2(\cdot),
\end{aligned}\end{equation*}where $g_1, g_2$ depending on $u_{1}$ and $\xi$ respectively, are defined through the following BFSDEs
\begin{equation}\label{adjoint-1}\left\{\begin{aligned}
&d Y_1=\Big[-A^\top Y_1+Q_2 X_1\Big]ds+\Big[-C^\top Y_1+S_2 Z_1\Big]dW(s),\\
&d X_1=\Big[AX_1+B_1u_1+B_2(R^2_{22})^{-1}B_2^\top Y_1+C Z_1\Big]ds+ Z_1dW(s),\\
&dg_1=-\Big[ A^\top  g_1-Q_2h_1-Q_1 X_1\Big]ds-\Big[C^\top g_1-S_2 q_1-S_1 Z_1\Big]dW(s),\\
&dh_1=\Big[ Ah_1+B_2(R^2_{22})^{-1}B_2^\top g_1+ Cq_1\Big]ds+q_1dW(s),\\
& Y_1(0)=H_2 X_1(0),\quad  X_1(T)=0,\quad g_1(0)=H_1 X_1(0)+ H_2 h_1(0),\quad h_1(T)=0,
\end{aligned}\right.
\end{equation}
\begin{equation}\label{adjoint-2}\left\{\begin{aligned}
&d Y_2=\Big[-A^\top Y_2+Q_2 X_2\Big]ds+\Big[-C^\top Y_2+S_2Z_2\Big]dW(s),\\
&d X_2=\Big[A X_2+B_2(R^2_{22})^{-1}B_2^\top Y_2+C Z_2\Big]ds+ Z_2dW(s),\\
&dg_2=-\Big[ A^\top  g_2-Q_2h_2-Q_1 X_2\Big]ds-\Big[C^\top g_2-S_2 q_2-S_1 Z_2\Big]dW(s),\\
&dh_2=\Big[ Ah_2+B_2(R^2_{22})^{-1}B_2^\top g_2+ Cq_2\Big]ds+q_2dW(s),\\
& Y_2(0)=H_2 X_2(0),\quad  X_2(T)=\xi,\quad g_2(0)=H_1 X_2(0)+ H_2 h_2(0),\quad h_2(T)=0.
\end{aligned}\right.
\end{equation}
\end{proposition}The proof of Proposition 4.1 follows from duality of BFSDEs and readers may refer \cite{YZ1999} for similar representation. It follows from \eqref{quadratic form of cost-leader} that $J_{1}$ is quadratic functional on $(\xi,u_1(\cdot))$ and we have the following result concerning its convexity on constrained admissible set
$\mathcal U(\mathcal K,\alpha,\beta)\times\mathcal U_1[0,T].$
\begin{proposition}\label{prop-convexity of cost-leader}
Let \textbf{\emph{(H1)}}-\textbf{\emph{(H3)}} hold. Then \emph{\textbf{(P)}} is convex iff %the following convexity condition holds:
\begin{equation}\label{convexity of cost-leader}
\begin{aligned}\text{block operator:} \quad \mathcal{M}\triangleq\left[
\begin{array}{c|c}
  \mathcal{M}_{1} &  \mathcal{M}_{0}^* \\
  \hline
   \mathcal{M}_{0} &  \mathcal{M}_{2}
\end{array}
\right]\geq 0 \Longleftrightarrow J_1(\zeta, v(\cdot))\geq0,\quad
\forall (\zeta, v(\cdot))\in  {\mathcal U}_{\widetilde{\mathcal K}}
\times\mathcal U_1[0,T],\\
\end{aligned}
\end{equation}
where $\mathcal M_0^{*}(u_1)=g_1(T): \mathcal U_1(0,T)\longmapsto L^{2}_{\mathcal{F}_T}(\Omega;\mathbb R^n)$ is the adjoint operator of $\mathcal M_0(\xi)$ and $\widetilde{\mathcal K} \triangleq \mathcal K-\mathcal K=\{x-y: x \in \mathcal K, y \in \mathcal K\}$ is the algebra difference of $\mathcal{K}$ \emph{(}it is also convex but not necessary to be closed unless $\mathcal K$ is compact\emph{)}. Moreover, $\emph{\textbf{(P}})$ is uniformly convex iff for some $\delta>0$,
%the following uniformly convexity condition holds:
\begin{equation}\label{uniformly convexity of cost-leader}
\begin{aligned}&J_1(\zeta,v(\cdot))\geq\delta\Big[\mathbb E|\zeta|^2+\mathbb E\int_0^T|v(s)|^2ds\Big], \quad \forall (\zeta, v(\cdot))\in  {\mathcal U}_{\widetilde{\mathcal K}}
\times\mathcal U_1[0,T].
\end{aligned}
\end{equation}
\end{proposition}\begin{proof}For $\forall (\xi, u_{1}), (\xi', u_{1}') \in \mathcal U(\mathcal K,\alpha,\beta)\times\mathcal U_1[0,T],$ denote $\xi^{\lambda}=\lambda \xi+(1-\lambda)\xi', u_{1}^{\lambda}=\lambda u_{1}+(1-\lambda)u_{1}'$ for $\lambda \in [0,1],$ then $\zeta=\xi-\xi' \in \mathcal U_{\widetilde{\mathcal K}}, v=u_1-u_1' \in \mathcal U_1[0,T].$ Then, by \eqref{quadratic form of cost-leader}, $J_1$ should be convex iff
\begin{equation}\nonumber
\begin{aligned}0 \geq& J_1(\xi^{\lambda}, u_{1}^{\lambda})-\lambda J_{1}(\xi, u_{1})-(1-\lambda)J_{1}(\xi', u_{1}'),\\
=&\frac{1}{2}\lambda(\lambda-1)\Big[\langle\mathcal M_2(v)(\cdot),v(\cdot)\rangle+\langle\mathcal M_1(\zeta),(\zeta)\rangle+2\langle\mathcal M_0(\zeta)(\cdot),v(\cdot)\rangle\Big]\\=&\lambda(\lambda-1)J_1(\zeta,v(\cdot)).\end{aligned}
\end{equation}Hence the result (19) follows. Similar arguments apply to uniformly convexity leading to (20).\end{proof}
\begin{remark}
Similar to Schur lemma, we have $J_{1}(\cdot, \cdot)$ is strictly convex iff $$\mathcal{M} > 0 \Longleftrightarrow \mathcal{M}_{2} > 0, \quad\mathcal M_1-\mathcal M_0^*\mathcal M_{2}^{-1}\mathcal M_{0} > 0 \Longleftrightarrow \mathcal{M}_{1} > 0, \quad \mathcal M_2-\mathcal M_0\mathcal M_{1}^{-1}\mathcal M_{0}^* > 0 .$$
\end{remark}It follows that the convexity on $(\xi, u_1(\cdot))$ jointly is stronger than convexity on $\xi$ and $u_1(\cdot)$ marginally.
As a consequence, we have the following result when $\mathcal{K}$ is further conic:\begin{corollary}\label{prop-convexity of cost-leader}
Let \textbf{\emph{(H1)}}-\textbf{\emph{(H3)}} hold and $\mathcal{K}$ is closed-convex cone. Then, $\emph{\textbf{(P)}}$ is convex iff
\begin{equation}\nonumber
\begin{aligned}J_1(\zeta, v_{1}(\cdot))\geq0,\qquad
\forall (\zeta, v_{1}(\cdot))\in  {\mathcal U}_{\text{aff}(\mathcal K)}
\times\mathcal U_1[0,T],\\
\end{aligned}
\end{equation}where $\text{aff}(\mathcal K)=\mathcal K-\mathcal K$ is the affine subspace generated by $\mathcal K.$  \end{corollary}
Noticing a \emph{closed} cone always contains $0$ thus $\widetilde{\mathcal K}=\text{aff}(\mathcal K)$ that may be a proper subset of full space $\mathbb R^n.$

In standard LQ control literature, when the admissible controls are from full linear space, then \emph{finiteness} of problem implies its \emph{convexity}. Alternatively, when admissible controls are only from some closed-convex proper subset, we have the following different results.

\begin{lemma}\label{prop-finite of cost-leader}
Suppose\textbf{\emph{(H1)}}-\textbf{\emph{(H3)}} hold and $\mathcal{K}$ is a closed-convex set containing origin $0$. Then, problem \emph{\textbf{(P)}} is finite only if $J_1$ is nonnegative functional on $\mathcal \mathcal U_{C_{\infty}(\mathcal K)}\times\mathcal U_1[0,T]$ where $C_{\infty}(\mathcal K)$ is the asymptotic \emph{(}recession\emph{)} cone of $\mathcal{K}$.
\end{lemma}
\begin{proof}
First, recall that $C_{\infty}(\mathcal{K}) \subseteq \mathcal{K}$ if origin $0 \in \mathcal{K}$ hence $\mathcal U_{C_{\infty}(\mathcal K)} \subseteq \mathcal U_{\mathcal K}.$ If the statement is not true, then $J_{1}$ is finite but there exists a pair $(\xi^{0},u_1^{0})\in \mathcal U_{C_{\infty}(\mathcal K)}\times\mathcal U_1[0,T]$ such that $J_1(\xi^{0},u_1^{0}(\cdot))<0.$ So, for any $k>0,$ $(k\xi^{0},ku_1^{0})$ is also admissible ($\mathcal{K}$ contains $0$ thus $k\xi^{0} \in \mathcal U_{C_{\infty}(\mathcal K)}).$ Thus, $J_1(k\xi^{0},ku_1^{0}(\cdot))=k^{2}J_1(\xi^{0},u_1^{0}(\cdot))\longrightarrow -\infty$ as $k\longrightarrow +\infty.$ Contradiction thus arises.
\end{proof}
We do not discuss if above result can be strengthen to be sufficient, with some additional conditions. However, in case $\mathcal{K}$ is conic, we do have the following equivalent result.
\begin{corollary}
Suppose \textbf{\emph{(H1)}}-\textbf{\emph{(H3)}} hold and $\mathcal{K}$ is closed-convex cone. Then, problem \emph{\textbf{(P)}} is finite iff $J_1$ is nonnegative on $\mathcal U_{\mathcal K}\times\mathcal U_1[0,T].$
\end{corollary}
\begin{proof}The necessary part follows from Lemma \ref{prop-finite of cost-leader} by noticing $C_{\infty}(\mathcal K)=\mathcal K$ when $\mathcal K$ is conic. The sufficient part is obvious.
\end{proof}
We point out closed-convex cone arises naturally from real applications, for example, $\mathcal{K}$ is positive orthant for no shorting constraint in finance portfolio selection (see \cite{ET2015,HZ2005,LZ18}). Combining Corollary 4.1 and Corollary 4.2, we have the following more explicit result:
\begin{corollary}
Suppose \textbf{\emph{(H1)}}-\textbf{\emph{(H3)}} hold and $\mathcal{K}$ is closed-convex cone. Then, \emph{\textbf{(P)}} is finite if it is convex.
\end{corollary}
We present some related remarks.
\begin{remark}The result of Corollary 4.3 differs from standard LQ problem \emph{(}see \emph{\cite{YZ1999}} pp. 287\emph{)} where finiteness implies convexity, but converse is not true. Also, by Proposition 4.2 and Lemma \ref{prop-finite of cost-leader}, for general convex set $\mathcal{K}$ \emph{(}not conic\emph{)}, the convexity and finiteness of problem \emph{(P)} have no direct relation. This also differs from standard LQ control where finiteness always implies convexity.\end{remark}
As implied by above, for (P) with general closed-convex set $\mathcal{K},$ it seems lacking tractable equivalent condition to characterize its finiteness. However, on the other hand, convexity is necessary to be established when we plan to apply Lagrange multiplier to tackle the involved constraints in (P). Thus, we primarily focus on \emph{convexity} and then discuss the related \emph{solvability} (that in turn implies \emph{finiteness}).

By representation \eqref{quadratic form of cost-leader}, the mapping $(\xi,u_1(\cdot))\longmapsto J_{1}(\xi,u_1(\cdot))$ is Fr\'{e}chet differentiable with Fr\'{e}chet derivative $\partial J_1=(\partial_{\xi} J_1, \partial_{u} J_1)$ given respectively by
\begin{equation}\label{Frechet 1}
\begin{aligned}
\partial_{\xi} J_1(\xi,u_1(\cdot))=\mathcal M_1(\xi)+\mathcal M_0^{*}(u_1),\quad\partial_{u} J_1(\xi,u_1(\cdot))=\mathcal M_2(u_1)+\mathcal M_0(\xi).
\end{aligned}
\end{equation}When (P) is convex, we have the following solvability result.
\begin{lemma}\label{optimal-leader-1}
If $\emph{\textbf{(P)}}$ is convex, then it is \emph{(}uniquely\emph{)} solvable iff there exists a \emph{(}unique\emph{)} minimizer $(\bar{\xi}, \bar{u}_1(\cdot))$ satisfying\begin{equation}\begin{aligned}
\langle\partial J_1(\bar{\xi}, \bar{u}_1(\cdot)), (\xi-\bar{\xi}, u_1-\bar{u}_1) \rangle \geq 0
\end{aligned}\Longleftrightarrow
\left\{\begin{aligned}
&\langle \mathcal M_1(\bar{\xi})+\mathcal M_0^{*}(\bar{u}_1), \quad \xi-\bar\xi\rangle \geq 0,\\
&\mathcal M_2(\bar{u}_1)+\mathcal M_0(\bar{\xi})=0,
\end{aligned}\right.
\end{equation}$\forall (\xi, u_{1}(\cdot)) \in \mathcal U(\mathcal K,\alpha,\beta)\times\mathcal U_1[0,T].$ If $\emph{\textbf{(P)}}$ is further strictly convex, then its minimizer\emph{(}s\emph{)}, if exist, should be unique.
\end{lemma}
The above criteria is called \emph{first-order regularity condition} for (global) optimality which is rather  constructive. A more direct and checkable condition for existence is as follows.
\begin{proposition}\label{optimal-leader-2}
If $\emph{\textbf{(P)}}$ is uniformly convex on $(\xi, u_1)$, then it admits an unique minimizer. \end{proposition}
\begin{proof}
We assume $\mathcal U(\widetilde{\mathcal K},\alpha,\beta)$ is not empty (otherwise, (P) becomes trivial), thus there exists $(\xi^{0}, u_1^{0})$ satisfying $-\infty<J_1(\xi^{0}, u_1^{0}).$ If $J_1$ is uniformly convex, it should also be \emph{coercive}, that is, $J_1(\xi, u_1)\longrightarrow +\infty$ as $||(\xi, u_1)|| \longrightarrow+\infty.$ To see this point, actually we have
\begin{equation}\nonumber
\begin{aligned}
&J_1(\xi, u_1)=J_1(\xi^{0}, u_1^{0})+J_1(\xi-\xi^{0}, u_1-u_1^{0})\\&+\Big[\langle\mathcal M_2(u_1-u_1^{0})(\cdot),u_1^0(\cdot)\rangle+\langle\mathcal M_1(\xi-\xi^{0}),\xi^{0}\rangle+\langle\mathcal M_0(\xi^{0}),u_1-u_1^{0})\rangle+\langle u_1^{0},\mathcal M_0(\xi-\xi^{0})\rangle\Big]\\
&\geq J_1(\xi^{0}, u_1^{0})+\delta ||(\xi-\xi^{0}, u_1-u_1^{0}) ||^{2}-\frac{||\mathcal M_{2}||^{2}+||\mathcal M_{1}||^{2}+||\mathcal M_{0}||^{2}}{\mu}||(\xi-\xi^{0}, u_1-u_1^{0}) ||^{2}-\frac{\mu}{2}||(\xi^{0}, u_1^{0}) ||^{2}\\&\geq J_1(\xi^{0}, u_1^{0})+\frac{\delta}{2} ||(\xi-\xi^{0}, u_1-u_1^{0}) ||^{2}-\frac{\mu}{2}||(\xi^{0}, u_1^{0}) ||^{2},\end{aligned}
\end{equation}for sufficiently large $\mu>0.$ Therefore, $J_1(\xi, u_1)\longrightarrow +\infty$ as $||(\xi, u_1)|| \longrightarrow+\infty.$ Note that Proposition \ref{optimal-leader-2} can only applied to $(\xi-\xi^{0}, u_1-u_1^{0}) \in  {\mathcal U}_{\widetilde{\mathcal K}}
\times\mathcal U_1[0,T]$ for uniformly convexity. In general, $(\xi, u_1)$ or $(\xi^{0}, u_1^{0}) \notin {\mathcal U}_{\widetilde{\mathcal K}}\times\mathcal U_1[0,T].$

Moreover, because $J_1$ is a proper quadratic functional with $\mathcal{M}_0, \mathcal{M}_1, \mathcal{M}_2$ being linear bounded operators thus $J_1(\cdot, \cdot)$ is also continuous (thus, lower semi-continuous (lsc)). By \cite{ET1976}, a lsc convex coercive functional admits at least one minimizer. Moreover, the uniform convexity of $J_1$ implies strict convexity thus (P) admits a unique minimizer. \end{proof}
We now discuss condition under which problem (P) becomes convex. First introduce the following standard assumption
\begin{equation*}\begin{aligned}
&\textbf{(SA-2):}\quad G_1\gg0,\quad H_1\geq0,\quad Q_1(\cdot)\geq0,\quad S_1(\cdot)\geq0,\quad R^1_{11}(\cdot)\gg0.
\end{aligned}\end{equation*}
%suffices the uniformly convexity on full space $L^2_{\mathcal F_T}(\Omega;\mathbb R^n)\times\mathcal U_1[0,T].$
Second, a more general sufficient condition to convexity is via the following stochastic Riccati equation (denoted by \textbf{(SRE-2)}):
\begin{equation*}\textbf{(SRE-2)}:\label{e1L RE}
\left\{\begin{aligned}
&dP_L=-\Big[\mathbb A^\top P_L+P_L\mathbb A+\mathbb C^\top P_L\mathbb C+\mathbb Q+\Lambda_L \mathbb C+\mathbb C^{\top}\Lambda_L-\Big(\mathbb B^\top P_L+\mathbb D^\top P_L\mathbb C+\mathbb D^{\top}\Lambda_L\Big)^\top\\
&\qquad\qquad\mathbb K^{-1}\Big(\mathbb B^\top P_L+\mathbb D^\top P_L\mathbb C+\mathbb D^{\top}\Lambda_L\Big)\Big]ds+\Lambda_L dW(s),\\
&P_L(T)=\left(\begin{matrix} 0&0\\0&G_1\end{matrix}\right),\\
&\mathbb K(s)\triangleq\mathbb R(s)+\mathbb D^\top(s) P_L(s)\mathbb D(s)>0,
\end{aligned}\right.\end{equation*}where\begin{equation}\nonumber\begin{aligned}
&\mathbb A=\left(\begin{matrix} -A^\top&Q_2\\B_2(R^2_{22})^{-1}B_2^\top &A\end{matrix}\right),\mathbb B=\left(\begin{matrix} 0&0\\ B_1 &C\end{matrix}\right),\mathbb C=\left(\begin{matrix} -C^\top&0\\0&0\end{matrix}\right),
\mathbb D=\left(\begin{matrix} 0&S_2\\0 &I\end{matrix}\right),\mathbb Q=\left(\begin{matrix} 0&0\\0 &Q_1\end{matrix}\right),\mathbb R=\left(\begin{matrix}  R^1_{11}&0\\0 &S_1\end{matrix}\right).
\end{aligned}
\end{equation}We have the following result concerning convexity and its proof is given in the Appendix, Section \ref{sec7.2}.
\begin{proposition}\label{Proposition leader cost functional convex}
Suppose \textbf{\emph{(SRE-2)}} has a solution $(P_L(\cdot),\Lambda_L(\cdot))\in L_{\mathbb F}^\infty(0,T;\mathbb S^n)\times L_{\mathbb F}^2(0,T;\mathbb S^n)$ such that
$$\left(\begin{matrix}0&0\\0&H_1\end{matrix}\right)+P_L(0)\geq0.$$Then, $J_1(\cdot,\cdot)$ is a convex functional with $(\xi,u_1(\cdot))$ over $L^2_{\mathcal F_T}(\Omega;\mathbb R^n)\times\mathcal  \mathcal U_1[0,T]$. In particular, under \emph{(\textbf{SA-2})}, $J_1(\cdot,\cdot)$ is uniformly convex over $L^2_{\mathcal F_T}(\Omega;\mathbb R^n)\times\mathcal U_1[0,T].$
\end{proposition}Proposition \ref{optimal-leader-2} only specifies the existence of optimal solution to (P) but does not discuss how to characterize such solution. This will be discussed below through some Lagrange multiplier method to (P). Our target is to remove the affine-expectation constraint and only keep pointwise constraint.

Further study of (P) involves some Lagrange duality for which we need first address the relevant feasibility, as given below.

\subsection{Feasibility of problem (P) constraints}\label{feasibility of constraint}
Recall problem (P) involves two (pointwise, affine-expectation) constraints, thus it is necessary to discuss their joint feasibility.
To start, for any convex-closed proper subset ${\mathcal K}\subset \mathbb{R}^{n},$ we can introduce its support functional: $h_{\mathcal K}^{*}(p)\triangleq\sup_{x \in \mathcal{K}}\langle p, x\rangle \in [0, +\infty].$ Its effective domain (i.e., $\{p: h_{\mathcal K}^{*}(p)<+\infty\}$) is $B(\mathcal{K}),$ the barrier cone of $\mathcal{K}.$ In particular, when $\mathcal{K}$ is convex-closed cone, then $B(\mathcal{K})$ is negative polar cone of $\mathcal{K}.$

Moreover, $-h_{\mathcal K}^{*}(-p)=\inf_{x \in \mathcal{K}}\langle p, x\rangle$ and $h_{\mathcal K}^{*}(p)+h_{\mathcal K}^{*}(-p) \in [0, +\infty]$ is called the \emph{breadth} for nonempty $\mathcal{K}$ along direction $p.$ The breadth takes value $0$ iff $\mathcal{K}$ is subset of affine hyperplane $\{y: \langle y, p\rangle=h_{\mathcal K}^{*}(p)\}$ which is orthogonal to $p.$ Now, we can discuss the feasibility of constrained $\mathcal U(\mathcal K,\alpha,\beta)$.

We first claim the following fundamental result that is obvious in its scalar case $(n=1)$ but not straightforward in vector case. A similar result may be found in \cite{Clark2013} pp. 44.
\begin{lemma}\label{property of expectation}
$\forall \xi \in \mathcal U_{\mathcal K}, \mathbb{E}\xi \in \mathcal K.$
\end{lemma}
\begin{proof}Recall that any convex-closed set ${\mathcal K}\subset\mathbb{R}^{n}$ can be equivalently defined as the intersection of all closed half-spaces containing it, thus for a.s. $\omega,$ $
\langle s_{j}, \xi(\omega) \rangle \leq r_{j}$ for some data $(s_{j}, r_{j}) \in \mathbb{R}^{n} \times \mathbb{R}$ from some index set $j \in J$. By linearity of expectation, $
\langle s_{j}, \mathbb{E}\xi\rangle \leq r_{j}$ for all $j \in J$ also, thus $\mathbb{E}\xi \in \mathcal K.$ Another proof is based on support functional as follows. $x \in {\mathcal K}$ iff $\langle x, p \rangle \leq h_{\mathcal K}^{*}(p)$ for each vector $p$. Again, by linearity of expectation, $\langle \mathbb{E}\xi, p \rangle \leq h_{\mathcal K}^{*}(p)$ for each vector $p$, hence $\mathbb{E}\xi \in \mathcal K.$
\end{proof}

By Lemma \ref{property of expectation}, a necessary condition for $\mathcal U(\mathcal K,\alpha,\beta)$ being non-empty is ${\mathcal K}_{\alpha, \beta}^{+}\triangleq \mathcal K \cap H_{\alpha,\beta}^{+} \neq \emptyset$ where $H_{\alpha,\beta}^{+}=\{x\in \mathbb{R}^{n}:\langle\alpha, x\rangle \geq \beta\}$ is one half-space delimited by the affine hyperplane $H_{\alpha,\beta}: \langle\alpha, x\rangle=\beta.$
Further discussion of feasibility to $\mathcal U(\mathcal K,\alpha,\beta)$, may depend on the following alternative assumptions.

({\textbf{F1}})(positive breadth along $\alpha$):\quad \quad  $ \quad h_{\mathcal K}^{*}(\alpha)+h_{\mathcal K}^{*}(-\alpha)>0.$

({\textbf{F2}})(degenerated breadth along $\alpha$): \quad \quad $ \quad h_{\mathcal K}^{*}(\alpha)+h_{\mathcal K}^{*}(-\alpha)=0.$

Depending on (F1) or (F2), we have the following feasibility results respectively. \begin{proposition}\label{feasible-1}
Under \textbf{\emph{(F1)}}, the terminal admissible set $\mathcal U(\mathcal K,\alpha,\beta)\triangleq\mathcal U_{\mathcal K}\bigcap\mathcal U_{\alpha,\beta}$ is \begin{itemize}
\item \emph{(i)} nontrivial \emph{(}non-empty and admitting two constraints both\emph{)}, if $-h_{\mathcal K}^{*}(-\alpha)<\beta<h_{\mathcal K}^{*}(\alpha);$
\item \emph{(ii)} trivial \emph{(}being reduced to pointwise constraint $\mathcal U_{\mathcal K}$ only\emph{)}, if $\beta \leq -h_{\mathcal K}^{*}(-\alpha);$
\item \emph{(iii)} trivial \emph{(}empty set\emph{)}, if $\beta>h_{\mathcal K}^{*}(\alpha);$
\item \emph{(iv)} trivial \emph{(}degenerated to the exposed face of $\mathcal{K}$\emph{)}, if $\beta=h_{\mathcal K}^{*}(\alpha).$\end{itemize}\end{proposition}

\begin{proposition}\label{feasible-2}
Under \textbf{\emph{(F2)}}, the terminal admissible set $\mathcal U(\mathcal K,\alpha,\beta)\triangleq\mathcal U_{\mathcal K}\bigcap\mathcal U_{\alpha,\beta}$ is \begin{itemize}
\item \emph{(ii')} trivial \emph{(}being reduced to pointwise constraint $\mathcal U_{\mathcal K}$ only\emph{)}, if $\beta \leq h_{\mathcal K}^{*}(\alpha);$
\item \emph{(iii')} trivial as being empty, if $\beta>h_{\mathcal K}^{*}(\alpha).$\end{itemize}\end{proposition}
The proofs of Propositions \ref{feasible-1}-\ref{feasible-2} follow from standard convex analysis, and readers may refer \cite{Rockafellar} Chapters 4 and 5. Of course, we are more interested to the nontrivial case (i). Some related remarks are as follows.
\begin{remark}
\emph{(a)} When $\mathcal{K}$ is bounded \emph{(}hence compact\emph{)}, $B(\mathcal{K})=\mathbb{R}^{n}$ thus $-\infty<-h_{\mathcal K}^{*}(-\alpha)<h_{\mathcal K}^{*}(\alpha)<+\infty$ and \emph{(i)} always holds true for all affine-expectation constraint pairs $(\alpha, \beta) \in \mathbb{R}^{n}\times (-h_{\mathcal K}^{*}(-\alpha), h_{\mathcal K}^{*}(\alpha)).$

\emph{(b)} For unbounded $\mathcal{K}$, its asymptotic cone provides more explicit representation of $B(\mathcal{K})$ and the range qualification to $(\alpha, \beta)$ jointly. We omit details here.

\emph{(c)} Notice that \emph{(iv)} above involves the exposed face. Recall for convex set $K$, a set $F$ is called its exposed face if there is a supporting hyperplane $H_{s,r}$ of $K$ such that $F=H_{s,r} \cap K.$ For unbounded $\mathcal{K}$, there have some subtle difference between exposed face and boundary of $\mathcal{K}.$

\end{remark}
    It is obvious that ${\mathcal K}_{\alpha, \beta}^{+}$ is convex-closed set. We can introduce $\mathcal U_{{\mathcal K}_{\alpha, \beta}^{+}}=L^2_{\mathcal F_T}(\Omega;{\mathcal K}_{\alpha, \beta}^{+})$ that satisfies $\mathcal U_{{\mathcal K}_{\alpha, \beta}^{+}}\subset \mathcal U(\mathcal K,\alpha,\beta)$ by Lemma \ref{property of expectation}. Noticing the inclusion here is strictly proper subset by noting, say, in scalar case, it is not very hard to construct a random variable with support on $\mathcal{K}=[0,1]$ but with expectation on $[\frac{1}{2}, +\infty)$(i.e., $\alpha=1, \beta=\frac{1}{2})$.

We continue to discuss the \emph{strict} feasibility that relates to Slater qualification to be invoked. To start, we first present some relative interior point result for pointwise constraint $\mathcal U_{\mathcal K}.$

\begin{proposition}\label{relative interior}
The constrained set $\mathcal U_{\mathcal K}$ admits no relative interior point.
\end{proposition}
\begin{proof}In case $\dim{\mathcal{K}}=n$, then $\text{aff}(\mathcal{K})=\text{aff}({\mathcal K}_{\alpha, \beta}^{+})=\mathbb{R}^{n},$ and $\text{aff}(\mathcal U_{\mathcal K}) \supseteq \text{aff}(\mathcal U_{{\mathcal K}_{\alpha, \beta}^{+}})=L^2_{\mathcal F_T}(\Omega;\text{aff}({\mathcal K}_{\alpha, \beta}^{+}))=L^2_{\mathcal F_T}(\Omega;\mathbb{R}^{n}).$ Then, $\text{aff}(\mathcal U_{\mathcal K})=L^2_{\mathcal F_T}(\Omega;\mathbb{R}^{n}).$ On the other hand, for any $\xi \in \mathcal U_{\mathcal K},$ we can always construct $\xi' \in B(\xi, \varepsilon)\subset L^2_{\mathcal F_T}(\Omega;\mathbb{R}^{n}),$ a small ball centered at $\xi$ with radius $\varepsilon>0,$ but support $\xi' \in \mathcal{K}^{c}$ with positive probability. Similar arguments can be applied to the case of $\dim{\mathcal{K}}<n.$ \end{proof}

Based on Proposition \ref{relative interior}, to apply the Lagrange multiplier method, its Slater qualification condition holds true iff ${\mathcal K}_{\alpha, \beta}^{++}\triangleq \mathcal K \cap H_{\alpha,\beta}^{++} \neq \emptyset$ with $H_{\alpha,\beta}^{++}=\{x\in \mathbb{R}^{n}:\langle\alpha, x\rangle >\beta\}$ being the strict half-space (noticing a crucial point here is that $\mathcal{U}_{\alpha, \beta}$ is an inequality constraint on (linear) affine expectation). Actually, for any $y \in {\mathcal K}_{\alpha, \beta}^{++},$ $\xi(\omega) \equiv y \quad a.s. \in \mathcal U(\mathcal K,\alpha,\beta)$ and satisfies the affine-expectation constraint strictly. Conversely, if there has any random variable $\xi$ satisfying affine-expectation constraint strictly, and $\xi \in \mathcal{U}_{\mathcal{K}},$ it is necessary to have non-empty ${\mathcal K}_{\alpha, \beta}^{++}$ for $\mathbb{E}\xi$ by Lemma \ref{property of expectation}.

In summary to Propositions 4.5-4.7, we set the following assumption under which $\mathcal U(\mathcal K,\alpha,\beta)$ is nontrivial, strictly feasible and Slater constraint qualification holds true.\\

({\textbf{F}}) The triple $( \mathcal{K}, \alpha, \beta)$ of terminal constraint parameter satisfy: \quad $-h_{\mathcal K}^{*}(-\alpha)<\beta<h_{\mathcal K}^{*}(\alpha).$

\subsection{Solution of primal problem (P) via duality}We introduce the following dual problem (D) associated to the primal (P):

\begin{equation*}\label{Langrage problem}\textbf{(D)}:
\text{Maximize} \qquad K(\lambda)\triangleq  \inf_{(\xi,u_1(\cdot)) \in \mathcal{U}_{\mathcal{K}}\times\mathcal U_1[0,T]} L(\lambda; \xi, u_{1}(\cdot))
 \text{\qquad subject to } \qquad \lambda \geq 0,
\end{equation*}
where $L(\lambda; \xi, u_{1}(\cdot))\triangleq J_1(\xi,u_1(\cdot))+\lambda(\beta-\mathbb E\langle\alpha, \xi\rangle)$ is called the \emph{Lagrange functional}, $K(\cdot)$ is called \emph{dual function} which is parallel to primal functional $J_{1}(\cdot,\cdot).$ Dual function $K(\cdot)$ is always concave (even $J_1(\cdot,\cdot)$ is not convex) since it is defined by infimum operation on a family of affine functionals.

We can introduce an auxiliary problem (KT) for given $\lambda_0 \geq 0$:\begin{equation*}\label{leader-BFLQ}\textbf{(KT)}:\begin{aligned}
\text{ Minimize } \qquad L(\lambda_{0}; \xi,u_1(\cdot))
\text{\qquad subject to } \qquad \eqref{state equation-leader},\ (\xi,u_1(\cdot))\in\mathcal U_{\mathcal K}\times\mathcal U_1[0,T].
\end{aligned}
\end{equation*}We stress that here $\xi \in \mathcal U_{\mathcal K}$ instead $\mathcal U(\mathcal{K}, \alpha, \beta)$ as in (P). Now, we can introduce the following definitions based on \cite{Rockafellar}.
\begin{definition} \emph{(Kuhn-Tucker  coefficient)} A Kuhn-Tucker coefficient \emph{(KT-coefficient)} for problem \emph{\textbf{(P)}} is any $\lambda_{0} \geq0$ satisfying $-\infty<K(\lambda_{0})=\mu_{p}.$

\emph{(KT-admissible)} Problem \emph{\textbf{(P)}} is said to be \emph{KT-admissible} if it has at least one \emph{KT-coefficient}.
\end{definition}Definition 4.3 imposes no assumption on existence of optimal solutions to primal {(P)}, dual {(D)} and {(KT)}.
Similar to {(P)}, we can further introduce the following definitions.
\begin{definition} \emph{(a)} Problem \emph{\textbf{(D)}} is said to be finite if $\mu_{d}\triangleq \sup_{\lambda \geq 0}K(\lambda)<+\infty,$ and $\mu_{d}$ is called the value of \emph{\textbf{(D)}}\emph{;}

\emph{(b)} Problem \emph{\textbf{(D)}} is said to be \emph{(}uniquely\emph{)} solvable if there exists a \emph{(}unique\emph{)} $\lambda^{*} \geq 0$ such that $\mu_{p}=K(\lambda^{*})$ and $\lambda^{*}$ is called maximizer of \emph{\textbf{(D)}}\emph{;}

\emph{(c)} Problem \emph{\textbf{(KT)}} is said to be finite if $K(\lambda_{0})>-\infty,$ and $K(\lambda_{0})$ is the value of \emph{\textbf{(KT)}}\emph{;}

\emph{(d)} Problem \emph{\textbf{(KT)}} is said to be \emph{(}uniquely\emph{)} solvable if there exists a \emph{(}unique\emph{)} $(\xi,u_1(\cdot))\in\mathcal U_{\mathcal K}\times\mathcal U_1[0,T]$ such that $K(\lambda_{0})=L(\lambda_{0}; \xi^{*},u_1^{*}(\cdot))$ and $(\xi^{*},u_1^{*}(\cdot))$ is called minimizer of \emph{\textbf{(KT)}}.
\end{definition}
The following relations among problem {(P)}, {(D)} and {(KT)} are obvious.
\begin{proposition}\label{KT admissible}
 \emph{(a)} If Problem \emph{\textbf{(P)}} is \emph{KT-admissible}, then it is finite.

\emph{(b)} The values of problem \emph{\textbf{(P)}, \textbf{(D)}} and \emph{\textbf{(KT)}} parameterized by $\lambda_0 \geq 0,$ always satisfy: $K(\lambda_{0})\leq \mu_{d} \leq \mu_{p}$ where $\mu_{p}-\mu_{d} \geq 0$ is called the duality gap. \end{proposition}
Note that (P) and (KT) in Proposition \ref{KT admissible} need not to be convex. Moreover, we have the following solvability relations among (P), (D) and (KT), which follow from convex analysis (e.g., see \cite{Rockafellar} Part VI) and proof details are omitted here:
\begin{lemma}\label{lemma4.1}
\emph{(a)} If Problem \emph{\textbf{(P)}} is \emph{KT-admissible}, then duality gap is $0$ \emph{(}namely, strong duality holds\emph{)} and problem \emph{\textbf{(D)}} is solvable. Note here, \emph{\textbf{(P)}} may not be convex.

\emph{(b)} If \emph{\textbf{(P)}} is \emph{KT-admissible}, convex and related \emph{\textbf{(KT)}} problem with \emph{KT-coefficient} $\lambda_{0}$ is solvable with optimal solution set $D=\{(\bar{\xi},\bar{u}_1(\cdot)): K(\lambda_{0})=L(\lambda_{0}; \bar{\xi}, \bar{u}_1(\cdot))\}.$ Then, the subset $D_{p}$ of $D$ satisfying complementary slackness condition: $\lambda_{0}(\beta-\mathbb E\langle\alpha, \bar\xi\rangle)=0$, is the optimal solution set to primal \emph{\textbf{(P)}}.\end{lemma}

\begin{remark}
We remark that in \emph{(}a\emph{)} above, problem \emph{\textbf{(P)}} and \emph{\textbf{(KT)}} may not be solvable even \emph{\textbf{(D)}} is solvable. Also, in \emph{(}b\emph{)}, \emph{\textbf{(KT)}} solvability  does not imply solvability of \emph{\textbf{(P)}}, conversely, solvability of primal \emph{\textbf{(P)}} does not imply it is \emph{KT}-solvable or even \emph{KT}-admissible.
\end{remark}
Part (b) of Lemma \ref{lemma4.1} specifies some \emph{sufficient} condition to find all optimal solutions to primal problem (P). In usual cases, we are more interested to equivalent condition for (P) solvability, and we thus report the following result which proof can be referred from  \cite{Rockafellar}  Part VI.

\begin{theorem}\label{relation between BFLQ and LBFLQ}
Assume \emph{\textbf{(H1)-(H3)}} and suppose \emph{\textbf{(P)}} is convex, then the following three statements\emph{:} \emph{\textbf{(i)}, \textbf{(ii)},} and \emph{\textbf{(iii)}} are equivalent\emph{:}

\emph{\textbf{(i):}} \emph{\textbf{(P)}} is \emph{KT-admissible} with coefficient $\lambda_0$, and \emph{\textbf{(P)}} is solvable with minimizer $(\xi^{*},u_1^{*}(\cdot))$\emph{;}

\emph{\textbf{(ii):}} The triple $(\lambda_0; \xi^{*},u_1^{*}(\cdot)) \in [0, +\infty)\times \mathcal U_{\mathcal K}\times\mathcal U_1[0,T]$ satisfies the following \emph{Karush-Kuhn-Tucker ({\textbf{KKT}})} system: \begin{equation}\begin{aligned}
\beta\leq\mathbb E\langle\alpha, \bar\xi\rangle, \quad
\bar\lambda(\beta-\mathbb E\langle\alpha, \bar\xi\rangle)=0;\quad
K(\lambda_{0})=L(\lambda_{0}; \bar{\xi}, \bar{u}_1(\cdot));
\end{aligned}\end{equation}\emph{\textbf{(iii):}} The triple $(\lambda_0; \xi^{*},u_1^{*}(\cdot))$ is a saddle point for Lagrange functional $L$:\begin{equation*}\label{saddlepoint condition}
L(\lambda; \bar\xi,\bar u_1(\cdot))\leq L(\bar\lambda; \bar\xi,\bar u_1(\cdot))\leq L(\bar\lambda;\xi,u_1(\cdot)).
\end{equation*}\end{theorem}
In Theorem 4.1, the KT-admissible and its coefficient $\lambda_{0}$ plays some crucial role. Thus, we present some sufficient condition ensuring them.
\begin{proposition}\label{KT-admissible}
Assume \emph{\textbf{(H1)-(H3)}}, and suppose  problem \emph{\textbf{(P)}} is convex, finite. Moreover, suppose feasibility condition \emph{(\textbf{F})} holds true, then \emph{\textbf{(P)}} is \emph{KT-admissible} for some $\lambda_{0} \geq0.$
\end{proposition}
\begin{proof}When (F) holds true, then (P) satisfies the Slater qualification condition hence it is also KT-admissible by \cite[Corollary 28.2.1]{Rockafellar}, considering (P) is finite and convex. Hence the result.
\end{proof}Noticing assumption (F) is crucial in above and the following example indicates it can usually be expected. We just present its scalar case for illustration, and the vector case can be constructed similarly.
\begin{example}
In case $n=1,$ suppose $\frac{\beta}{\alpha}\in \mathcal K^o$, where $\mathcal K^o$ is the interior of $\mathcal K$.
Then, \emph{(\textbf{F})} holds.
\end{example}
Introduce the following assumption:
\begin{description}
	\item[(H4)] $G_1>0$. $ R^1_{11}(\cdot)$ is invertible and $(R^1_{11}(\cdot))^{-1}\in L^\infty_{\mathbb F}(0,T;\mathbb R^{m_1})$.
\end{description}
\begin{lemma}\label{th 1L general LGL}
Let \textbf{\emph{(H1)}}-\textbf{\emph{(H4)}} hold and \emph{\textbf{(P)}} is convex. Then, \emph{\textbf{(KT)}} parameterized by coefficient $\lambda\geq 0$ is \emph{(}uniquely\emph{)} solvable iff the following BFSDEs \begin{equation*}\label{general case Hamiltonian system}
\textbf{\emph{(BFSDE-1)}}:\left\{\begin{aligned}	
&dg=-\Big[A^\top  g-Q_1\bar X-Q_2h\Big]ds-\Big[C^\top g-S_1\bar Z-S_2q\Big]dW(s),\\
&d\bar Y=\Big[-A^\top \bar Y+Q_2\bar  X\Big]ds+\Big[-C^\top \bar Y+S_2\bar  Z\Big]dW(s),\\
&d\bar X=\Big[A\bar  X+B_1(R^1_{11})^{-1}B_1^\top g+B_2(R^2_{22})^{-1}B_2^\top\bar Y+C\bar Z\Big]ds+ \bar ZdW(s),\\
&dh=\Big[ Ah+B_2(R^2_{22})^{-1}B^\top_2 g+ Cq\Big]ds+qdW(s),\\
& g(0)=H_1\bar X(0)+ H_2 h(0),\quad \bar Y(0)=H_2\bar  X(0),\\
& \bar  X(T)=\text{Proj}_{\mathcal K}\Big[G_1^{-1}(-g(T)+\lambda\alpha)\Big],\quad h(T)=0,
\end{aligned}\right.\end{equation*}
admits a \emph{(}unique\emph{)} solution $(\bar Y,g,\bar X,\bar Z,h,q) \in  L^2_{\mathbb F}(\Omega;C([0,T];\mathbb R^n))\times L^2_{\mathbb F}(\Omega;C([0,T];\mathbb R^n))\times L^2_{\mathbb F}(0,T;\mathbb R^n)\times L^2_{\mathbb F}(\Omega;C([0,T];\mathbb R^n))\times L^2_{\mathbb F}(\Omega;C([0,T];\mathbb R^n))\times L^2_{\mathbb F}(0,T;\mathbb R^n)$, where $\text{Proj}_{\mathcal K}(\cdot)$ is the projection mapping from $\mathbb R^n$ to closed-convex set $\mathcal K$ under the norm $|x|^2_{G_1}\triangleq\langle G_1^\frac{1}{2}x,G_1^\frac{1}{2}x\rangle$. In this case, the \emph{(}unique\emph{)} minimizer $(\bar \xi,\bar u_1(\cdot))$ to \textbf{\emph{(KT)}} with coefficient $\lambda$ is given by$$(\bar \xi, \bar u_1(\cdot))=\Big(\text{Proj}_{\mathcal K}\Big[G_1^{-1}(-g(T)+\lambda\alpha)\Big], \quad (R^1_{11}(\cdot))^{-1}B_1^\top(\cdot) g(\cdot)\Big).$$
\end{lemma}
\begin{proof} Note that {(P)} is convex, then for any $\lambda \geq0 ,$ the Lagrange functional $L(\lambda;\xi,u_1(\cdot))$ thus (KT) are also convex.
Similar to Proposition \ref{presentation of cost-leader}, we have that
\begin{equation}\nonumber
L(\lambda; \xi,u_1(\cdot))=\frac{1}{2}\mathbb E\Big[\langle\mathcal M_2(u_1)(\cdot),u_1(\cdot)\rangle+\langle\mathcal M_1(\xi)-2\lambda\alpha,\xi\rangle+2\langle\xi,\mathcal M_0^*(u_1(\cdot))\rangle+2\lambda \beta  \Big].
\end{equation}
Consequently, similar to Lemma \ref{optimal-leader-1}, problem (KT) is solvable iff there exists a pair $(\xi^{*}, u_{1}^*(\cdot))$ satisfying
\begin{equation}\label{eq:27}
\left\{\begin{aligned}
&\langle \mathcal M_1(\xi^*)-\lambda\alpha+\mathcal M_0^{*}(u_1^*), \xi^*-\xi_1 \rangle \leq 0, \quad \text{$\forall\xi_1 \in \mathcal{U}_{\mathcal{K}}$},\\
&\mathcal M_2(u_1^*)+\mathcal M_0(\xi^*)=0.
\end{aligned}\right.
\end{equation}
Let $(\bar\xi,\bar u_1(\cdot))$ be an optimal control,
by \eqref{eq:27}, we have
\begin{equation}\label{e state equation 2}
\left\{\begin{aligned}
&\mathbb E\langle-g_1(T)-g_2(T)+\lambda\alpha-G_1\bar\xi ,\xi_1-\bar\xi\rangle\leq 0, \quad \forall \xi_1 \in \mathcal{U}_{\mathcal{K}},\\
&R^1_{11}(s)\bar u_1(s)-B_1^\top(s)g_1(s)-B_1^\top(s)g_2(s)=0,
\end{aligned}\right.
\end{equation}
where $(Y_1,g_1,X_1,Z_1,h_1,q_1)$ and $(Y_2,g_2,X_2,Z_2,h_2,q_2)$ are the solutions of \eqref{adjoint-1} and \eqref{adjoint-2} corresponding to $(\bar \xi,\bar u_1)$, respectively. Let
\begin{equation*}\begin{aligned}
\bar Y=Y_1+Y_2,\quad\bar X=X_1+X_2,\quad\bar Z=Z_1+Z_2,\quad g=g_1+g_2,\quad h=h_1+h_2,\quad q=q_1+q_2,
\end{aligned}\end{equation*}
and it follows that $(\bar Y,g,\bar X,\bar Z,h,q)$ satisfying (BFSDE-1).
Under (H4), it follows from \eqref{e state equation 2} that
$$\bar u_1(\cdot)=(R^1_{11}(\cdot))^{-1}B_1^\top(\cdot) g(\cdot),$$
and
\begin{equation}\nonumber
\mathbb E\Big\langle G_1^{\frac{1}{2}}[G_1^{-1}(-g(T)+\lambda\alpha)-\bar\xi],G_1^{\frac{1}{2}}(\xi_1-\bar\xi)\Big\rangle\leq 0, \quad\forall \xi_1\in\mathcal U_{\mathcal K}.
\end{equation}
Note that $|\cdot|_{G_1}$ is equivalent to the Euclidean norm.
Let $\xi_1=\text{\emph{Proj}}_{\mathcal K}[G_1^{-1}(-g(T)+\lambda\alpha)]$, then by Propositions 4.1 and 4.3 in \cite{HHL2017}, we have
	\begin{equation}\nonumber
	\begin{aligned}
	&\mathbb E\Big|\text{\emph{Proj}}_{\mathcal K}[G_1^{-1}(-g(T)+\lambda\alpha)]-\bar\xi\Big|^2_{G_1}
	\leq \mathbb E\Big\langle G_1^{\frac{1}{2}}[G_1^{-1}(-g(T)+\lambda\alpha)-\bar\xi],G_1^{\frac{1}{2}}(\xi_1-\bar\xi)\Big\rangle
	\leq 0.
	\end{aligned}
	\end{equation}
	Thus, we get
$$
\bar\xi=\text{\emph{Proj}}_{\mathcal K}\Big[G_1^{-1}(-g(T)+\lambda\alpha)\Big].
$$
The uniqueness follows from the uniqueness of the solution of (BFSDE-1).
\end{proof}
Combing Theorem \ref{relation between BFLQ and LBFLQ}, Proposition \ref{KT-admissible} and Lemma \ref{th 1L general LGL}, we have
\begin{theorem}\label{th 1L general GL lambda}
Let \emph{\textbf{(H1)-(H4)}} hold. Suppose \emph{(\textbf{F})} hold and \emph{\textbf{(P)}} is convex and finite, then \emph{\textbf{(P)}} is \emph{KT-admissible} with some coefficient $\lambda_0 \geq0$. Moreover, \emph{\textbf{(P)}} is solvable with an  optimal solution $(\bar\xi,\bar u_1(\cdot))$ iff there exist a  $7$-tuple $(\lambda;\bar Y,g,\bar X,\bar Z,h,q)$ satisfying both \emph{\textbf{(BFSDE-1)}} and \emph{{\textbf{(KKT)}}} system:
\begin{equation}\label{e1L general case two lambda}\left\{\begin{aligned}
&\text{\emph{complimentary slackness:}}  \quad \quad  \lambda\Big(\beta-\mathbb E\Big\langle\alpha, \text{Proj}_{\mathcal K}\Big[G_1^{-1}(-g(T)+\lambda\alpha)\Big] \Big\rangle\Big)=0;\\
& \emph{\text{primal- and dual-constraint:}} \quad \lambda \geq0;  \quad \beta \leq \mathbb E\Big\langle\alpha, \text{Proj}_{\mathcal K}\Big[G_1^{-1}(-g(T)+\lambda\alpha)\Big]\Big\rangle.
\end{aligned}\right.\end{equation}
In this case, $\lambda$ is a  \emph{KT-coefficient} of \emph{\textbf{(P)}}, and an  optimal solution to problem \emph{\textbf{(P)}} is given by$$(\bar\xi,\bar u_1(\cdot))=\Big(\text{Proj}_{\mathcal K}\Big[G_1^{-1}(-g(T)+\lambda\alpha)\Big],(R^1_{11}(\cdot))^{-1}B_1^\top(\cdot) g(\cdot)\Big).$$
\end{theorem}
As a corollary, we have

\begin{corollary}Let \emph{\textbf{(H1)-(H4)}} and \emph{(\textbf{F})} hold true. Suppose \emph{\textbf{(P)}} is uniformly convex, then it admits a unique optimal solution $(\bar\xi,\bar u_1(\cdot))=\Big(\text{Proj}_{\mathcal K}\Big[G_1^{-1}(-g(T)+\lambda\alpha)\Big],(R^1_{11}(\cdot))^{-1}B_1^\top(\cdot) g(\cdot)\Big)$ with $(\lambda; \bar Y,g,\bar X,\bar Z,h,q)$ is a solution for system \emph{\textbf{(BFSDE-1)}} and \textbf{\emph{(KKT)}} system.
\end{corollary}

\subsection{Some special cases}This subsection will consider two special cases of problem (P) with more detailed analysis.

\subsubsection{Pointwise constraint}
This subsection considers the case with only pointwise constraint $\mathcal U_{\mathcal K}$. In this special case, Problem (P) now assumes the following form
\begin{equation}\nonumber
\begin{aligned}
\text{\textbf{(}\textbf{P}$_1$\textbf{)}: Minimize } \qquad J_1(\xi,u_1(\cdot))\
\text{\qquad subject to }\qquad  \eqref{state equation-leader},\ (\xi,u_1(\cdot))\in\mathcal U_{\mathcal K}\times\mathcal U_1[0,T].
\end{aligned}
\end{equation}
By Lemma \ref{th 1L general LGL}, we have the following result.

\begin{corollary}\label{th general terminal optimal}
Let \textbf{\emph{(H1)}}-\textbf{\emph{(H4)}} hold and \emph{\textbf{(}}\textbf{\emph{P}}$_1$\emph{\textbf{)}} is convex. Then \emph{\textbf{(}}\textbf{\emph{P}}$_1$\emph{\textbf{)}} admits an \emph{(}unique\emph{)} optimal control $(\bar \xi,\bar u_1(\cdot))$ iff the following BFSDEs
\begin{equation*}\textbf{\emph{(BFSDE-2)}}:\label{special case-pointwise constraint}
\left\{\begin{aligned}	
&dg=-\Big[A^\top g-Q_1\bar X-Q_2h\Big]ds-\Big[C^\top g-S_1\bar Z-S_2q\Big]dW(s),\\
&d\bar Y=\Big[-A^\top\bar Y+Q_2\bar  X\Big]ds+\Big[-C^\top \bar Y+S_2\bar  Z\Big]dW(s),\\
&d\bar X=\Big[A\bar  X+B_1(R^1_{11})^{-1}B_1^\top g+B_2(R^2_{22})^{-1}B_2^\top\bar Y+C\bar Z\Big]ds+ \bar ZdW(s),\\
&dh(s)=\Big[ Ah+B_2(R^2_{22})^{-1}B^\top_2 g+ Cq\Big]ds+qdW(s),\\
& g(0)=H_1\bar X(0)+ H_2 h(0),\quad \bar Y(0)=H_2\bar  X(0),\quad\bar  X(T)=\text{Proj}_{\mathcal K}[-G_1^{-1}g(T)],\quad h(T)=0,
\end{aligned}\right.\end{equation*}
admits a \emph{(}unique\emph{)} solution $(\bar Y,g,\bar X,\bar Z ,h,q).$ Moreover, a \emph{(}the\emph{)} minimizer of \emph{\textbf{(}}\textbf{\emph{P}}$_1$\emph{\textbf{)}} is given by
\begin{equation}\label{optimal-case 1}
(\bar \xi,\bar u_1(\cdot))=\Big(\text{Proj}_{\mathcal K}[-G_1^{-1}g(T)],\quad (R^1_{11}(\cdot))^{-1}B_1^\top(\cdot) g(\cdot)\Big).
\end{equation}
\end{corollary}

\subsubsection{Affine constraint}
This subsection focus on the case with only constraint $\mathcal{U}_{\alpha,\beta}$ for terminal variable $\xi$. In this case, (P) takes the following form:
\begin{equation*}\begin{aligned}
\text{ \textbf{(}\textbf{P}$_2$\textbf{)}: Minimize } \qquad J_1(\xi,u_1(\cdot))
\text{\qquad subject to } \qquad \eqref{state equation-leader},\ (\xi,u_1(\cdot))\in \mathcal U_{\alpha,\beta}\times\mathcal U_1[0,T].
\end{aligned}
\end{equation*}
By Theorem \ref{th 1L general GL lambda}, we have the following result.

\begin{corollary}\label{special case-affine constrait}
Let \emph{\textbf{(H1)-(H4)}} hold  and suppose \emph{\textbf{(}}\textbf{\emph{P}}$_2$\emph{\textbf{)}} is convex and finite, then \emph{\textbf{(}}\textbf{\emph{P}}$_2$\emph{\textbf{)}} is \emph{KT-admissible} with some coefficient $\lambda_0 \geq0$. Moreover, \emph{\textbf{(}}\textbf{\emph{P}}$_2$\emph{\textbf{)}} is solvable with an optimal solution $(\bar\xi,\bar u_1(\cdot))$ iff there exist a  $7$-tuple $(\lambda; g,\bar Y,\bar X,\bar Z ,h,q)$ satisfying the following BFSDEs
 \begin{equation*}\label{e1L special case II}\textbf{\emph{(BFSDE-3)}}:\left\{\begin{aligned}	
&dg=-\Big[A^\top  g-Q_1\bar X-Q_2h\Big]ds-\Big[C^\top g-S_1\bar Z-S_2q\Big]dW(s),\\
&d\bar Y=\Big[-A^\top \bar Y+Q_2\bar  X\Big]ds+\Big[-C^\top \bar Y+S_2\bar  Z\Big]dW(s),\\
&d\bar X=\Big[A\bar  X+B_1(R^1_{11})^{-1}B_1^\top g+B_2(R^2_{22})^{-1}B_2^\top\bar Y+C \bar Z\Big]ds+ \bar ZdW(s),\\
&dh=\Big[ Ah+B_2(R^2_{22})^{-1}B^\top_2 g+ Cq\Big]ds+qdW(s),\\
& g(0)=H_1\bar X(0)+ H_2 h(0),\quad \bar Y(0)=H_2\bar  X(0),\\
& \bar  X(T)=G_1^{-1}(-g(T)+\lambda\alpha),\quad h(T)=0,\\
&\lambda\Big(\beta-\mathbb E\Big\langle\alpha, G_1^{-1}(-g(T)+\lambda\alpha)\Big\rangle\Big)=0, \quad \lambda \geq0,  \quad \beta \leq \mathbb E\Big\langle\alpha, G_1^{-1}(-g(T)+\lambda\alpha)\Big\rangle.
\end{aligned}\right.\end{equation*}
 In this case, $\lambda$ is a  \emph{KT-coefficient} of \emph{\textbf{(}}\textbf{\emph{P}}$_2$\emph{\textbf{)}}, and an optimal solution to problem \emph{\textbf{(}}\textbf{\emph{P}}$_2$\emph{\textbf{)}} is
\begin{equation}\label{optimal-case 2}(\bar\xi,\bar u_1(\cdot))=\Big(G_1^{-1}(-g(T)+\lambda\alpha),(R^1_{11}(\cdot))^{-1}B_1^\top(\cdot) g(\cdot)\Big).
\end{equation}
\end{corollary}For Corollaries \ref{th general terminal optimal} and \ref{special case-affine constrait}, it follows that {(BFSDE-2)} and {(BFSDE-3)} play some key roles in determining the optimal solution. Specifically, {(BFSDE-2)} is a \emph{nonlinear} (because of the projection operator) fully-coupled BFSDEs; {(BFSDE-3)} is a linear but constrained (because of {(KKT)} condition) fully-coupled BFSDEs. Both are non-standard in BFSDEs theory. Thus, it remains a challenge to show the global solvability of them, together with {(SRE-1)}, {(SRE-2)}. To this end, we study the wellposedness  (existence, uniqueness) of {(BFSDE-2)}, {(BFSDE-3)}, Riccati equations in Sections \ref{Solv-(FBSDE-2)}, \ref{Solv-(FBSDE-3)} and \ref{Riccati equation}, respectively.

\section{Existence and uniqueness of BFSDEs and Riccati equations}\label{wellposedness of FBSDEs}

\subsection{Discounting method}
In this subsection, we will use the discounting method (see \cite{PT1999}) to study the wellposedness of BFSDEs.
To begin with, we first give some results for general nonlinear mean-field BFSDEs:

\begin{equation}\label{eA2 1}\left\{\begin{aligned}
&dY(s)=b(s,Y(s),X(s),Z(s),\mathbb EZ(s))ds+\sigma(s,Y(s),X(s),Z(s))dW(s),\\
&-dX(s)=f(s,Y(s),X(s),Z(s),\mathbb EZ(s))ds-ZdW(s),\\
&Y(0)=h(X(0)),\quad X(T)=g(Y(T),\mathbb EY(T)).
\end{aligned}\right.\end{equation}

Accordingly, the following assumptions are imposed:\\
\textbf{(H5)} There exist $\rho_1,\rho_2\in\mathbb R$ and positive constants $ k_i,i=1,2,\cdots,10$ such that for all $s\in[0,T]$,
$y,
y_1,y_2,\bar y_1,\bar y_2\in\mathbb R^{n_1}$, $x,x_1,x_2,z,z_1,z_2,\bar z_1,\bar z_2\in\mathbb R^{n_2}$ a.s.,
\begin{description}
\item[(i)] $\langle b(s,y_1,x,z,\bar z)-b(s,y_2,x,z,\bar z), y_1-y_2\rangle\leq \rho_1 |y_1-y_2|^2$, \\$|b(s,y,x_1,z_1,\bar z_1)-b(s,y,x_2,z_2,\bar z_2)|\leq k_1 |x_1-x_2|+k_2|z_1-z_2|+k_3|\bar z_1-\bar z_2|$,
\item[(ii)] $\langle f(s,y,x_1,z,\bar z)-f(s,y,x_2,z,\bar z), x_1-x_2\rangle\leq \rho_2 |x_1-x_2|^2$,\\ $|f(s,y_1,x,z_1,\bar z_1)-f(s,y_2,x,z_2,\bar z_2)|\leq k_4 | y_1-y_2|+k_5|z_1-z_2|+k_6|\bar z_1-\bar z_2|$,
\item[(iii)] $|\sigma(s,y_1,x_1,z_1)-\sigma(s,y_2,x_2,z_2)|^2\leq k^2_7 | y_1-y_2|^2+k^2_8|x_1-x_2|^2+k^2_9|z_1-z_2|^2$,
\item[(iv)] $|h(x_1)-h(x_2)|\leq k_{10}|x_1-x_2|$, $|g(y_1,\bar y_1)-g(y_2,\bar y_2)|\leq k_{11}|y_1-y_2|+k_{12}|\bar y_1-\bar y_2|,$
\item[(v)] $\mathbb E\left\{|h(0)|^2+|g(0,0)|^2+\int^T_0(|b(s,0,0,0,0)|^2+|\sigma(s,0,0,0)|^2+|f(s,0,0,0,0)|^2)ds\right\}<\infty$.
\end{description}

Now we present the main result of this subsection on wellposedness of mean-field BFSDEs \eqref{eA2 1}. Its proof is in the appendix.
\begin{theorem}\label{th A 2}
Under \textbf{\emph{(H5)}}, there exists a $\delta_1>0$, which depends on $\rho_1,\rho_2,T,k_i,i=5,6,7$, such that when $k_i\in [0, \delta_1)$, $i=1,2,3,4,8,9,10$, there exists a unique adapted solution $(Y(\cdot),X(\cdot),Z(\cdot))\in L^2_{\mathbb F}(0,T;\mathbb R^{n_1})\times L^2_{\mathbb F}(0,T;\mathbb R^{n_2})\times L^2_{\mathbb F}(0,T;\mathbb R^{n_2})$
to mean-field BFSDEs \eqref{eA2 1}. Further, if $2(\rho_1 + \rho_2) < -2k^2_5-2k^2_6-k_7^2$, there exists a $\delta_2 > 0$,
which depends on $\rho_1,\rho_2,k_i,i=5,6,7$, and is independent of $T$, such that when $k_i\in [0, \delta_1)$, $i=1,2,3,4,8,9,10$, there exists a unique adapted solution $(Y(\cdot),X(\cdot),Z(\cdot))\in L^2_{\mathbb F}(0,T;\mathbb R^{n_1})\times L^2_{\mathbb F}(0,T;\mathbb R^{n_2})\times L^2_{\mathbb F}(0,T;\mathbb R^{n_2})$ to mean-field BFSDEs \eqref{eA2 1}.
\end{theorem}

\subsection{Solvability of (BFSDE-2)}\label{Solv-(FBSDE-2)}

In order to apply Theorem \ref{th A 2}, denote $\mathbb Y=(g^\top,\bar Y^\top)^\top,\mathbb X=(\bar X^\top,h^\top)^\top,\mathbb Z=(\bar Z^\top,q^\top)^\top$.  Rewrite {(BFSDE-2)} as the following $2n\times 2n$-BFSDEs:
\begin{equation*}\textbf{(BFSDE-2$^\prime$)}: \left\{\begin{aligned}	
&d\mathbb Y=\left[-\left(\begin{matrix}A&0\\0&A\end{matrix}\right)^\top  \mathbb Y+\left(\begin{matrix}Q_1&Q_2\\Q_2&0\end{matrix}\right)\mathbb X\right]ds+\left[-\left(\begin{matrix}C&0\\0&C\end{matrix}\right)^\top \mathbb Y+\left(\begin{matrix}S_1&S_2\\S_2&0\end{matrix}\right)\mathbb Z\right]dW(s),\\
&d\mathbb X=\left[\left(\begin{matrix}B_1(R^1_{11})^{-1}B_1^\top& B_2(R^2_{22})^{-1}B_2^\top \\
B_2(R^2_{22})^{-1}B^\top_2 &0\end{matrix}\right)\mathbb Y\right.\left.+\left(\begin{matrix}A&0\\0&A\end{matrix}\right)\mathbb  X+\left(\begin{matrix}C&0\\0&C\end{matrix}\right)\mathbb Z\right]ds+\mathbb ZdW(s),\\
& \mathbb Y(0)=\left(\begin{matrix}H_1&H_2 \\H_2&0\end{matrix}\right)\mathbb X(0),\quad  \mathbb  X(T)=\mathbf{Proj}_{\mathcal K}\left[\left(\begin{matrix}-G^{-1}_1& 0 \\0&0\end{matrix}\right)\mathbb Y(T)\right],
\end{aligned}\right.\end{equation*}
where $\mathbf{Proj}_{\mathcal K}(\cdot)=\left(\begin{matrix}\text{\emph{Proj}}_{\mathcal K}(\cdot)\\ \text{\emph{Proj}}_{\mathbb R^n}(\cdot)\end{matrix}\right)$.
Now let $\rho^*=\text{esssup}_{0\leq s\leq T}\text{esssup}_{\omega\in\Omega}\Lambda_{\max}(-\frac{1}{2}(A(s)+A(s)^\top))$, where $\Lambda_{\max}(M)$ is the largest eigenvalue of the matrix $M$. Comparing {(BFSDE-2$^\prime$)} with \eqref{eA2 1}, by the Proposition 4.2 in \cite{HHL2017}, we can check that the parameters of (H5) can be chosen as follows:
\begin{equation}\nonumber\begin{aligned}
&\rho_1=\rho_2=\rho^*, k_2=k_3=k_6=k_8=k_{12}=0, k_1=\left|\left|\left(\begin{matrix}Q_1&Q_2\\Q_2&0\end{matrix}\right)\right|\right|,
k_5=\sqrt{2}\left|\left|C\right|\right|,k_7=2\left|\left|C\right|\right|, \\ &k_9=\sqrt{2}\left|\left|\left(\begin{matrix}S_1&S_2\\S_2&0\end{matrix}\right)\right|\right| ,k_4=\left|\left|\left(\begin{matrix}B_1(R^1_{11})^{-1}B_1^\top&B_2(R^2_{22})^{-1}B_2^\top \\
B_2(R^2_{22})^{-1}B^\top_2 &0\end{matrix}\right)\right|\right|,k_{10}=\left|\left|\left(\begin{matrix}H_1&H_2 \\H_2&0\end{matrix}\right)\right|\right|,\\
&
k_{11}=\left|\left|\left(\begin{matrix}-G^{-1}_1& 0 \\0&0\end{matrix}\right)\right|\right|,
\end{aligned}\end{equation}
where for $M(\cdot)\in L^\infty_{\mathbb F}(0,T;\mathbb R^{n\times n})$, $\|M(\cdot)\|\triangleq\mathop{esssup}\limits_{0\leq s\leq T}\mathop{esssup}\limits_{\omega\in\Omega}\|M(s)\|$.
By Theorem \ref{th A 2}, we have
\begin{theorem}\label{wellposeness of {e1LK FBSDE}}
Suppose that
$\rho^*<-4\left|\left|C(\cdot)\right|\right|^2.$ There exists a $\delta_1>0$, which depends on $\rho^*,k_i,i=5,7$, such that when $k_1, k_4, k_9,k_{10} \in[0, \delta_1)$, there exists a unique adapted solution to {\emph{(BFSDE-2$^\prime$)}}.
\end{theorem}
\begin{remark}
By the definition of $\rho^*$, Theorem \ref{wellposeness of {e1LK FBSDE}} establishes the existence and uniqueness of {(BFSDE-2)}  under some condition on the matrix $A(\cdot)$.
\end{remark}
Combining Corollary \ref{th general terminal optimal} and Theorem \ref{wellposeness of {e1LK FBSDE}}, we have the following result.
\begin{theorem}
Let \textbf{\emph{(H1)}}-\textbf{\emph{(H4)}} and \emph{\textbf{(}}\textbf{\emph{P}}$_1$\emph{\textbf{)}} is convex. Suppose that
$\rho^*<-4\left|\left|C(\cdot)\right|\right|^2$ and there exists a $\delta_1>0$  depending on $\rho^*,k_i,i=5,7$, such that $k_1, k_4, k_9,k_{10} \in[0, \delta_1)$. Then \emph{\textbf{(}}\textbf{\emph{P}}$_1$\emph{\textbf{)}} admits a unique optimal control given by \eqref{optimal-case 1}
where  $(\bar Y,g,\bar X,\bar Z ,h,q)$ is the unique solution of {\emph{(BFSDE-2)}}.
\end{theorem}

\subsection{Wellposedness of \eqref{Hamiltonian system-follower}}\label{solva-Hamil-follwer}

In this subsection, we will give a direct result on wellposedness of \eqref{Hamiltonian system-follower} by Theorem \ref{th A 2}.
Let $\rho^*=\text{esssup}_{0\leq s\leq T}\text{esssup}_{\omega\in\Omega}\Lambda_{\max}(-\frac{1}{2}(A(s)+A(s)^\top))$,
\begin{equation}\nonumber\begin{aligned}
&\rho_1=\rho_2=\rho^*, k_2=k_3=k_6=k_8=k_{11}=k_{12}=0, k_1=\left|\left|Q_2\right|\right|,
k_5=\left|\left|C\right|\right|,k_7=\sqrt{2}\left|\left|C\right|\right|, \\
&k_9=\sqrt{2}\left|\left|S_2\right|\right|,k_{10}=\left|\left|H_2\right|\right|,k_4=\left|\left|B_2(R^2_{22})^{-1}B_2^\top \right|\right|. %k_{10}=\frac{|\alpha|^2}{|\mathbb E\left\langle \alpha,G^{-1}\alpha\right\rangle|}\left|\left|\left(\begin{matrix}G^{-1}_1& 0 \\0&0\end{matrix}\right)\right|\right|^2.
\end{aligned}\end{equation}
\begin{theorem}\label{solva-3.8}
Suppose that
$\rho^*<-2\left|\left|C(\cdot)\right|\right|^2.$ There exists a $\delta_1>0$ depending on $\rho^*,k_i,i=5,7$, such that when $k_1, k_4, k_9,k_{10} \in[0, \delta_1)$, there exists a unique adapted solution to \eqref{Hamiltonian system-follower}.
\end{theorem}

\subsection{Solvability of (BFSDE-3)}\label{Solv-(FBSDE-3)}Now, we consider the solvability of {(BFSDE-3)} which is a standard fully-coupled BFSDEs but combining with the {(KKT)} qualification condition. Hence, it becomes \emph{non-standard} BFSDEs with constraint on its terminal expectation via Lagrange variable $\lambda$ involved. In this sense, we may call it \emph{terminal-mean-constrained \emph{BFSDEs}}. To our knowledge, such class of BFSDEs has not been well studied and this sections aims some essential endeavor to it.
To this end, we may first rewrite (BFSDE-3) as the following $2n\times 2n$-BFSDEs (with same notations to {(BFSDE-2)}):
\begin{equation*}\textbf{(BFSDE-3$^\prime$)}:
\left\{\begin{aligned}
&d\mathbb Y=-\left[\left(\begin{matrix}A&0\\0&A\end{matrix}\right)^\top  \mathbb Y-\left(\begin{matrix}Q_1&Q_2\\Q_2&0\end{matrix}\right)\mathbb X\right]dt-\left[\left(\begin{matrix}C&0\\0&C\end{matrix}\right)^\top \mathbb Y-\left(\begin{matrix}S_1&S_2\\S_2&0\end{matrix}\right)\mathbb Z\right]dW(s),\\
&d\mathbb X=\left[\left(\begin{matrix}B_1(R^1_{11})^{-1}B_1^\top& B_2(R^2_{22})^{-1}B_2^\top \\
B_2{R^2_{22}}^{-1}B^\top_2 &0\end{matrix}\right)\mathbb Y+\left(\begin{matrix}A&0\\0&A\end{matrix}\right)\mathbb  X+\left(\begin{matrix}C&0\\0&C\end{matrix}\right)\mathbb Z\right]ds+\mathbb ZdW(s),\\
& \mathbb Y(0)=\left(\begin{matrix}H_1&H_2 \\H_2&0\end{matrix}\right)\mathbb X(0),\quad  \mathbb  X(T)=\left(\begin{matrix}G_1^{-1}& 0 \\0&0\end{matrix}\right)\left(-\mathbb Y(T)+\lambda\left(\begin{matrix}\alpha \\ 0\end{matrix}\right)\right),
\\& \lambda\left(\beta-\mathbb E\left\langle\left(\begin{matrix}\alpha\\0\end{matrix}\right), \left(\begin{matrix}G_1^{-1}&0\\0&0\end{matrix}\right)\left(-\mathbb Y(T)+\lambda\left(\begin{matrix}\alpha\\0\end{matrix}\right)\right)\right\rangle\right)=0,\quad \lambda\geq0,\\
&\beta-\mathbb E\left\langle\left(\begin{matrix}\alpha\\0\end{matrix}\right), \left(\begin{matrix}G_1^{-1}&0\\0&0\end{matrix}\right)\left(-\mathbb Y(T)+\lambda\left(\begin{matrix}\alpha\\0\end{matrix}\right)\right)\right\rangle\leq0.
\end{aligned}\right.\end{equation*}
By the first slackness condition of (KKT) system, there arise two cases with $\lambda=0$ or $\lambda=\Big(\beta+\mathbb E\left\langle\left(\begin{matrix}\alpha \\ 0\end{matrix}\right),\left(\begin{matrix}G_1^{-1}& 0 \\0&0\end{matrix}\right)\mathbb Y(T)\right\rangle\Big)\Big(\mathbb E\langle \alpha,G_1^{-1}\alpha\rangle\Big)^{-1}.$ We have the following more detailed analysis along these two cases.
\subsubsection{Multiplier $\lambda=0$}\label{zero}
In this case, (BFSDE-3$^\prime$) takes the following form:
\begin{equation}\label{eq56}
\left\{\begin{aligned}
&d\mathbb Y=-\left[\left(\begin{matrix}A&0\\0&A\end{matrix}\right)^\top  \mathbb Y-\left(\begin{matrix}Q_1&Q_2\\Q_2&0\end{matrix}\right)\mathbb X\right]dt-\left[\left(\begin{matrix}C&0\\0&C\end{matrix}\right)^\top \mathbb Y-\left(\begin{matrix}S_1&S_2\\S_2&0\end{matrix}\right)\mathbb Z\right]dW(t),\\
&d\mathbb X=\left[\left(\begin{matrix}B_1(R^1_{11})^{-1}B_1^\top& B_2(R^2_{22})^{-1}B_2^\top \\
B_2{R^2_{22}}^{-1}B^\top_2 &0\end{matrix}\right)\mathbb Y+\left(\begin{matrix}A&0\\0&A\end{matrix}\right)\mathbb  X+\left(\begin{matrix}C&0\\0&C\end{matrix}\right)\mathbb Z\right]ds+\mathbb ZdW(s),\\
& \mathbb Y(0)=\left(\begin{matrix}H_1&H_2 \\H_2&0\end{matrix}\right)\mathbb X(0),\quad  \mathbb  X(T)=-\left(\begin{matrix}G_1^{-1}& 0 \\0&0\end{matrix}\right)\mathbb Y(T),\\
&\beta+ \mathbb E\Big\langle\left(\begin{matrix}\alpha\\0\end{matrix}\right), \left(\begin{matrix}G_1^{-1}&0\\0&0\end{matrix}\right)\mathbb Y(T)\Big\rangle\leq 0.\qquad \textbf{primal constraint in (KKT)}
\end{aligned}\right.\end{equation}
We will use Riccati decoupling method to study the wellposedness of \eqref{eq56}. Define  $\widetilde{\mathbb Y}=\mathbb Y-\left(\begin{matrix}H_1&H_2 \\H_2&0\end{matrix}\right)\mathbb X$,
therefore, $\widetilde{\mathbb Y}(0)=0$ and
\begin{equation}\nonumber
\begin{aligned}
\mathbb X(T)=&-\left(\begin{matrix}G_1^{-1}& 0 \\0&0\end{matrix}\right)\mathbb Y(T)
=-\left(\begin{matrix}G_1^{-1}& 0 \\0&0\end{matrix}\right)\widetilde{\mathbb Y}(T)-
\left(\begin{matrix}G_1^{-1}H_1& G_1^{-1}H_2 \\0&0\end{matrix}\right)\mathbb X(T).\\
\end{aligned}
\end{equation}
If $\det\left[I+G_1^{-1}H_1\right]\neq0$, then
the matrix $\left(\begin{matrix} I+G_1^{-1}H_1& G_1^{-1}H_2\\ 0& I\end{matrix}\right)$ is invertible,
and consequently, $$\mathbb X(T)=\widetilde{ G} \widetilde {\mathbb Y}(T),$$
where
\begin{equation}
\nonumber
\begin{aligned}
\widetilde { G}=-\left(\begin{matrix} I+G_1^{-1}H_1& G_1^{-1}H_2\\ 0& I\end{matrix}\right)^{-1}\left(\begin{matrix}G_1^{-1}& 0 \\0&0\end{matrix}\right)
=-\left(\begin{matrix}
(I+G_1^{-1}H_1)^{-1}G_1^{-1}&0\\0&0
\end{matrix}\right).
\end{aligned}
\end{equation}
Therefore, if $\det\left[I+G_1^{-1}H_1\right]\neq0$, after some manipulations, we have
\begin{equation}\label{eq58}
\left\{\begin{aligned}
d\widetilde{\mathbb Y}=&-\left[\widetilde{ A}\widetilde{ \mathbb Y}+\widetilde{ B}\mathbb X+\widetilde{ C} \mathbb Z\right]dt-\left[\widetilde{ A}_1\widetilde{\mathbb Y}+\widetilde{ B}_1\mathbb X+\widetilde{ C}_1 \mathbb Z\right]dW(t),\\
d\mathbb X=&\left[\widehat{ A}\widetilde{\mathbb Y}+\widehat{ B}\mathbb X+\widehat{ C} \mathbb Z\right]dt+\mathbb ZdW(t),\\
\widetilde {\mathbb Y}(0)=&0,\quad  \mathbb  X(T)=\widetilde{ G} \widetilde {\mathbb Y}(T),
\end{aligned}\right.\end{equation}
where
\begin{equation}\label{new notation-1}
\begin{aligned}
&\widetilde{ A}=\left(\begin{matrix}A^{\top}+H_1B_1(R^1_{11})^{-1}B_1^\top+H_2B_2(R^2_{22})^{-1}B_2^\top &
H_1B_2(R^2_{22})^{-1}B_2^\top\\
H_2 B_1(R^1_{11})^{-1}B_1^\top&
A^{\top}+H_2B_2(R^2_{22})^{-1}B_2^\top\end{matrix}\right),\\
&\widetilde{ B}=\left(\begin{matrix}\widetilde{ B}_{11}&\widetilde{ B}_{12}\\ \widetilde{ B}_{21}&\widetilde{ B}_{22}\end{matrix}\right),\\
&\widetilde{ B}_{11}=-Q_1+H_1A+A^{\top}H_1+H_1B_1(R^1_{11})^{-1}B_1^\top H_1+H_2B_2(R^2_{22})^{-1}B_2^{\top}H_1
+H_1B_2(R^2_{22})^{-1}B_2^{\top}H_2,\\
&\widetilde{ B}_{12}=-Q_2+H_2A+A^{\top}H_2+H_1B_1(R^1_{11})^{-1}B_1^\top H_2+H_2B_2(R^2_{22})^{-1}B_2^{\top}H_2,\\
&\widetilde{ B}_{21}=-Q_2+H_2A+A^{\top}H_2+H_2B_1(R^1_{11})^{-1}B_1^\top H_1+H_2B_2(R^2_{22})^{-1}B_2^{\top}H_2,\\
&\widetilde{ B}_{22}=H_2B_1(R^1_{11})^{-1}B_1^\top H_2,\\
&\widetilde{ C}=\left(\begin{matrix}H_1C & H_2C\\ H_2C&0\end{matrix}\right),\quad
\widetilde{ A}_1=\left(\begin{matrix}C & 0\\ 0&C\end{matrix}\right)^{\top},\quad
\widetilde{ B}_1=\left(\begin{matrix}C^{\top}H_1 & C^{\top}H_2\\ C^{\top}H_2&0\end{matrix}\right),\\
&
\widetilde{ C}_1=-\left(\begin{matrix}S_1-H_1 & S_2-H_2\\ S_2-H_2&0\end{matrix}\right),\quad\widehat{ A}=\left(\begin{matrix}B_1(R^1_{11})^{-1}B_1^\top &
B_2(R^2_{22})^{-1}B_2^\top \\
B_2(R^2_{22})^{-1}B_2^\top&
0\end{matrix}\right),\\
& \widehat{ B}=\left(\begin{matrix}A+B_1(R^1_{11})^{-1}B_1^\top H_1+B_2(R_{22}^2)^{-1}B_2^{\top}H_2&
B_1(R^1_{11})^{-1}B_1^\top H_2 \\
B_2(R^2_{22})^{-1}B_2^\top H_1&
A+B_2(R_{22}^2)^{-1}B_2^{\top}H_2\end{matrix}\right),\quad\widehat{
 C}=\left(\begin{matrix}C & 0\\ 0&C\end{matrix}\right).\\
\end{aligned}
\end{equation}
Note that  $\widehat{A}$, $\widetilde{B}$ are symmetric and $\widehat{ B}=\widetilde{A}^\top,\quad\widehat{C}=\widetilde{A}_1^\top,\quad \widetilde{C}=\widetilde{B}_1^\top.$
\begin{remark}
	Since $G_1^{-1}$ is symmetric,
 it follows from \emph{\cite{Y02}} that $(I+G_1^{-1}H_1)^{-1}G_1^{-1}$ is symmetric, i.e., $\widetilde{ G}$ is symmetric.
\end{remark}
Suppose the following linear relation holds true,
\begin{equation}\label{decouple-X}
\mathbb X(s)=\widetilde P(s)\widetilde{\mathbb Y}(s)+\widetilde p(s),\quad s\in[0,T],\quad a.s.
\end{equation}
If $\det\left[I+G_1^{-1}H_1\right]\neq0$, \eqref{eq56} is solvable if the following stochastic Riccati equation and BSDE are solvable
\begin{equation}\label{general form Riccati equation of special case II}
\left\{
\begin{aligned}
&d\widetilde P=\bigg\{\widehat{ A}+\widehat{ B}\widetilde P+\widetilde P\widetilde{ A}+\widetilde P\widetilde{ B}\widetilde P+\widetilde\Lambda\left(\widetilde{ A}_1+\widetilde{ B}_1\widetilde P\right)+\left(\widehat{ C}+\widetilde P\widetilde{ C}+\widetilde \Lambda\widetilde{ C}_1\right)\left(I+\widetilde P\widetilde{ C}_1\right)^{-1}\\
&\qquad \quad \left[\widetilde\Lambda-\widetilde P\left(\widetilde{ A}_1+\widetilde{ B}_1\widetilde P\right)\right]
\bigg\}ds+\widetilde\Lambda d W(s),\\
&\widetilde P(T)=\ \widetilde{ G},\\
&\det\left[I+\widetilde P\widetilde{ C}_1\right]\neq0,
\end{aligned}
\right.
\end{equation}
and
\begin{equation}
\label{general BSDE of special case II}
\left\{
\begin{aligned}
d\widetilde p=&\bigg\{\Big[\widehat{ B}+\widetilde P\widetilde{ B}+\widetilde\Lambda\widetilde{ B}_1-(\widehat{ C}+\widetilde P\widetilde{ C}+\widetilde\Lambda\widetilde{ C}_1)(I+\widetilde P\widetilde{ C}_1)^{-1}\widetilde P\widetilde{ B}_1\Big]\widetilde p\\
&\quad+(\widehat{ C}+\widetilde P\widetilde{ C}+\widetilde\Lambda\widetilde{ C}_1)(I+\widetilde P\widetilde{ C}_1)^{-1}\widetilde q\bigg\}ds+\widetilde qdW(s),\\
\widetilde p(T)=&0,
\end{aligned}
\right.
\end{equation}such that (KKT) in \eqref{eq56} is satisfied.
It is easy to check that
\begin{equation}\label{decouple-Z}
Z=(I+\widetilde P\widetilde C_1)^{-1}[(\widetilde\Lambda-\widetilde P\widetilde A_1-\widetilde P\widetilde B_1\widetilde P)\widetilde{\b Y}-\widetilde P\widetilde B_1\widetilde p+\widetilde q].
\end{equation}
Next we introduce another assumption under which we will obtain some new form of \eqref{general form Riccati equation of special case II} and \eqref{general BSDE of special case II},
\begin{description}
	\item[(H6)] $\det[S_2-H_2]\neq0$.

\end{description}
Under (H6), we have $\det[\widetilde C_1]\neq0$, hence
\begin{equation}\nonumber
\begin{aligned}
&\left(\widehat{ C}+\widetilde P\widetilde{ C}+\widetilde \Lambda\widetilde{ C}_1\right)\left(I+\widetilde P\widetilde{ C}_1\right)^{-1}\left[\widetilde\Lambda-\widetilde P\left(\widetilde{ A}_1+\widetilde{ B}_1\widetilde P\right)\right]\\
=&\left(\widetilde \Lambda+\widehat{ C}\widetilde{ C}_1^{-1}+\widetilde P\widetilde{ C}\widetilde{ C}_1^{-1}\right)\left(\widetilde{ C}_1^{-1}+\widetilde P\right)^{-1}\left[\widetilde\Lambda-\widetilde P\left(\widetilde{ A}_1+\widetilde{ B}_1\widetilde P\right)\right]\\
=&\left(\widetilde \Lambda+\left(\begin{matrix}C&0\\0&C\end{matrix}\right)\widetilde{ C}_1^{-1}+\widetilde P\left(\begin{matrix}H_1&H_2\\H_2&0\end{matrix}\right)\left(\begin{matrix}C&0\\0&C\end{matrix}\right)\widetilde{ C}_1^{-1}\right)(\widetilde{ C}_1^{-1}+\widetilde P)^{-1}\\&\left(\widetilde\Lambda-\widetilde P\left(\begin{matrix}C & 0\\ 0&C\end{matrix}\right)^{\top}-\widetilde P\left(\begin{matrix}C & 0\\ 0&C\end{matrix}\right)^{\top}\left(\begin{matrix}H_1 & H_2\\ H_2&0\end{matrix}\right)\widetilde P\right)\\
=&\left(\widetilde\Lambda+\left(I+\widetilde P\left(\begin{matrix}H_1&H_2\\H_2&0\end{matrix}\right)\right)\left(\begin{matrix}C&0\\0&C\end{matrix}\right)\widetilde{ C}_1^{-1}\right)\left(\widetilde{ C}_1^{-1}+\widetilde P\right)^{-1}
\left(\widetilde\Lambda-\widetilde P\left(\begin{matrix}C&0\\0&C\end{matrix}\right)^\top\left(I+\left(\begin{matrix}H_1&H_2\\H_2&0\end{matrix}\right)\widetilde P\right)\right)\\
=&
\left(\widetilde\Lambda-\left(I+\widetilde P\left(\begin{matrix}H_1&H_2\\H_2&0\end{matrix}\right)\right)\left(\begin{matrix}C&0\\0&C\end{matrix}\right)\widetilde P\right)\left(\widetilde{ C}_1^{-1}+\widetilde P\right)^{-1}
\left(\widetilde\Lambda-\widetilde P\left(\begin{matrix}C&0\\0&C\end{matrix}\right)^\top\left(I+\left(\begin{matrix}H_1&H_2\\H_2&0\end{matrix}\right)\widetilde P\right)\right)\\
&+\left(I+\widetilde P\left(\begin{matrix}H_1&H_2\\H_2&0\end{matrix}\right)\right)\left(\begin{matrix}C&0\\0&C\end{matrix}\right)\left(\widetilde\Lambda-\widetilde P\left(\begin{matrix}C&0\\0&C\end{matrix}\right)^\top\left(I+\left(\begin{matrix}H_1&H_2\\H_2&0\end{matrix}\right)\widetilde P\right)\right)\\
=&\left(\widetilde \Lambda-(\widehat{ C}+\widetilde P\widetilde{ C})\widetilde P\right)\left(\widetilde{ C}_1^{-1}+\widetilde P\right)^{-1}
\left(\widetilde \Lambda-\widetilde P(\widehat{ C}^\top+\widetilde{ C}^\top\widetilde P)\right)+(\widehat{ C}+\widetilde P\widetilde{ C})\left(\widetilde\Lambda-\widetilde P(\widehat{ C}^\top+\widetilde{ C}^\top\widetilde P)\right)\\
=&\left(\widetilde \Lambda-(\widehat{ C}+\widetilde P\widetilde{ C})\widetilde P\right)\left(\widetilde{ C}_1^{-1}+\widetilde P\right)^{-1}
\left(\widetilde \Lambda-\widetilde P(\widehat{ C}^\top+\widetilde{ C}^\top\widetilde P)\right)+\widehat{ C}\widetilde\Lambda+\widetilde P\widetilde{ C}\widetilde\Lambda-(\widehat{ C}+\widetilde P\widetilde{ C})\widetilde P(\widehat{ C}^\top+\widetilde{ C}^\top\widetilde P).
\end{aligned}
\end{equation}
Therefore,
 \eqref{general form Riccati equation of special case II} and \eqref{general BSDE of special case II} take the following forms:
\begin{equation}\label{symmetric Riccati equation of special case II}
\left\{
\begin{aligned}
&d\widetilde P=\bigg\{\widehat{ A}+\widehat{ B}\widetilde P+\widetilde P\widehat{ B}^\top+\widetilde P\widetilde{ B}\widetilde P+\widetilde\Lambda\left(\widehat{ C}^\top+\widetilde{ C}^\top\widetilde P\right)+\left(\widehat{ C}+\widetilde P\widetilde{ C}\right)\widetilde\Lambda-\left(\widehat{ C}+\widetilde P\widetilde{ C}\right)\widetilde P\left(\widehat{ C}^\top+\widetilde{ C}^\top\widetilde P\right)\\
&\qquad\quad+\left(\widetilde \Lambda-\left(\widehat{ C}+\widetilde P\widetilde{ C}\right)\widetilde P\right)\left(\widetilde{ C}_1^{-1}+\widetilde P\right)^{-1}
\left(\widetilde \Lambda-\widetilde P\left(\widehat{ C}^\top+\widetilde{ C}^\top\widetilde P\right)\right)
\bigg\}ds+\widetilde\Lambda d W(s),\\
&\widetilde P(T)=\widetilde{ G},\\
&\det\left[I+\widetilde P\widetilde{ C}_1\right]\neq0,
\end{aligned}
\right.
\end{equation}
and
\begin{equation}
\label{BSDE of special case II}
\left\{
\begin{aligned}
d\widetilde p=&\bigg\{\Big[\widehat{ B}+\widetilde P\widetilde{ B}+\widetilde\Lambda\widetilde{ B}_1-(\widehat{ C}\widetilde{ C}_1 ^{-1}+\widetilde P\widetilde{ C}\widetilde{ C}_1 ^{-1}+\widetilde\Lambda)(\widetilde{ C}_1 ^{-1}+\widetilde P)^{-1}\widetilde P\widetilde{ B}_1\Big]\widetilde p\\
&\quad+(\widehat{ C}\widetilde{ C}_1 ^{-1}+\widetilde P\widetilde{ C}\widetilde{ C}_1 ^{-1}+\widetilde\Lambda)(\widetilde{ C}_1 ^{-1}+\widetilde P)^{-1}\widetilde q\bigg\}ds+\widetilde qdW(s),\\
\widetilde p(T)=&0.
\end{aligned}
\right.
\end{equation}
Finally, plugging \eqref{decouple-X} and \eqref{decouple-Z} into \eqref{eq58}, we have
\begin{equation*}\begin{aligned}
d\widetilde{\mathbb Y}=&-\left[\widetilde{\mathbb A}\widetilde{ \mathbb Y}+\widetilde b\right]dt-\left[\widetilde{\mathbb A}_1\widetilde{\mathbb Y}+\widetilde{ \sigma}\right]dW(t),\quad \widetilde {\mathbb Y}(0)=0,
\end{aligned}\end{equation*}
where
$$\widetilde{\mathbb A}=\widetilde A+\widetilde B\widetilde P+\w C(I+\w P\w C_1)^{-1}(\w\Lambda-\w P\w A_1-\w P\w B_1\w P),$$
$$\w b=\w B+\w P\w p-\w C(I+\w P\w C_1)^{-1}\w P\w B_1\w p+\w C(I+\w P\w C_1)^{-1}\w q,$$
$$\widetilde{\mathbb A}_1=\widetilde A_1+\widetilde B_1\widetilde P+\w C_1(I+\w P\w C_1)^{-1}(\w\Lambda-\w P\w A_1-\w P\w B_1\w P),$$
$$\w \sigma=\w B_1+\w P\w p-\w C_1(I+\w P\w C_1)^{-1}\w P\w B_1\w p+\w C_1(I+\w P\w C_1)^{-1}\w q.$$
Therefore,
\begin{equation*}
\w{\mathbb Y}(t)=\Phi(t)\int_0^t\Phi(s)^{-1}[\w b(s)-\w{\mathbb A}_1(s)\w\sigma(s)]ds+\Phi(t)\int_0^t\Phi(s)^{-1}\w\sigma(s)dW(s),\quad t\in[0,T],
\end{equation*}
where
$$d\Phi(t)=\w{\mathbb A}(t)\Phi(t)dt+\w{\mathbb A}_1\Phi(t)dW(t),\quad \Phi(0)=I.$$
Hence,
$$\mathbb Y(T)=\widetilde{\mathbb Y}(T)-\left(\begin{matrix}H_1&H_2 \\H_2&0\end{matrix}\right)\mathbb X(T)=\w{\mathbb Y}(T)-\left(\begin{matrix}H_1&H_2 \\H_2&0\end{matrix}\right)\widetilde{ G} \widetilde {\mathbb Y}(T)=\Big[I-\left(\begin{matrix}H_1&H_2 \\H_2&0\end{matrix}\right)\widetilde{ G}\Big]\widetilde {\mathbb Y}(T),$$ and the (KKT) condition becomes
\begin{equation}\begin{aligned}\label{KTT-2}
&\beta+\Big\langle\left(\begin{matrix}\alpha\\0\end{matrix}\right),\mathbb E\left(\begin{matrix}G_1^{-1}&0\\0&0\end{matrix}\right)\Big[I-\left(\begin{matrix}H_1&H_2 \\H_2&0\end{matrix}\right)\widetilde{ G}\Big]\Phi(T)\int_0^T\Phi(s)^{-1}[\w b(s)-\w{\mathbb A}_1(s)\w\sigma(s)]ds\Big\rangle\\
&+\Big\langle\left(\begin{matrix}\alpha\\0\end{matrix}\right),\mathbb E\left(\begin{matrix}G_1^{-1}&0\\0&0\end{matrix}\right)\Big[I-\left(\begin{matrix}H_1&H_2 \\H_2&0\end{matrix}\right)\widetilde{ G}\Big]\Phi(T)\int_0^T\Phi(s)^{-1}\w\sigma(s)dW(s)\Big\rangle\\
=&\beta+\Big\langle\left(\begin{matrix}\alpha\\0\end{matrix}\right),\mathbb E\left(\begin{matrix}G_1^{-1}&0\\0&0\end{matrix}\right)\Big[I-\left(\begin{matrix}H_1&H_2 \\H_2&0\end{matrix}\right)\widetilde{ G}\Big]\int_0^T\Big[\w{\mathbb A}(s)\mathbb Y(s)+\w b(s)\Big]ds\Big\rangle\leq0.
\end{aligned}\end{equation}
\begin{proposition}
Under \emph{\textbf{(H1)}}-\emph{\textbf{(H4)}} and \emph{\textbf{(H6)}}, suppose $\det\left[I+G_1^{-1}H_1\right]\neq0$.
If \eqref{symmetric Riccati equation of special case II} and \eqref{BSDE of special case II} admit solutions such that \eqref{KTT-2} hold, then terminal-mean-constrained BFSDEs \eqref{eq56} is solvable.
\end{proposition}
In case with deterministic coefficients, \eqref{KTT-2} takes the following form
\begin{equation*}\begin{aligned}
&\beta+\Big\langle\left(\begin{matrix}\alpha\\0\end{matrix}\right),\left(\begin{matrix}G_1^{-1}&0\\0&0\end{matrix}\right)\Big[I-\left(\begin{matrix}H_1&H_2 \\H_2&0\end{matrix}\right)\widetilde{ G}\Big]\int_0^T[\w{\mathbb A}(s)\mathbb E\w{\mathbb Y}(s)+\w b(s)\Big]ds\Big\rangle\leq0.
\end{aligned}\end{equation*}
Let the fundamental solution matrices of ordinary differential equation (ODE)
$$d\w\varphi=-\w{\mathbb A}\w\varphi dt,\qquad\w\varphi(0)=I,$$ be $\w\Phi(t,0)$. Then
$$\bE\w{\b Y}(t)=-\w\Phi(t,0)\int_0^t\w\Phi(s,0)\w b(s)ds.$$
Therefore, the condition \eqref{KTT-2} becomes
\begin{equation}\begin{aligned}\label{KTT-2-deterministic}
&\beta+\Big\langle\left(\begin{matrix}\alpha\\0\end{matrix}\right),\left(\begin{matrix}G_1^{-1}&0\\0&0\end{matrix}\right)
\Big[I-\left(\begin{matrix}H_1&H_2 \\H_2&0\end{matrix}\right)\widetilde{ G}\Big]\int_0^T\Big[-\w{\mathbb A}(s)\w\Phi(s,0)\int_0^s\w\Phi(r,0)\w b(r)dr+\w b(s)\Big]ds\Big\rangle\leq0.
\end{aligned}\end{equation}
\begin{corollary}\label{deterministic-zero}Under \emph{\textbf{(H1)}}-\emph{\textbf{(H4)}} and \emph{\textbf{(H6)}}, suppose $\det\left[I+G_1^{-1}H_1\right]\neq0$.
If \eqref{symmetric Riccati equation of special case II} and \eqref{BSDE of special case II} admit solutions such that \eqref{KTT-2-deterministic} hold, then terminal-mean-constrained BFSDEs \eqref{eq56} is solvable. \end{corollary}
\begin{remark}
 Besides the Riccati equation decoupling method,  wellposedness of \eqref{eq56} can be established by some direct method. For example, under the conditions of Theorem \ref{wellposeness of {e1LK FBSDE}}, we know that there exists a unique adapted solution to {(BFSDE-2$^\prime$)}. Moreover, if $\beta+ \mathbb E \langle\left(\begin{matrix}\alpha\\0\end{matrix}\right), \left(\begin{matrix}G_1^{-1}&0\\0&0\end{matrix}\right)\mathbb Y(T) \rangle\leq 0$, then \eqref{eq56} admits a unique solution.
\end{remark}
\subsubsection{Multiplier $\lambda>0$}\label{nonzero}
In this section, we need to assume that the coefficients are deterministic, i.e., $A,B_1,B_2,C$, $G_1$,$Q_1,Q_2$,$S_1$,$S_2$,$R^1_{11}$ and $R^2_{22}$ are deterministic because the BFSDEs now takes some mean-field type form and its expectation is required to be computed.
In this case, (BFSDE-3$^\prime$) take the following form:
\begin{equation}\label{FBSDE-case 2}\left\{\begin{aligned}	
&dg=-\Big[A^\top  g-Q_1\bar X-Q_2h\Big]ds-\Big[C^\top g-S_1\bar Z-S_2q\Big]dW(s),\\
&d\bar Y=\Big[-A^\top \bar Y+Q_2\bar  X\Big]ds+\Big[-C^\top \bar Y+S_2\bar  Z\Big]dW(s),\\
&d\bar X=\Big[A\bar  X+B_1(R^1_{11})^{-1}B_1^\top g+B_2(R^2_{22})^{-1}B_2^\top\bar Y+C \bar Z\Big]ds+ \bar ZdW(s),\\
&dh=\Big[ Ah+B_2(R^2_{22})^{-1}B_2^\top g+ Cq\Big]ds+qdW(s),\\
& g(0)=H_1\bar X(0)+ H_2 h(0),\quad \bar Y(0)=H_2\bar  X(0),\\
& \bar  X(T)=-G_1^{-1}g(T)+G_1^{-1}\frac{\beta+\langle\alpha,G_1^{-1}\mathbb E g(T)\rangle}{ \langle\alpha,G_1^{-1}\alpha\rangle }\alpha,\quad h(T)=0,\\
&\beta+\langle\alpha, G_1^{-1} \mathbb Eg(T)\rangle>0.\qquad
\end{aligned}\right.\end{equation}
Note that \eqref{FBSDE-case 2} is solvable if  and only if  the following BFSDEs is solvable
\begin{equation*}\left\{\begin{aligned}	
&d\mathbb Eg=-\Big[A^\top\bE g- Q_1\bE\bar X- Q_2\bE h\Big]ds,\\
&d(g-\bE g)=-\Big[A^\top\bE (g-\bE g)- Q_1(\bar X-\bE\bar X)- Q_2(h-\bE h)\Big]ds\\
&\qquad-\Big[C^\top\bE g+C^\top (g-\bE g)-S_1\bar Z-S_2q\Big]dW(s),\\
&d\bE \bar Y=\Big[-A^\top\bE \bar Y+Q_2\bE\bar  X\Big]ds,\\
&d(\bar Y-\bE\bar Y)=\Big[-A^\top (\bar Y-\bE\bar Y)+Q_2(\bar  X-\bE\bar X)\Big]ds+\Big[-C^\top\bE\bar Y-C^\top (\bar Y-\bE \bar Y)+S_2\bar  Z\Big]dW(s),\\
&d\bE \bar X=\Big[B_1(R^1_{11})^{-1}B_1^\top\bE g+B_2(R^2_{22})^{-1}B_2^\top\bE\bar Y+A\bE\bar  X+C\bE \bar Z\Big]ds,\\
&d(\bar X-\bE\bar X)=\Big[B_1(R^1_{11})^{-1}B_1^\top (g-\bE g)+B_2(R^2_{22})^{-1}B_2^\top(\bar Y-\bE\bar Y)+A(\bar  X-\bE\bar X)+C \bar Z-C\bE\bar Z\Big]ds+ \bar ZdW(s),\\
&d\bE h=\Big[ B_2(R^2_{22})^{-1}B_2^\top\bE g+A\bE h+ C\bE q\Big]ds,\\
&d(h-\bE h)=\Big[B_2(R^2_{22})^{-1}B_2^\top (g-\bE g)+ A(h-\bE h)+ Cq-C\bE q\Big]ds+qdW(s),\\
&\bE g(0)=H_1\bE \bar X(0)+ H_2 \bE h(0),\quad g(0)-\bE g(0)=H_1(\bar X(0)-\bE \bar X(0))+ H_2 (h(0)-\bE h(0)),\\
& \bE \bar Y(0)=H_2\bE \bar  X(0),\quad \bar Y(0)-\bE\bar Y(0)=H_2(\bar X(0)-\bE\bar X(0))\\
& \bE \bar  X(T)=-G_1^{-1}\bE g(T)+\frac{ G_1^{-1}\alpha\alpha^\top G_1^{-1}}{\langle\alpha,G_1^{-1}\alpha\rangle}\mathbb E g(T)+\frac{G_1^{-1}\alpha\beta}{\langle\alpha,G_1^{-1}\alpha\rangle},\quad \bar X(T)-\bE\bar X(T)=-G_1^{-1}(g(T)-\bE g(T)),\\
&\bE h(T)=0,\quad h(T)-\bE h(T)=0,\\
&\beta+\langle\alpha, G_1^{-1} \mathbb Eg(T)\rangle>0.
\end{aligned}\right.\end{equation*}
Let $\c Y=(\bE g^\top, (g-\bE g)^\top, \bE\bar Y^\top, (\bar Y-\bE\bar Y)^\top)^\top$, $\c X=(\bE \bar X^\top, (\bar X-\bE \bar X)^\top, \bE h^\top,(h-\bE h)^\top)^\top$ and $\c Z=(0,\bar Z^\top,0, q^\top)^\top$, we have
\begin{equation*}\left\{\begin{aligned}	
&d\check{Y}=-[\c A\c Y+\c B\c X]dt-[\c A_1\c Y+\c B_1\c Z]dW,\\
&d\c X=[\c A_2\c Y+\c B_2\c X+\c C_2\check Z+\c D_2\bE\check Z]dt+\c ZdW,\\
&\c Y(0)=\c H\c X(0),\quad \c X(T)=\c G\c Y(T)+\c f,\\
&\beta+\langle\alpha, (G_1^{-1}\ 0\ 0\ 0) \mathbb E\c Y(T)\rangle>0,
\end{aligned}\right.\end{equation*}
where
\begin{equation}\label{new notation-2}\begin{aligned}
&\c A=\left(\begin{matrix}A^\top&0&0&0\\0&A^\top&0&0\\0&0&A^\top&0\\0&0&0&A^\top\end{matrix}\right),\quad
\c B=\left(\begin{matrix}-Q_1&0&-Q_2&0\\0&-Q_1&0&-Q_2\\-Q_2&0&0&0\\0&-Q_2&0&0\end{matrix}\right),\quad
\c A_1=\left(\begin{matrix}0&0&0&0\\C^\top&C^\top&0&0\\0&0&0&0\\0&0&C^\top&C^\top\end{matrix}\right),\\
&\c B_1=\left(\begin{matrix}0&0&0&0\\0&-S_1&0&-S_2\\0&0&0&0\\0&-S_2&0&0\end{matrix}\right),\quad
\c A_2=\left(\begin{matrix}B_1(R_{11}^1)^{-1}B_1^\top&0&B_2(R_{22}^2)^{-1}B_2^\top&0\\0&B_1(R_{11}^1)^{-1}B_1^\top&0&B_2(R_{22}^2)^{-1}B_2^\top
\\B_2(R_{22}^2)^{-1}B_2^\top&0&0&0\\0&B_2(R_{22}^2)^{-1}B_2^\top&0&0\end{matrix}\right),\\
&\c B_2=\left(\begin{matrix}A&0&0&0\\0&A&0&0\\0&0&A&0\\0&0&0&A\end{matrix}\right),\quad
\c C_2=\left(\begin{matrix}0&0&0&0\\0&C&0&0\\0&0&0&0\\0&0&0&C\end{matrix}\right),\quad
\c D_2=\left(\begin{matrix}0&C&0&0\\0&-C&0&0\\0&0&0&C\\0&0&0&-C\end{matrix}\right),\\
&\c H=\left(\begin{matrix}H_1&0&H_2&0\\0&H_1&0&H_2\\H_2&0&0&0\\0&H_2&0&0\end{matrix}\right),\quad
\c G=\left(\begin{matrix}-G_1^{-1}+\frac{ G_1^{-1}\alpha\alpha^\top G_1^{-1}}{\langle\alpha,G_1^{-1}\alpha\rangle}&0&0&0\\0&-G_1^{-1}&0&0\\0&0&0&0\\0&0&0&0\end{matrix}\right),\quad\c f=\left(\begin{matrix}\frac{G_1^{-1}\alpha\beta}{\langle\alpha,G_1^{-1}\alpha\rangle}\\0\\0\\0\end{matrix}\right).
\end{aligned}\end{equation}
Let $\c{\mathbb Y}=\c Y-\c H\c X$, then $\c{\b Y}(0)=0$ and
$(I-\c G\c H)\c X(T)=\c G\c{\b Y}(T)+\c f.$ Suppose $\det[I+(G_1^{-1}-\frac{ G_1^{-1}\alpha\alpha^\top G_1^{-1}}{\langle\alpha,G_1^{-1}\alpha\rangle})H_1]\neq0$, $\det[I+ G_1^{-1} H_1]\neq 0$, then $\det[I-\c G\c H]\neq 0$. Hence
\begin{equation*}\left\{\begin{aligned}	
&d\check{\b Y}=-[\c{\b A}\c{\b Y}+\c{\b B}\c X+\c{\b C}\c Z+\c{\b D}\bE\c Z]dt-[\c{\b A}_1\c{\b Y}+\c{\b B}_1\c X+\c{\b C}_1\c Z]dW,\\
&d\c X=[\c{\b A}_2\c{\b Y}+\c{\b B}_2\c X+\c{\b C}_2\check Z+\c{\b D}_2\bE\check Z]dt+\c ZdW,\\
&\c{\b Y}(0)=0,\quad \c X(T)=(I-\c G\c H)^{-1}\c G\c{\b Y}(T)+(I-\c G\c H)^{-1}\c f,\\
&\beta+\langle\alpha, (G_1^{-1}\ 0\ 0\ 0) (\mathbb E\c{\b Y}(T)+\c H\b E\c X(T)\rangle>0,
\end{aligned}\right.\end{equation*}where
\begin{equation}\label{new notation-3}\begin{aligned}
&\c{\b A}=\c A+\c H\c A_2,\quad \c{\b B}=\c A\c H+\c B+\c H\c A_2\c H+\c H\c B_2,\quad\c{\b C}=\c H\c C_2,\quad \c{\b D}=\c H\c D_2,\quad \c{\b A}_1=\c A_1,\quad \c{\b B}_1=\c A_1\c H,\\
&\c{\b C}_1=\c B_1+\c H,\quad\c{\b A}_2=\c A_2,\quad \c{\b B}_2=\c A_2\c H+\c B_2,\quad \c{\b C}_2=\c C_2,\quad \c{\b D}_2=\c D_2.
\end{aligned}\end{equation}
Suppose $\c X=\c P\c{\b Y}+\c p$, applying It\^{o}'s formula, we have
\begin{equation*}\begin{aligned}
d\c X%=&\c Pd\c{\b Y}+(d\c P)\c{\b Y}+d\c p\\
%=&\Big[-\c P\c{\b A}\c {\b Y}-\c P\c{\b B}\c X-\c P\c{\b C}\c Z-\c P\c{\b D}\bE\c Z\Big]dt+\Big[-\c P\c{\b A}_1\c{\b Y}-\c P\c{\b B}_1\c X-\c P\c{\b C}_1\c Z\Big]dW+(d\c P)\c{\b Y}+d\c p\\
=&\Big[-\c P\c{\b A}\c {\b Y}-\c P\c{\b B}\c P\c{\b Y}-\c P\c{\b B}\c p-\c P\c{\b C}\c Z-\c P\c{\b D}\bE\c Z\Big]dt+\Big[-\c P\c{\b A}_1\c{\b Y}-\c P\c{\b B}_1\c P\c{\b Y}-\c P\c{\b B}_1\c p-\c P\c{\b C}_1\c Z\Big]dW+(d\c P)\c{\b Y}+d\c p.
\end{aligned}\end{equation*}
Comparing the coefficients of the diffusion term, we have
$$-\c P\c{\b A}_1\c{\b Y}-\c P\c{\b B}_1\c P\c{\b Y}-\c P\c{\b B}_1\c p-\c P\c{\b C}_1\c Z=\c Z.$$
If $\det[I+\c P\c{\b C}_1]\neq0$,
$$\b E\c Z=-(I+\c P\c{\b C}_1)^{-1}(\c P\c{\b A}_1+\c P\c{\b B}_1\c P)\b E\c{\b Y}-(I+\c P\c{\b C}_1)^{-1}\c P\c{\b B}_1\c p.$$
 By taking expectation and comparing the coefficients of the drift term,
 we have the following Riccati equation
\begin{equation}\label{Riccati-case 2}\left\{\begin{aligned}
&\dot{\c P}-\c P\c{\b A}-\c P\c{\b B}\c P+(\c P\c{\b C}+\c P\c{\b D}+\c{\b C}_2+\c{\b D}_2)(I+\c P\c{\b C}_1)^{-1}(\c P\c{\b A}_1+\c P\c{\b B}_1\c P)-\c{\b A}_2-\c{\b B}_2\c P=0,\\
&\c P(T)=(I-\c G\c H)^{-1}\c G,\\
&\det[I+\c P\c{\b C}_1]\neq0,
\end{aligned}\right.\end{equation}
and the following backward ODE
\begin{equation}\label{BSDE-case 2}\left\{\begin{aligned}
&\dot{\c p}-\c P\c{\b B}\c p+(\c P\c{\b C}+\c P\c{\b D}+\c{\b C}_2+\c{\b D}_2)(I+\c P\c{\b C}_1)^{-1}\c P\c{\b B}_1\c p-\c{\b B}_2\c p=0, \\
&\c p(T)=(I-\c G\c H)^{-1}\c f.
\end{aligned}\right.\end{equation}
Moreover, we have
\begin{equation*}\begin{aligned}
d\bE\c{\b Y}=[\mathbf{A}\bE\c{\b Y}+\mathbf b]dt,\qquad \bE\c{\b Y}(0)=0,
\end{aligned}\end{equation*}
where
$$\mathbf A=-\c{\b A}-\c{\b B}\c P+(\c{\b C}+\c {\b D})(I+\c P\c{\b C}_1)^{-1}(\c P\c{\b A}_1+\c P\c{\b B}_1\c P),\quad\mathbf b=-\c{\b B}\c p+(\c{\b C}+\c {\b D})(I+\c P\c{\b C}_1)^{-1}\c P\c{\b B}_1\c p.$$
Let the fundamental solution matrices of ODE
$$d\c\varphi=\mathbf A\c\varphi dt,\qquad \c\varphi(0)=I,$$ be $\c\Phi(t,0)$. Then
$$\bE\c{\b Y}(t)=\c\Phi(t,0)\int_0^t\c\Phi(s,0)\mathbf b(s)ds.$$
Hence,
$$\bE\c Y(t)=(I+\c H\c P)\c\Phi(t,0)\int_0^t\c\Phi(s,0)\mathbf b(s)ds+\c H\c p(t).$$
Therefore, the (KKT) condition becomes
\begin{equation}\label{KKT-4}
\beta+\langle\alpha,(G_1^{-1}\ 0\ 0\ 0)(I+\c H\c P)\c\Phi(T,0)\int_0^T\c\Phi(s,0)\mathbf b(s)ds\rangle+\langle\alpha,(G_1^{-1}\ 0\ 0\ 0)\c H(I-\c G\c H)^{-1}\c f\rangle>0.
\end{equation}

\begin{proposition}\label{deterministic-nonzero}
Under \emph{\textbf{(H1)}}-\emph{\textbf{(H4)}}, suppose $\det[I+(G_1^{-1}-\frac{ G_1^{-1}\alpha\alpha^\top G_1^{-1}}{\langle\alpha,G_1^{-1}\alpha\rangle})H_1]\neq0$, $\det[I+ G_1^{-1} H_1]\neq 0$.
If \eqref{Riccati-case 2} and \eqref{BSDE-case 2} admit solutions such that \eqref{KKT-4} hold, then \eqref{FBSDE-case 2} is solvable.
\end{proposition}

\begin{remark}
Now let
\begin{equation}\nonumber\begin{aligned}
&\rho_1=\text{esssup}_{0\leq s\leq T}\text{esssup}_{\omega\in\Omega}\Lambda_{\max}(-\frac{1}{2}(\check{\mathbb A}(s)+\check{\mathbb A}(s)^\top)),\\
&\rho_2=\text{esssup}_{0\leq s\leq T}\text{esssup}_{\omega\in\Omega}\Lambda_{\max}(-\frac{1}{2}(\check{\mathbb B}_2(s)+\check{\mathbb B}_2(s)^\top)),\\
& k_{10}=k_{12}=0, k_1=\left|\left|\check{\mathbb B}\right|\right|,k_2=\left|\left|\check{\mathbb C}\right|\right|,k_3=\left|\left|\check{\mathbb D}\right|\right|,k_4=\left|\left|\check{\mathbb A}_2\right|\right|,
k_5=\left|\left|\check{\mathbb C}_2\right|\right|,k_6=\left|\left|\check{\mathbb D}_2\right|\right|,\\
&k_7=\sqrt{3}\left|\left|\check{\mathbb A}_1\right|\right|,k_8=\sqrt{3}\left|\left|\check{\mathbb B}_1\right|\right|, k_9=\sqrt{3}\left|\left|\check{\mathbb C}_1\right|\right|, k_{11}=\left|\left|(I-\check G\check H)^{-1}\check G\right|\right|. %k_{10}=\frac{|\alpha|^2}{|\mathbb E\left\langle \alpha,G^{-1}\alpha\right\rangle|}\left|\left|\left(\begin{matrix}G^{-1}_1& 0 \\0&0\end{matrix}\right)\right|\right|^2.
\end{aligned}\end{equation}
If $2(\rho_1 + \rho_2) < -2\left|\left|\check{\mathbb C}_2\right|\right|^2-2\left|\left|\check{\mathbb D}_2\right|\right|^2-3\left|\left|\check{\mathbb A}_1\right|\right|^2$, there exists a $\delta_2 > 0$,
which depends on $\rho_1,\rho_2,k_i,i=5,6,7$, and is independent of $T$, such that when $k_i\in [0, \delta_1)$, $i=1,2,3,4,8,9$, \eqref{FBSDE-lambda>0} admits a unique adapted solution. Moreover, if $\beta+\langle\alpha, (G_1^{-1}\ 0\ 0\ 0) (\mathbb E\check{\mathbb Y}(T)+\check H\mathbb E\check X(T)\rangle>0$, then \eqref{FBSDE-case 2} admits a unique solution.
\end{remark}

\subsubsection{Solvability of (5.9) and (5.16)}
 In Section \ref{zero} and Section \ref{nonzero}, we have discussed the solvability of (BFSDE-3) through Riccati equations \eqref{symmetric Riccati equation of special case II} and \eqref{Riccati-case 2}. Note that \eqref{symmetric Riccati equation of special case II} and \eqref{Riccati-case 2} are not standard Riccati equations and the general solvability remain widely open. We will present the solvability for some special but nontrivial cases. Suppose the coefficients are deterministic and $C=0$, in this case, \eqref{symmetric Riccati equation of special case II} and \eqref{Riccati-case 2} reduce to
\begin{equation}\label{5.9-2}
\begin{aligned}
\dot{\widetilde P}-\widehat{ A}-\widehat{ B}\widetilde P-\widetilde P\widehat{ B}^\top-\widetilde P\widetilde{ B}\widetilde P =0,\quad\widetilde P(T)=\widetilde{ G},
\end{aligned}
\end{equation}
and
\begin{equation}\label{5.16-2}\begin{aligned}
\dot{\check P}-\check P\check{\mathbb A}-\check P\check{\mathbb B}\check P-\check{\mathbb A}_2-\check{\mathbb B}_2\check P=0,\quad\check P(T)=(I-\check G\check H)^{-1}\check G.
\end{aligned}\end{equation}
%We have the following result concerning the solvability of \eqref{ } and \eqref{ }.
\begin{proposition}
For any $s\in[0,T]$, let $\Psi_1(\cdot,s)$ and $\Psi_2(\cdot,s)$ be the solutions of the following ODEs:
\begin{equation*}\label{funda1}\begin{aligned}
\frac{d}{dt}\Psi_1(t,s)=\widehat{\mathbf{A}}_1(t)\Psi_1(t,s),\quad  t\in[s,T],\quad\Psi_1(s,s)=I,
\end{aligned}\end{equation*}
and
\begin{equation*}\label{funda2}\begin{aligned}
\frac{d}{dt}\Psi_2(t,s)=\widehat{\mathbf{A}}_2(t)\Psi_2(t,s),\quad  t\in[s,T],\quad\Psi_2(s,s)=I,
\end{aligned}\end{equation*}
respectively, where
\begin{equation}\nonumber\begin{aligned}
\widehat{\mathbf{A}}_1(\cdot)=\left(
                                  \begin{smallmatrix}
                                    -\widehat B^\top & -\widetilde B \\
                                    \widehat A & \widehat B \\
                                  \end{smallmatrix}
                                \right),\quad \widehat{\mathbf{A}}_2(\cdot)=\left(
                                  \begin{smallmatrix}
                                    -\check{\mathbb A} & -\check{\mathbb B} \\
                                    \check{\mathbb A}_2 & \check{\mathbb B}_2 \\
                                  \end{smallmatrix}
                                \right).
\end{aligned}\end{equation}
Suppose
$$\Bigg[\left(
               \begin{array}{cc}
                 0 & I \\
               \end{array}
             \right)
\Psi_1(T,t)\left(
                             \begin{array}{c}
                               0 \\
                               I \\
                             \end{array}
                           \right)\Bigg]^{-1}\in L^1(0,T;\mathbb R^{2n\times 2n}),$$\ $$\Bigg[\left(
               \begin{array}{cc}
                 0 & I \\
               \end{array}
             \right)\Psi_2(T,t)\left(
                             \begin{array}{c}
                               0 \\
                               I \\
                             \end{array}
                           \right)\Bigg]^{-1}\in L^1(0,T;\mathbb R^{4n\times 4n}).$$
Then Riccati equation \eqref{5.9-2} and \eqref{5.16-2} admit unique solutions $\widetilde P(\cdot)$ and $\check P(\cdot)$, which are given by
\begin{equation}\label{Phi-explicit}\begin{aligned}
\widetilde P(t)=-\Bigg[\left(
               \begin{array}{cc}
                 0 & I \\
               \end{array}
             \right)\Psi_1(T,t)\left(
                             \begin{array}{c}
                               0 \\
                               I \\
                             \end{array}
                           \right)\Bigg]^{-1}\left(
               \begin{array}{cc}
                 0 & I \\
               \end{array}
             \right)\Psi_1(T,t)\left(
                             \begin{array}{c}
                               I \\
                               0 \\
                             \end{array}
                           \right),\qquad  t\in[0,T],
\end{aligned}\end{equation}
and
\begin{equation}\label{phi-explicit}\begin{aligned}
\check P(t)=-\Bigg[\left(
               \begin{array}{cc}
                 0 & I \\
               \end{array}
             \right)\Psi_2(T,t)\left(
                             \begin{array}{c}
                               0 \\
                               I \\
                             \end{array}
                           \right)\Bigg]^{-1}\left(
               \begin{array}{cc}
                 0 & I \\
               \end{array}
             \right)\Psi_2(T,t)\left(
                             \begin{array}{c}
                               I \\
                               0 \\
                             \end{array}
                           \right),\qquad  t\in[0,T],
\end{aligned}\end{equation}
respectively.
\end{proposition}
\begin{remark}
In general, \eqref{Riccati-case 2} is \emph{asymmetric} matric Riccati equation hence its solvability is more challenging than that of \eqref{symmetric Riccati equation of special case II}.
For example, if $H_1=H_2=Q_1=Q_2=0$, \eqref{symmetric Riccati equation of special case II} reduces to
\begin{equation*}
\left\{
\begin{aligned}
&d\widetilde P= \{\widehat{ A}+\widehat{ B}\widetilde P+\widetilde P\widehat{ B}^\top+\widetilde\Lambda\widehat{ C}^\top+\widehat{ C}\widetilde\Lambda-\widehat{ C}\widetilde P\widehat{ C}^\top\\
&\qquad\quad+(\widetilde \Lambda-\widehat{ C}\widetilde P)(\widetilde{ C}_1^{-1}+\widetilde P)^{-1}
(\widetilde \Lambda-\widehat{ C}^\top)
 \}ds+\widetilde\Lambda d W(s),\\
&\widetilde P(T)=\widetilde{ G},\qquad\det\left[I+\widetilde P\widetilde{ C}_1\right]\neq0,
\end{aligned}
\right.
\end{equation*}
which is the type of Riccati equation studied in \cite{Tang2003}. For this kind of Riccati equations, Please refer Section \ref{Riccati equation} for more information.
\end{remark}

\subsection{Solvability of Riccati equations}\label{Riccati equation}
 In this subsection, we will give the general solvability of (SRE-1) and (SRE-2).
For $a,c\in L^\infty_{\mathbb F}([0,T];\mathbb R^{n\times n})$, $b,d\in L^\infty_{\mathbb F}([0,T];\mathbb R^{n\times k})$, $q\in L^\infty_{\mathbb F}([0,T];\mathbb S^{n})$, $s\in L^\infty_{\mathbb F}([0,T];\mathbb S^{m})$, $M\in L^\infty_{\mathcal F_T}(\Omega;\mathbb R^{n\times n})$, consider the following Riccati equation
\begin{equation}\label{tang-1}\left\{\begin{aligned}
&dP=-\Big\{a^\top P+Pa+q-[Pb+K d][s+d^\top Pd]^{-1}[Pb+K d]^\top\Big\}dt+K dW,\\
&P(T)=M.
\end{aligned}\right.\end{equation}
If $q(\cdot)\geq0,\ M\geq0,\ s(\cdot)\gg0$, it follows from \cite[Theorem 5.3]{Tang2003} that \eqref{tang-1}
admits a unique solution $(P,K)\in L^\infty_{\mathbb F}([0,T];\mathbb S^{ n})\times L^2_{\mathbb F}([0,T];\mathbb S^{n})$ such that $P(\cdot)\geq0$.
Let $$k=n+m,d=\left(\begin{matrix}I&0\end{matrix}\right)_{n\times(n+m)}, b=\left(\begin{matrix}C&B_2\end{matrix}\right)_{n\times(n+m)}, s=\left(\begin{matrix}S_2&0\\0&R_{22}^2\end{matrix}\right)_{(n+m)\times(n+m)},$$ we have
\begin{equation*}\begin{aligned}
&(Pb+K d)(s+d^\top Pd)^{-1}(Pb+K d)^\top\\
=&\left(\begin{matrix}PC+K&PB_2\end{matrix}\right)\left(\begin{matrix}(P+S_2)^{-1}&0\\0&(R_{22}^2)^{-1}\end{matrix}\right)
\left(\begin{matrix}PC+K&PB_2\end{matrix}\right)^\top\\
=&(PC+K)(P+S_2)^{-1}(PC+K)^\top+PB_2(R_{22}^2)^{-1}B_2^\top P.
\end{aligned}\end{equation*}
Therefore, we have the following result.
\begin{proposition}\label{solvability of riccati-3}
If $Q_2(\cdot)\geq0$, $M\geq 0$, $S_2(\cdot)\gg0$ and $R_{22}^2(\cdot)\gg0$, then \emph{\textbf{(SRE-1)}} admits a unique solution $(P(\cdot),K(\cdot))\in L_{\mathbb F}^\infty(0,T;\mathbb S^n_+)\times L_{\mathbb F}^2(0,T;\mathbb S^n)$.
\end{proposition}
Furthermore, for (SRE-1) with scalar value, i.e., $n=m_1=m_2=1$, we have a better result as follows.
\begin{proposition}
Let $S_2(\cdot)\geq0$ and $Q_2(\cdot)\geq0$, then Riccati equation \textbf{\emph{(SRE-1)}} admits a unique solution $(P(\cdot),\Lambda(\cdot))\in L_{\mathbb F}^\infty(0,T;\mathbb R)\times L_{\mathbb F}^2(0,T;\mathbb R)$.
\end{proposition}
\proof
 For simplicity, we only consider the case $S_2(\cdot)=0$ since the proof of $S_2(\cdot)>0$ is similar.
Consider the following equation:
\begin{equation}\label{inverse}\begin{aligned}
dy=- [(B_2)^2(R_{22}^2)^{-1}+(C^2-2A)y-Q_2(s)y^2+2Cz ]ds+zdW(s),\qquad
y(T)=M^{-1}.
\end{aligned}\end{equation}
We will show that \eqref{inverse} admits a unique solution $ (y(s),z(s))\in L_{\mathbb F}^\infty(0,T;\mathbb R)\times L_{\mathbb F}^2(0,T;\mathbb R)$.
First we will prove the uniqueness.
Let $(\check{y}(s),\check{z}(s))$ and $(\widetilde{y}(s),\widetilde{z}(s))$ be two solutions of \eqref{inverse} such that $\check{z}\cdot W\triangleq\int_0^\cdot\check zdW(s)$ and $\widetilde{z}\cdot W$ are bounded-mean-oscillation (BMO) martingles (see \cite{HJZ2012}). Set $\hat y=\check{y}-\widetilde{y},$ $\hat z=\check{z}-\widetilde{z}.$
Then
\begin{equation*}\begin{aligned}
d\hat y= [Q_2(\check{y}+\widetilde{y})\hat y+(2A-C^2)\hat y-2C\hat z ]ds+\hat zdW,\qquad
\hat y(T)=0.
\end{aligned}\end{equation*}
Applying It\^{o}'s formula to $|\hat y|^2$ and taking conditional expectation, we deduce that there exists a constant $k>0$ such that
\begin{equation*}\begin{aligned}
|\hat y(s)|^2+\mathbb E_s\int_s^T|\hat z(r)|^2dr
=&\mathbb E [\int_s^T (-2Q_2(\check{y}+\widetilde{y})\hat y^2-(2A-C)\hat y^2+4C\hat y\hat z )dr |\mathcal F_s ]\\
\leq&k\mathbb E [\int_s^T|\hat y|^2dr |\mathcal F_s ]+\frac{1}{2}\mathbb E [\int_s^T|\hat z|^2dr |\mathcal F_s ].
\end{aligned}\end{equation*}
Therefore,
$$\check{y}(s)=\widetilde{y}(s),\qquad \check{z}(s)=\widetilde{z}(s),\qquad a.e.\ s\in[0,T],\ \mathbb P-a.s.$$
Hence, BSDE \eqref{inverse} admits at most one solution in $ L_{\mathbb F}^\infty(0,T;\mathbb R)\times L_{\mathbb F}^2(0,T;\mathbb R)$.

Let us now prove the existence. For $h(\cdot)\in L^\infty_{\mathbb F}([0,T];\mathbb R)$, define $\|h(\cdot)\|_\infty=\mathop{esssup}\limits_{0\leq s\leq T}\mathop{esssup}\limits_{\omega\in\Omega}|h(s)|$.
First, introduce the following equation:
\begin{equation}\label{upper bound-inverse}\begin{aligned}
d\bar y(s)=- [\|(B_2)^2(R_{22}^2)^{-1}\|_\infty+\|C^2-2A\|_\infty\bar y+2C\bar z ]ds+\bar zdW,\qquad
\bar y(T)=M^{-1}.\\
\end{aligned}\end{equation}
BSDE \eqref{upper bound-inverse} is a standard BSDE with Lipschitz continuous generator, therefore there exists a unique solution $(\bar y,\bar z)\in L^2_{\mathbb{F}}(\Omega;C([t,T];\mathbb R))\times L_{\mathbb F}^2(0,T;\mathbb R)$ and $\bar z\cdot W$ is a BMO martingale. Rewrite BSDE \eqref{upper bound-inverse} as
\begin{equation*}\begin{aligned}
d\bar y(s)=- [\|(B_2)^2(R_{22}^2)^{-1}\|_\infty+\|C^2-2A\|_\infty\bar y ]ds+\bar z (dW-2Cds ),\qquad
\bar y(T)=M^{-1}.\\
\end{aligned}\end{equation*}
Note that $2C(s)\cdot W$ is a BMO martingale, there exists a new probability measure $\widetilde{\mathbb P}$ such that
$W_s^{\widetilde{\mathbb P}}\triangleq W_s-\int_0^s2C(s)ds$ is a Brownian motion under $\widetilde{\mathbb P}$. Therefore,
$$\bar y(s)=\mathbb E^{\widetilde{\mathbb P}}
 [e^{\|C^2-2A\|_\infty(T-s)}+\|(B_2)^2(R_{22}^2)^{-1}\|_\infty\int_s^Te^{\|C^2-2A\|_\infty(s-v)}dv |\mathcal F_s ],$$
from which we deduce that $\bar y(s)\leq c_1$
where
$c_1=e^{\|C^2-2A\|_\infty T}+\|(B_2)^2(R_{22}^2)^{-1}\|_\infty T e^{\|C^2-2A\|_\infty T}.$
Next, introduce the following BSDE:
\begin{equation}\label{lower bound-inverse}\begin{aligned}
d\underline{y}(s)=- [-\|C^2-2A\|_\infty \underline{y}(s)-c_1Q_2\underline{y}(s)
+2C\underline{z}(s) ] ds+\underline{z}(s) dW(s),\qquad
\underline y(T)=M^{-1}.\\
\end{aligned}\end{equation}
BSDE \eqref{lower bound-inverse} is a standard BSDE with Lipschitz continuous generator, therefore there exists a unique solution $(\underline y,\underline z)\in L^2_{\mathbb{F}}(\Omega;C([t,T];\mathbb R))\times L_{\mathbb F}^2(0,T;\mathbb R)$ and $\underline z\cdot W$ is a BMO martingale. Rewrite BSDE \eqref{lower bound-inverse} as
\begin{equation*}\begin{aligned}
d\underline y(s)=- [-\|C^2-2A\|_\infty \underline{y}(s)-c_1Q_2\underline{y}(s) ]ds+\underline z(dW-2Cds),\qquad
\underline y(T)=M^{-1}.\\
\end{aligned}\end{equation*}Therefore,
$\underline y(s)=\mathbb E^{\widetilde{\mathbb P}}
 [e^{-2\|C^2-2A\|_\infty(T-s)-c_1Q_2(T-s)} |\mathcal F_s ],$
from which we deduce that $\underline y(s)\geq c_2$,
where $c_2=e^{-2\|C^2-2A\|_\infty T-c_1Q_2T}.$
Moreover, by comparison theorem for BSDE with Lipschitz continuous generator, for $s\in[0,T]$ we have $ c_2\leq \underline y(s)\leq\bar y(s)\leq c_1,\ \mathbb P-a.s.$
Define $\Theta_{c_1,c_2}(y)\triangleq c_1I\{y<c_1\}+pI\{c_1\leq y\leq c_2\}+c_2I\{y>c_2\},$  and introduce the following BSDE
\begin{equation*}\begin{aligned}
dy=- [(B_2)^2(R_{22}^2)^{-1} +(C^2-2A)y- Q_2\Theta_{c_1,c_2}(y)y +2Cz ] ds+z dW(s),\qquad
y(T)=M^{-1}.
\end{aligned}\end{equation*}
The above BSDE is a standard quadratic BSDE and by \cite[Theorem 2.3]{Kobylanski2000}, it admits at most one solution $(y^{c_1,c_2}(s),z^{c_1,c_2}(s))\in L_{\mathbb F}^\infty(0,T;\mathbb R)\times L_{\mathbb F}^2(0,T;\mathbb R)$.
Furthermore, let
\begin{equation*}\left\{\begin{aligned}
&f_1(y,z)=(B_2)^2(R_{22}^2)^{-1} +(C^2-2A)y- Q_2\Theta_{c_1,c_2}(y)y +2Cz,\\
&f_2(y,z)=\|(B_2)^2(R_{22}^2)^{-1}\|_\infty+\|C^2-2A\|_\infty y+2C z,\\
&f_3(y,z)=-\|C^2-2A\|_\infty y-c_1Q_2y
+2Cz.\\
\end{aligned}\right.\end{equation*}
It is easy to check that there exist positive constants $k_1,k_2,k_3$ such that
$$|f_1(y,z)|\leq k_1|y|+k_2z^2+k_3,\qquad\frac{\partial f_1}{\partial z}=2C,\qquad\frac{\partial f_1}{\partial y}\leq C^2-2A-Q_2c_2,\qquad\mathbb P-a.s.$$
Moreover, we have
$\forall s\in[0,T], f_1(\bar y(s),\bar z(s))\leq f_2(\bar y(s),\bar z(s)),
f_1(\underline y(s),\underline z(s))\geq f_3(\underline y(s),\underline z(s)), \mathbb P-a.s.$
Hence, it follows from \cite[Theorem 2.6]{Kobylanski2000} that $\forall s\in[0,T], \underline y(s)\leq y(s)\leq \bar y(s),\mathbb P-a.s.$
Therefore, \eqref{inverse} admits a solution $(y(s),z(s))\in L_{\mathbb F}^\infty(0,T;\mathbb R)\times L_{\mathbb F}^2(0,T;\mathbb R)$ and there exist two positive constants $c_1,c_2$ such that
$\forall s\in[0,T], c_2\leq y(s)\leq c_1,\mathbb P-a.s.$
Let $P(s)=y^{-1}(s),$ $K(s)=-z(s)y^{-2}(s),$ we have
\begin{equation*}\begin{aligned}
dP=- [Q_2+2A P-B_2^2(R_{22}^2)^{-1} P^2-(PC +K)^2P^{-1} ] ds+K dW(s),\qquad
P(T)=M,
\end{aligned}\end{equation*}
i.e., {(SRE-1)} admits a solution $(P(s),K(s))\in L_{\mathbb F}^\infty(0,T;\mathbb R)\times L_{\mathbb F}^2(0,T;\mathbb R)$. Moreover, the uniqueness of solution of {(SRE-1)} follows from that of \eqref{inverse}. $\Box$

For (SRE-2), by \cite[Theorem 5.3]{Tang2003} again, we have the following result.
\begin{proposition}\label{solvability of riccati-4}
Let $Q_1(\cdot)\geq0,G_1\geq0,S_1(\cdot)\gg0,R_{11}^1(\cdot)\gg0$, then Riccati equation \textbf{\emph{(SRE-2)}} admits a unique solution $(P_L(\cdot),\Lambda_L(\cdot))\in L_{\mathbb F}^\infty(0,T;\mathbb S^n_+)\times L_{\mathbb F}^2(0,T;\mathbb S^n)$.
\end{proposition}
\begin{remark}
The wellposedness of (SRE-1) and (SRE-2) are established under some positive definite assumptions. For the indefinite case, please refer \cite{SXY2021} for more information.
\end{remark}

\section{Application}\label{motivation}

To simplify presentation, we consider a financial market with only one (risk-free) bond and one (risky) stock. Their prices $P_0(\cdot), P_1(\cdot)$ evolve respectively:\begin{equation}\label{30}
\begin{aligned}
dP_0(s)=r(s)P_0(s)ds,\quad
dP_1(s)=P_1(s)[\mu(s)ds+\sigma(s)dW(s)],\quad P_0(0)=p_0,\quad P_1(0)=p_1.
\end{aligned}
\end{equation}
Here, random processes $r(\cdot),\mu(\cdot), \sigma(\cdot)$ are respectively interest rate, risky return rate, and instantaneous volatility. Assume that $\mu(s)> r(s), a.s.$ for any $0\leq s\leq T$, thus the risk premium is positive.
Suppose there involve two economic agents formulated in leader-follower decision pattern: one agent acts as leader (it may be interpreted as firm owner or principal) wish to achieve or hedge some terminal wealth objective $\xi$. It can also be interpreted as some payoff target to be replicated in pension planning. In addition, the leader may utilize some continuous consumption process with instantaneous rate $c_1(\cdot)$. Another agent is the follower (e.g., pension fund manager) who may implement a dynamic operation (or, wage) process $c_2(\cdot).$ Thus, the state process $X(s)$ becomes the following BSDE
\begin{equation}\label{wealth-2}
\begin{aligned}
dX(s)= [r(s)X(s)+\frac{\mu(s)-r(s)}{\sigma(s)}Z(s)-c_1(s)-c_2(s) ]ds+Z(s)dW(s),\qquad
X(T)=\xi,
\end{aligned}\end{equation}
where $Z(s)=\pi(s)\sigma(s)$ and $\pi(\cdot)$ is the amount of risky allocation from wealth process.
For $i=1,2$, let $\mathcal U_i\triangleq\{c_i:[0,T]\times\Omega\rightarrow\mathbb R|c_i(\cdot)\text{ is }\mathbb F-\text{progressively measurable},\ \mathbb E\int_0^T|c_i(t)|^2dt<\infty\}$ represent the operation and consumption process. Also, the terminal target $\xi$ is subject to some practical constraints
$\mathcal U_{\mathcal K}$, $\mathcal U_{\alpha,\beta}$ and $\mathcal U(\mathcal K,\alpha,\beta)$.
%$\mathcal U_{\mathcal K}=L^2_{\mathcal F_T}(\Omega;\mathcal K)$ and $\mathcal U_{\alpha,\beta}\triangleq\{\xi|\xi\in L^2_{\mathcal F_T}(\Omega;\mathbb R),\alpha\mathbb E\xi\geq\beta\}.$ Define $\mathcal U(\mathcal K,\alpha,\beta)\triangleq\mathcal U_{\mathcal K}\cap\mathcal U_{\alpha,\beta}$. Any element $(\xi,c_1(\cdot))\in\mathcal U(\mathcal K,\alpha,\beta)\times \mathcal U_1$ is called an admissible control of the leader and any element $c_2(\cdot)\in\mathcal U_2$ is called an admissible control of follower.
 For quadratic hedging, the following functionals are often employed (see \cite{DR1991}):
\begin{equation}\label{utility-1}\begin{aligned}
&J_1(\xi,c_1(\cdot),c_2(\cdot))\triangleq\frac{1}{2}\mathbb E \{G_1\xi^2+H_1X^2(0)+\int_0^T [Q_1(s)X^2(s)+S_1(s)Z^2(s)+R_1(s)c_1^2(s) ]ds \},\\
&J_2(\xi,c_1(\cdot),c_2(\cdot))\triangleq\frac{1}{2}\mathbb E \{H_2X^2(0)+\int_0^T [Q_2(s)X^2(s)+S_2(s)Z^2(s)+R_2(s)c_2^2(s) ]ds \},
\end{aligned}\end{equation}where $H_1, H_2$ denote the initial hedging surplus index.
Comparing with \eqref{e state equation} and \eqref{cost functionals}, we obtain that
$
A=r,\quad B_1=B_2\equiv-1,\quad C=\frac{\mu-r}{\sigma},\quad R_{11}^1=R_1,\quad R_{22}^2=R_2.
$ Thus {(SRE-1)} takes the following form:
\begin{equation}\label{appli-ricatti-1}\left\{\begin{aligned}
&dP=- [Q_2+2Pr-\frac{P^2}{R_2}- (P\frac{\mu-r}{\sigma} +K )^2\frac{1}{P+S_2} ] ds+K dW(s),\\
&P(T)=M>0,\qquad P(s)+S_2(s)>0,\quad 0\leq s\leq T.
\end{aligned}\right.\end{equation}
 Now, we give the following assumption:
\begin{description}
	\item[(H7)] All the coefficients in \eqref{wealth-2} and \eqref{utility-1} are bounded. Moreover, $H_1\geq0,Q_1(\cdot)\geq0,G_1>0,S_1(\cdot)\gg0,R_{1}(\cdot)\gg0,Q_2(\cdot)\geq0,S_2(\cdot)\gg0,R_2(\cdot)\gg0.$
\end{description}
Note that in (H7), there has no positive (semi-)definite assumption on $H_2$.
Under (H7),
It follows from Proposition \ref{solvability of riccati-3} that \eqref{appli-ricatti-1} admits a unique solution.
Moreover, if
$P(0)+H_2\geq 0$, then by Proposition \ref{convexity of follower} and Theorem \ref{th1}, the optimal consumption $\bar c_2(\cdot)$ of the follower is given by
$
\bar c_2(\cdot)=-\frac{\bar Y(\cdot)}{R_2(\cdot)},
$
where $(\bar Y,\bar X,\bar Z)$ is the solution of the following BFSDEs
\begin{equation}\label{Hamilton-con-2}\left\{\begin{aligned}
&d\bar Y=(-r\bar Y+Q_2\bar X)ds-(\frac{\mu-r}{\sigma}\bar Y-S_2\bar Z)dW(s),\\
 &d\bar X= [r\bar X-c_1+\frac{\bar Y}{R_2}+\frac{\mu-r}{\sigma}\bar Z ]ds+\bar ZdW,\\
&\bar Y(0)=H_2\bar X(0),\qquad \bar X(T)=\xi.
\end{aligned}\right.\end{equation}For the leader, {(SRE-2)} takes the following form:
\begin{equation}\label{Riccati-application-2}
\left\{\begin{aligned}
&dP_L=- [\mathbb A^\top P_L+P_L\mathbb A+\mathbb C^\top P_L\mathbb C+\mathbb Q+\Lambda_L \mathbb C+\mathbb C^{\top}\Lambda_L- (\mathbb B^\top P_L+\mathbb D^\top P_L\mathbb C+\mathbb D^{\top}\Lambda_L )^\top\\
&\qquad\qquad\mathbb K^{-1} (\mathbb B^\top P_L+\mathbb D^\top P_L\mathbb C+\mathbb D^{\top}\Lambda_L ) ]ds+\Lambda_L dW(s),\\
&P_L(T)=\left(\begin{matrix} 0&0\\0&G_1\end{matrix}\right),\qquad\mathbb K(s)\triangleq\mathbb R(s)+\mathbb D^\top(s) P_L(s)\mathbb D(s)>0,\quad 0\leq s\leq T,
\end{aligned}\right.\end{equation}
where
$
\mathbb A=\left(\begin{matrix} -r&Q_2\\\frac{1}{R_{2}} &r\end{matrix}\right),\mathbb B=\left(\begin{matrix} 0&0\\ -1 &\frac{\mu-r}{\sigma}\end{matrix}\right),\mathbb C=\left(\begin{matrix} -\frac{\mu-r}{\sigma}&0\\0&0\end{matrix}\right),\mathbb D=\left(\begin{matrix} 0&S_2\\0 &1\end{matrix}\right),\mathbb Q=\left(\begin{matrix}  0&0\\0 &Q_1\end{matrix}\right),\mathbb R=\left(\begin{matrix}  R_{1}&0\\0 &S_1\end{matrix}\right).
$ Under (H7), it follows from Proposition \ref{solvability of riccati-4} that \eqref{Riccati-application-2} admits a unique solution.
Furthermore, suppose that $P_L(0)+\left(\begin{matrix}0&0\\0&H_1\end{matrix}\right)\geq0$ and (F) holds, it follows from Proposition \ref{Proposition leader cost functional convex} and Theorem \ref{th 1L general GL lambda} that an optimal control of  leader is given by
$
(\bar\xi,\bar c_1(\cdot))= (\text{\emph{Proj}}_{\mathcal K} [\frac{-g(T)+\lambda\alpha}{G_1} ],-\frac{g(\cdot)}{R_1(\cdot)} ),
$
where $(\lambda;\bar Y,g,\bar X,\bar Z ,h,q)$ is the solution of the following BFSDEs
\begin{equation}\label{appli-general case}
\left\{\begin{aligned}	
&dg=(-rg+Q_1\bar X+Q_2h)ds-(\frac{\mu-r}{\sigma} g-S_1\bar Z-S_2q)dW(s),\\
 &d\bar Y=(-r\bar Y+Q_2\bar X)ds-(\frac{\mu-r}{\sigma} \bar Y-S_2\bar Z)dW(s),\\
&d\bar X= [r\bar  X+\frac{g}{R_1}+\frac{\bar Y}{R_2}+\frac{\mu-r}{\sigma}\bar Z ]ds+ \bar ZdW(s),\quad dh= [ rh+\frac{g}{R_2}+ \frac{\mu-r}{\sigma}q ]ds+qdW(s),\\
& g(0)=H_1\bar X(0)+H_2h(0),\quad \bar Y(0)=H_2\bar X(0),\quad\bar  X(T)= \text{\emph{Proj}}_{\mathcal K} [\frac{-g(T)+\lambda\alpha}{G_1} ],\quad h(T)=0,\\
&\lambda (\beta-\alpha\mathbb E\text{\emph{Proj}}_{\mathcal K} [\frac{-g(T)+\lambda\alpha}{G_1} ] )=0, \quad \lambda \geq 0, \quad \beta \leq \alpha\mathbb E\text{\emph{Proj}}_{\mathcal K} [\frac{-g(T)+\lambda\alpha}{G_1} ].
\end{aligned}\right.\end{equation}

\subsection{Pointwise constraint}
In case there has only one constraint $\xi\in\mathcal U_{\mathcal K}$, \eqref{appli-general case} assumes the following form:
\begin{equation}\label{appli-case 1}
\left\{\begin{aligned}	
&dg=(-rg+Q_1\bar X+Q_2h)ds-(\frac{\mu-r}{\sigma} g-S_1\bar Z-S_2q)dW(s),\\
 &d\bar Y=(-r\bar Y+Q_2\bar X)ds-(\frac{\mu-r}{\sigma} \bar Y-S_2\bar Z)dW(s),\\
&d\bar X= [r\bar  X+\frac{g}{R_1}+\frac{\bar Y}{ R_2}+\frac{\mu-r}{\sigma}\bar Z ]ds+ \bar ZdW(s),\quad dh= [ rh+\frac{g}{R_2}+ \frac{\mu-r}{\sigma}q ]ds+qdW(s),\\
& g(0)=H_1\bar X(0)+H_2h(0),\quad \bar Y(0)=H_2\bar X(0),\quad \bar  X(T)= \text{\emph{Proj}}_{\mathcal K} [-G_1^{-1}g(T) ],\quad h(T)=0.
\end{aligned}\right.\end{equation}
Here, the parameters of (H5) can be chosen as follows:
{\small{
\begin{equation}\label{coefficients-1}\begin{aligned}
&\rho_1=\rho_2=-\mathop{essinf}\limits_{0\leq s\leq T}\mathop{essinf}\limits_{\omega\in\Omega}|r(s)|, k_1=\left\|\left(\begin{matrix} Q_1 & Q_2\\ Q_2 & 0\end{matrix}\right)\right\|, k_2=k_3=k_6=k_8=k_{12}=0,\\
&k_4=\left\|\left(\begin{matrix} R_1^{-1}(\cdot) & R_2^{-1}(\cdot)\\ R_2^{-1}(\cdot) & 0\end{matrix}\right)\right\|,k_5=\mathop{esssup}\limits_{0\leq s\leq T}\mathop{esssup}\limits_{\omega\in\Omega}\sqrt{2}\left|\frac{\mu(s)-r(s)}{\sigma(s)}\right|,k_7=\mathop{esssup}\limits_{0\leq s\leq T}\mathop{esssup}\limits_{\omega\in\Omega}2\left|\frac{\mu(s)-r(s)}{\sigma(s)}\right|, \\
 &k_9=\sqrt{2}\left\|\left(\begin{matrix} S_1 & S_2\\ S_2 & 0\end{matrix}\right)\right\|,k_{10}=\left\|\left(\begin{matrix} H_1 & H_2\\ H_2 & 0\end{matrix}\right)\right\|,k_{11}=\mathop{esssup}\limits_{\omega\in\Omega}G_1^{-1}.
\end{aligned}\end{equation}}}
Therefore, by Theorem \ref{wellposeness of {e1LK FBSDE}}, we have the following result.
\begin{proposition}\label{wellposeness of FBSDE-appli-2}
Suppose that
$2\rho_1<-2k_5^2-k_7^2.$ There exists a $\delta_1>0$, which depends on $\rho_1,k_i,i=5,7$, such that when $k_1,k_4,k_9,k_{10} \in[0, \delta_1)$, there exists a unique adapted solution to \eqref{appli-case 1}.
\end{proposition}
Under (H7), suppose $P_L(0)+\left(\begin{matrix}0&0\\0&H_1\end{matrix}\right)\geq0$ and conditions of  Proposition \ref{wellposeness of FBSDE-appli-2} holds, the optimal control of $\mathcal A_L$ is given by
$
(\bar\xi,\bar c_1(\cdot))= (\text{\emph{Proj}}_{\mathcal K} [\frac{-g(T)}{G_1} ],-\frac{g(\cdot)}{R_1(\cdot)} ),
$
where $(\bar Y,g,\bar X,\bar Z ,h,q)$ is the solution of \eqref{appli-case 1}.

Next, we give a more specific condition for wellposedness of \eqref{appli-case 1}. For $c_1$, $c_3$, $c_4$, $\bar\rho_1$ and $\bar\rho_2$, please refer Lemma 7.2 and Lemma 7.3 in the appendix of \cite{FHH}.
\begin{remark}
For some $\varepsilon>0$, set $c_1=\frac{k_1}{\varepsilon}$, $c_4=\frac{k_4}{\varepsilon}$, $c_5=\frac{k_5}{2(k_5^2+\varepsilon)}$ and  $c_6=\frac{k_6}{2(k_6^2+\varepsilon)}$.  Suppose $2(\rho_1+\rho_2)<-2k_5^2-k_7^2-3\varepsilon$ and define $d=-2k_5^2-k_7^2-3\varepsilon-2\rho_1-2\rho_2=-4k_5^2-3\varepsilon-4\rho_1$. Therefore, we can choose $\rho$ such that $\bar\rho_1=\bar\rho_2=\frac{d}{2}$. In this case, let
$$\theta
%= \left(\frac{1}{\bar\rho_2}+\frac{1}{1-k_4c_4}+1\right)\left(k^2_9 +\frac{k_3c_3}{\bar\rho_1}\right)
=
%=(\frac{2}{d}+3)(k_9^2+\frac{2}{d})=
\left(\frac{2}{-4\rho_1-4k_5^2-3\varepsilon}+5+\frac{2k_5^2+2k_6^2}{\varepsilon}\right)\left(2k_9^2+\frac{2k_4^2}
{-\varepsilon(4\rho_1+4k_5^2+3\varepsilon)}\right).$$
That is, if
$
4\rho_1<-4k_5^2-3\varepsilon, k_9^2\theta<1, k_{10}^2\theta<1, \frac{k_1^2\theta}{\varepsilon}<1,
$
 there exists a unique  solution to \eqref{appli-case 1}.
\end{remark}

\subsection{Affine constraint}
In this subsection, suppose that there is only one constraint $\xi\in\mathcal U_{\alpha,\beta}$ and all the coefficients are deterministic. We will study the case $\lambda=0$ and $\lambda\neq0$ separately.

In case $\lambda=0$,
 \eqref{appli-general case} becomes
\begin{equation}\label{appli-case 2.1}
\left\{\begin{aligned}	
&dg=(-rg+Q_1\bar X+Q_2h)ds-(\frac{\mu-r}{\sigma} g-S_1\bar Z-S_2q)dW(s),\\
 &d\bar Y=(-r\bar Y+Q_2\bar X)ds-(\frac{\mu-r}{\sigma} \bar Y-S_2\bar Z)dW(s),\\
&d\bar X= [r\bar  X+\frac{g}{R_1 }+\frac{\bar Y}{ R_2}+\frac{\mu-r}{\sigma}\bar Z ]ds+ \bar ZdW(s), dh= [ rh+\frac{g}{R_2}+ \frac{\mu-r}{\sigma}q ]ds+qdW(s),\\
& g(0)=H_1\bar X(0)+H_2h(0),\quad \bar Y(0)=H_2\bar X(0),\quad \bar  X(T)=-G_1^{-1} g(T),\quad h(T)=0,\\
&\beta+\alpha G_1^{-1}\mathbb Eg(T)\leq 0.
\end{aligned}\right.\end{equation}Here, we present some detailed solution.
Note that \eqref{appli-case 2.1} is linear and homogeneous. Thus if \eqref{appli-case 2.1} admits an unique solution, it must be $\bar Y=g=\bar X=\bar Z=h=q\equiv0$. In this case, if $\beta\leq0$, (KKT) condition holds. Let $\rho_1,\rho_2,k_i,i=1,\cdots,12$ be defined as in \eqref{coefficients-1}. Therefore, by Theorem \ref{wellposeness of {e1LK FBSDE}}, suppose that
$2\rho_1<-2k_5^2-k_7^2$ and $\beta\leq0$, if there exists a $\delta_2>0$ depending on $\rho_1,k_i,i=5,7$, such that $k_1,k_4,k_9,k_{10} \in[0, \delta_2)$, there exists a unique adapted solution to \eqref{appli-case 2.1}. Therefore, under (H7), suppose that $P_L(0)+\left(\begin{matrix}0&0\\0&H_1\end{matrix}\right)\geq0$,
$2\rho_1<-2k_5^2-k_7^2$ and $\beta\leq0$, if there exists a $\delta_2>0$ depending on $\rho_1,k_i,i=5,7$, such that $k_1,k_4,k_9,k_{10} \in[0, \delta_2)$, the optimal control of the leader is given by
$
(\bar\xi,\bar c_1(\cdot))=(0,0).
$

Next we consider the case {$\lambda>0$}. \eqref{appli-general case} becomes
\begin{equation}\label{appli-case 2.2}
\left\{\begin{aligned}	
&dg=(-rg+Q_1\bar X+Q_2h)ds-(\frac{\mu-r}{\sigma} g-S_1\bar Z-S_2q)dW(s),\\
&d\bar Y=(-r\bar Y+Q_2\bar X)ds-(\frac{\mu-r}{\sigma} \bar Y-S_2\bar Z)dW(s),\\
&d\bar X= [r\bar  X+\frac{g}{R_1 }+\frac{\bar Y}{ R_2}+\frac{\mu-r}{\sigma}\bar Z ]ds+ \bar ZdW(s),\quad dh= [ rh+\frac{g}{R_2}+ \frac{\mu-r}{\sigma}q ]ds+qdW(s),\\
& g(0)=H_1\bar X(0)+H_2h(0),\quad \bar Y(0)=H_2\bar X(0),\quad \bar  X(T)=-G_1^{-1} g(T)+G_1^{-1}\mathbb E g(T)+\frac{\beta}{\alpha},\\
& h(T)=0,\quad\beta+\alpha G_1^{-1}\mathbb Eg(T)> 0.
\end{aligned}\right.\end{equation}
Hence, \eqref{Riccati-case 2} and \eqref{BSDE-case 2} take the form
\begin{equation}\label{application-riccati-4}\left\{\begin{aligned}
&\dot{\check P}-\check P\check{\mathbb A}-\check P\check{\mathbb B}\check P+(\check P\check{\mathbb C}+\check P\check{\mathbb D}+\check{\mathbb C}_2+\check{\mathbb D}_2)(I+\check P\check{\mathbb C}_1)^{-1}(\check P\check{\mathbb A}_1+\check P\check{\mathbb B}_1\check P)-\check{\mathbb A}_2-\check{\mathbb B}_2\check P=0,\\
&\check P(T)=(I-\check G\check H)^{-1}\check G,\qquad\det[I+\check P\check{\mathbb C}_1]\neq0,
\end{aligned}\right.\end{equation}
\begin{equation}\label{back ODE-2}
\dot{\check p}-\check P\check{\mathbb B}\check p+(\check P\check{\mathbb C}+\check P\check{\mathbb D}+\check{\mathbb C}_2+\check{\mathbb D}_2)(I+\check P\check{\mathbb C}_1)^{-1}\check P\check{\mathbb B}_1\check p-\check{\mathbb B}_2\check p=0, \quad\check p(T)=(I-\check G\check H)^{-1}\check f,
\end{equation}
where the notations of the coefficients are defined in \eqref{new notation-3}. Now (KKT) condition \eqref{KKT-4}
becomes
\begin{equation}\label{KTT-deterministic-nonzero}
\beta+\alpha(G_1^{-1}\ 0\ 0\ 0)(I+\check H\check P)\check\Phi(T,0)\int_0^T\check\Phi(s,0)\mathbf b(s)ds+\alpha(G_1^{-1}\ 0\ 0\ 0)\check H(I-\check G\check H)^{-1}\check f>0,
\end{equation}
where $\check\Phi(t,0)$ is the fundamental solution matrices of ODE
$$d\check\varphi=[-\check{\mathbb A}-\check{\mathbb B}\check P+(\check{\mathbb C}+\check {\mathbb D})(I+\check P\check{\mathbb C}_1)^{-1}(\check P\check{\mathbb A}_1+\check P\check{\mathbb B}_1\check P)]\check\varphi dt,\qquad \check\varphi(0)=1.$$
Under (H7), if $G_1^{-1} H_1\neq -1$, by Proposition \ref{deterministic-nonzero}, if \eqref{application-riccati-4} and \eqref{back ODE-2} admit solutions such that \eqref{KTT-deterministic-nonzero} holds, then \eqref{appli-case 2.2} is solvable. Therefore, an optimal control of the leader is given by
$
(\bar\xi,\bar c_1(\cdot))= (\frac{-g(T)+\lambda\alpha}{G_1},-\frac{g(\cdot)}{R_1(\cdot)} ),
$
where $(\lambda;\bar Y,g,\bar X,\bar Z ,h,q)$ is the solution of \eqref{appli-case 2.2}.

\section*{Conclusion}
We discuss an open-loop backward Stackelberg differential game where the state is characterized by BSDE  and the decisions of leader consist of a static terminal-perturbation and a dynamic linear-quadratic control. The terminal control is subject to pointwise and
 expectation constraints. Our open-loop Stackelberg equilibrium is represented by some coupled BFSDEs with mixed initial-terminal conditions and the global solvability of such
BFSDEs is discussed in some nontrivial cases.

\section{Appendix}\label{appendix}

\subsection{Proof of Proposition \ref{convexity of follower}:}\label{sec7.1}
Before we give the proof the Proposition \ref{convexity of follower}, first we prove the following lemma.
\begin{lemma}\label{lemma-lower bound of BSDE}
For any $u_2(s)\in\mathcal U_2[0,T]$, let $(x^{(u_2)}(s),z^{(u_2)}(s))$ be the solution of
\begin{equation*}\begin{aligned}
dx^{(u_2)}(s)=\Big[A(s)x^{(u_2)}(s)+B_2(s)u_2(s)+C(s)z^{(u_2)}(s)\Big]ds+z^{(u_2)}(s)dW(s),\quad x^{(u_2)}(T)=0.
\end{aligned}\end{equation*}
\end{lemma}
Then for any $\Theta(\cdot)\in L^\infty_{\mathbb F}(0,T;\mathbb R^{m_2\times n})$, there exists a constant $L>0$ such that
\begin{equation}\label{lower bound of BSDE}
\mathbb E\int_0^T\Big|u_2(s)-\Theta(s)x^{(u_2)}(s)\Big|^2ds\geq L\mathbb E\int_0^T|u_2(s)|^2ds,\qquad\forall u_2(\cdot)\in\mathcal U_2[0,T].
\end{equation}
\proof
Let $\Theta(\cdot)\in L^\infty_{\mathbb F}(0,T;\mathbb R^{m_2\times n})$, define a bounded linear operator $\mathcal L:\mathcal U_2[0,T]\rightarrow\mathcal U_2[0,T]$ by
$\mathcal L u_2=u_2-\Theta x^{(u_2)}.$ Then $\mathcal L$ is a bijection, and its inverse is given by
$\mathcal L^{-1}u_2=u_2+\Theta\widetilde x^{(u_2)},$
where $\widetilde X^{(u_2)}(s)$ is the solution of
\begin{equation*}\begin{aligned}
d\widetilde x^{(u_2)}(s)=\Big[A(s)\widetilde x^{(u_2)}(s)+B_2(s)(\Theta(s)\widetilde x^{(u_2)}(s)+u_2(s))+C(s)\widetilde z^{(u_2)}(s)\Big]ds+\widetilde z^{(u_2)}(s)dW(s),\quad \widetilde x^{(u_2)}(T)=0.
\end{aligned}\end{equation*}
By the bounded inverse theorem, $\mathcal L^{-1}$ is bounded with norm $\|\mathcal L^{-1}\|>0$. Therefore,
\begin{equation*}\begin{aligned}
\mathbb E\int_0^T|u_2(s)|^2ds\leq&\|\mathcal L^{-1}\|\mathbb E\int_0^T|\mathcal Lu_2(s)|^2ds
=\|\mathcal L^{-1}\|\mathbb E\int_0^T\Big|u_2(s)-\Theta(s)x^{(u_2)}(s)\Big|^2ds.\quad\Box
\end{aligned}\end{equation*}
Now we will give the proof of Proposition \ref{convexity of follower}.
First, let $$\Gamma\triangleq-\Big(Q_2+PA+A^\top P-(PC +K)(P+S_2)^{-1}(C^\top P+K)-PB_2(R_{22}^2)^{-1}B_2^\top P\Big).$$
Let processes $P(\cdot)$ satisfy the following equations
\begin{equation*}\begin{aligned}
dP(s)=\Gamma(s) ds+K(s) dW(s),\qquad
P(T)=M^{-1}.
\end{aligned}\end{equation*}
Applying It\^{o}'s formula to $\langle Px,x\rangle$, 
integrating from $0$ to $T$, we have
\begin{equation*}\begin{aligned}
&-\mathbb E\langle P(0)x(0),x(0)\rangle
=\mathbb E\int_0^T\Big[\langle(\Gamma+PA+A^\top P)x,x\rangle+
2\langle x,PB_2u_2\rangle+\langle Pz,z\rangle+2\langle (PC+K)z,x\rangle\Big]ds.
\end{aligned}\end{equation*}
Therefore,
\begin{equation*}\begin{aligned}
J(u_2(\cdot))
=&\mathbb{E}\langle H_2 x(0),x(0)\rangle\!+\mathbb E\int_{0}^{T}\Big[\langle Q_2x,x\rangle+\langle S_2z,z\rangle+\langle R^2_{22} u_2,u_2\rangle\Big] ds+\mathbb E\langle P(0)x(0),x(0)\rangle\\
&+\mathbb E\int_0^T\Big[\langle(\Gamma+PA+A^\top P)x,x\rangle+
2\langle x,PB_2u_2\rangle\langle Pz,z\rangle+2\langle (PC+K)z,x\rangle\Big]ds.
\end{aligned}\end{equation*}
First, consider the terms involving $u_2$,
\begin{equation*}\begin{aligned}
&\langle R^2_{22} u_2,u_2\rangle+2\langle x,PB_2u_2\rangle
=\Big\langle R_{22}^2\Big(u_2+(R_{22}^2)^{-1}B_2^\top Px\Big),u_2+(R_{22}^2)^{-1}B_2^\top Px\Big\rangle-\Big\langle x,PB_2(R_{22}^2)^{-1}B_2^\top Px\Big\rangle.
\end{aligned}\end{equation*}
Next, consider the terms involving $z$,
\begin{equation*}\begin{aligned}
&\langle S_2z,z\rangle+\langle Pz,z\rangle+2\langle (PC+K)z,x\rangle\\
=&\Big\langle(P+S_2)\Big(z+(P+S_2)^{-1}(C^\top P+K)x\Big),z+(P+S_2)^{-1}(C^\top P+K)x\Big\rangle\\
&-\Big\langle x,(PC +K)(P+S_2)^{-1}(C^\top P+K)x\Big\rangle.
\end{aligned}\end{equation*}
Therefore,
\begin{equation*}\begin{aligned}
J(u_2(\cdot))
=&\mathbb{E}\Big\langle (H_2+P(0)) x(0),x(0)\rangle\Big\rangle+\mathbb E\int_0^T\Big\langle R_{22}^2\Big(u_2+(R_{22}^2)^{-1}B_2^\top Px\Big),u_2+(R_{22}^2)^{-1}B_2^\top Px\Big\rangle ds\\
&+\mathbb E\int_{t}^{T}\Big\langle(P+S_2)\Big(z+(P+S_2)^{-1}(C^\top P+K)x\Big),z+(P+S_2)^{-1}((C^\top P+K)x\Big\rangle ds
\geq0.
\end{aligned}\end{equation*}
Moreover, if $R_{22}^2(\cdot)\geq\delta I$, then it follows from
 Lemma \ref{lemma-lower bound of BSDE} that
\begin{equation*}\begin{aligned}
\mathcal J(u_2(\cdot))
\geq\delta\mathbb E\int_0^T\Big\langle u_2+(R_{22}^2)^{-1}B_2^\top Px,u_2+(R_{22}^2)^{-1}B_2^\top Px\Big\rangle ds
\geq\delta\gamma\mathbb E\int_0^T\Big|u_2(s)\Big|^2 ds.\quad\Box
\end{aligned}\end{equation*}

\subsection{Proof of Proposition \ref{Proposition leader cost functional convex}:}\label{sec7.2}
For simplicity, let
\begin{equation}\nonumber
\begin{aligned}
\Gamma=-\Big(&\mathbb Q+P_L\mathbb A+\Lambda\mathbb C+\mathbb C^\top\Lambda+\mathbb A^\top P_L+\mathbb C^\top P_L\mathbb C\\
&-(\mathbb B^\top P_L+\mathbb D^\top\Lambda+\mathbb D^\top P_L\mathbb C)^\top(\mathbb R+\mathbb D^\top P_L\mathbb D)^{-1}(\mathbb B^\top P_L+\mathbb D^\top\Lambda+\mathbb D^\top P_L\mathbb C)\Big).
\end{aligned}
\end{equation}
Applying It\^{o}'s formula to $\left\langle P_L\left(\begin{matrix} Y\\X\end{matrix}\right),
\left(\begin{matrix} Y\\X\end{matrix}\right)\right\rangle$,
we have
\begin{equation}\begin{aligned}\nonumber
&d\left\langle P_L\left(\begin{matrix} Y\\X\end{matrix}\right),
\left(\begin{matrix} Y\\X\end{matrix}\right)\right\rangle\\
=&\left\langle P_L\mathbb A\left(\begin{matrix} Y\\X\end{matrix}\right)+P_L\mathbb B\left(\begin{matrix} u_1\\Z\end{matrix}\right),\left(\begin{matrix} Y\\X\end{matrix}\right)\right\rangle ds+\left\langle \Gamma\left(\begin{matrix} Y\\X\end{matrix}\right),\left(\begin{matrix} Y\\X\end{matrix}\right)\right\rangle ds+\left\langle\Lambda\mathbb C\left(\begin{matrix} Y\\X\end{matrix}\right)+\Lambda\mathbb D\left(\begin{matrix} u_1\\Z\end{matrix}\right),\left(\begin{matrix} Y\\X\end{matrix}\right)\right\rangle ds\\
&+\left\langle P_L\left(\begin{matrix} Y\\X\end{matrix}\right),\mathbb A\left(\begin{matrix} Y\\X\end{matrix}\right)+\mathbb B\left(\begin{matrix} u_1\\Z\end{matrix}\right)\right\rangle ds+\left\langle \Lambda\left(\begin{matrix} Y\\X\end{matrix}\right),\mathbb C\left(\begin{matrix} Y\\X\end{matrix}\right)+\mathbb D\left(\begin{matrix} u_1\\Z\end{matrix}\right)\right\rangle ds\\
&+\left\langle P_L\left(\mathbb C\left(\begin{matrix} Y\\X\end{matrix}\right)+\mathbb D\left(\begin{matrix} u_1\\Z\end{matrix}\right)\right),\mathbb C\left(\begin{matrix} Y\\X\end{matrix}\right)+\mathbb D\left(\begin{matrix} u_1\\Z\end{matrix}\right)\right\rangle ds+[\cdots]dW(s).
\end{aligned}\end{equation}
Thus,
\begin{equation}\begin{aligned}\nonumber
&\mathbb E\left\langle P_L(T)\left(\begin{matrix} Y(T)\\X(T)\end{matrix}\right),
\left(\begin{matrix} Y(T)\\X(T)\end{matrix}\right)\right\rangle-
\mathbb E\left\langle P_L(0)\left(\begin{matrix} Y(0)\\X(0)\end{matrix}\right),
\left(\begin{matrix} Y(0)\\X(0)\end{matrix}\right)\right\rangle\\
=&\mathbb E\int_0^T\left\langle \left(\begin{matrix} Y\\X\end{matrix}\right),P_L\mathbb A\left(\begin{matrix} Y\\X\end{matrix}\right)+\Gamma\left(\begin{matrix} Y\\X\end{matrix}\right)+\Lambda\mathbb C\left(\begin{matrix} Y\\X\end{matrix}\right)+\mathbb C^\top\Lambda\left(\begin{matrix} Y\\X\end{matrix}\right)+\mathbb A^\top P_L\left(\begin{matrix} Y\\X\end{matrix}\right)+\mathbb C^\top P_L\mathbb C\left(\begin{matrix} Y\\X\end{matrix}\right)\right\rangle ds\\
&+\mathbb E\int_0^T\left\langle \left(\begin{matrix} u_1\\Z\end{matrix}\right),2\mathbb B^\top P_L\left(\begin{matrix} Y\\X\end{matrix}\right)+2\mathbb D^\top\Lambda\left(\begin{matrix} Y\\X\end{matrix}\right)+2\mathbb D^\top P_L\mathbb C\left(\begin{matrix} Y\\X\end{matrix}\right)\right\rangle ds\\
&+\mathbb E\int_0^T\left\langle\left(\begin{matrix} u_1\\D\end{matrix}\right),\mathbb D^\top P_L\mathbb D\left(\begin{matrix} u_1\\D\end{matrix}\right)\right\rangle ds.\\
\end{aligned}\end{equation}
Adding this into the functional, we have
\begin{equation}\nonumber
\begin{aligned}
&J(\xi,u_1(\cdot))\\
=&\frac{1}{2}\mathbb E\left\{\langle G_1\xi,\xi\rangle+\langle H_1X(0),X(0)\rangle+\int^T_0\left[\langle Q_1 X, X\rangle+\langle S_1Z, Z\rangle+\left\langle R_{11}^1
      	u_{1},
      	u_{1})\right\rangle\right]ds\right.\\
&-\left\langle P_L(T)\left(\begin{matrix} Y(T)\\X(T)\end{matrix}\right),
\left(\begin{matrix} Y(T)\\X(T)\end{matrix}\right)\right\rangle+\left\langle P_L(0)\left(\begin{matrix} Y(0)\\X(0)\end{matrix}\right),
\left(\begin{matrix} Y(0)\\X(0)\end{matrix}\right)\right\rangle\\
&+\int_0^T\left\langle \left(\begin{matrix} Y\\X\end{matrix}\right),P_L\mathbb A\left(\begin{matrix} Y\\X\end{matrix}\right)+\Gamma\left(\begin{matrix} Y\\X\end{matrix}\right)+\Lambda\mathbb C\left(\begin{matrix} Y\\X\end{matrix}\right)+\mathbb C^\top\Lambda\left(\begin{matrix} Y\\X\end{matrix}\right)+\mathbb A^\top P_L\left(\begin{matrix} Y\\X\end{matrix}\right)+\mathbb C^\top P_L\mathbb C\left(\begin{matrix} Y\\X\end{matrix}\right)\right\rangle ds\\
&+\int_0^T\left\langle \left(\begin{matrix} u_1\\Z\end{matrix}\right),2\mathbb B^\top P_L\left(\begin{matrix} Y\\X\end{matrix}\right)+2\mathbb D^\top\Lambda\left(\begin{matrix} Y\\X\end{matrix}\right)+2\mathbb D^\top P_L\mathbb C\left(\begin{matrix} Y\\X\end{matrix}\right)\right\rangle ds\\
&+\left.\int_0^T\left\langle\left(\begin{matrix} u_1\\Z\end{matrix}\right),\mathbb D^\top P_L\mathbb D\left(\begin{matrix} u_1\\Z\end{matrix}\right)\right\rangle ds\right\}\\
=&\frac{1}{2}\mathbb E\left\{\left\langle\left[\left(\begin{matrix} 0&0\\0&H_1\end{matrix}\right)+P_L(0)\right]\left(\begin{matrix} Y(0)\\X(0)\end{matrix}\right),\left(\begin{matrix} Y(0)\\X(0)\end{matrix}\right)\right\rangle+\frac{1}{2}\langle G_1\xi,\xi\rangle\right.\\
&+\int_0^T\left\langle \left(\begin{matrix} Y\\X\end{matrix}\right),\mathbb Q\left(\begin{matrix} Y\\X\end{matrix}\right)+P_L\mathbb A\left(\begin{matrix} Y\\X\end{matrix}\right)+\Gamma\left(\begin{matrix} Y\\X\end{matrix}\right)+\Lambda\mathbb C\left(\begin{matrix} Y\\X\end{matrix}\right)+\mathbb C^\top\Lambda\left(\begin{matrix} Y\\X\end{matrix}\right)+\mathbb A^\top P_L\left(\begin{matrix} Y\\X\end{matrix}\right)+\mathbb C^\top P_L\mathbb C\left(\begin{matrix} Y\\X\end{matrix}\right)\right\rangle ds\\
&+\int_0^T\left\langle \left(\begin{matrix} u_1\\Z\end{matrix}\right),2\mathbb B^\top P_L\left(\begin{matrix} Y\\X\end{matrix}\right)+2\mathbb D^\top\Lambda\left(\begin{matrix} Y\\X\end{matrix}\right)+2\mathbb D^\top P_L\mathbb C\left(\begin{matrix} Y\\X\end{matrix}\right)\right\rangle ds\\
&\left.+\int_0^T\left\langle\left(\begin{matrix} u_1\\Z\end{matrix}\right),(\mathbb R+\mathbb D^\top P_L\mathbb D)\left(\begin{matrix} u_1\\Z\end{matrix}\right)\right\rangle ds\right\}.\\
\end{aligned}
\end{equation}
Note that
\begin{equation}\begin{aligned}\nonumber
&\mathbb E\int_0^T\left\langle \left(\begin{matrix} u_1\\Z\end{matrix}\right),2\mathbb B^\top P_L\left(\begin{matrix} Y\\X\end{matrix}\right)+2\mathbb D^\top\Lambda\left(\begin{matrix} Y\\X\end{matrix}\right)+2\mathbb D^\top P_L\mathbb C\left(\begin{matrix} Y\\X\end{matrix}\right)\right\rangle ds+\mathbb E\int_0^T\left\langle\left(\begin{matrix} u_1\\Z\end{matrix}\right),(\mathbb R+\mathbb D^\top P_L\mathbb D)\left(\begin{matrix} u_1\\Z\end{matrix}\right)\right\rangle ds\\
=&\mathbb E\int_0^T\left\langle(\mathbb R+\mathbb D^\top P_L\mathbb D)\left(\left(\begin{matrix} u_1\\Z\end{matrix}\right)+(\mathbb R+\mathbb D^\top P_L\mathbb D)^{-1}(\mathbb B^\top P_L+\mathbb D^\top\Lambda+\mathbb D^\top P_L\mathbb C)\left(\begin{matrix} Y\\X\end{matrix}\right)\right)\right.,\\
&\left.\qquad\left(\begin{matrix} u_1\\Z\end{matrix}\right)+(\mathbb R+\mathbb D^\top P_L\mathbb D)^{-1}(\mathbb B^\top P_L+\mathbb D^\top\Lambda+\mathbb D^\top P_L\mathbb C)\left(\begin{matrix} Y\\X\end{matrix}\right)\right\rangle ds\\
&-\mathbb E\int_0^T\left\langle(\mathbb B^\top P_L+\mathbb D^\top\Lambda+\mathbb D^\top P_L\mathbb C)\left(\begin{matrix} Y\\X\end{matrix}\right),(\mathbb R+\mathbb D^\top P_L\mathbb D)^{-1}(\mathbb B^\top P_L+\mathbb D^\top\Lambda+\mathbb D^\top P_L\mathbb C)\left(\begin{matrix} Y\\X\end{matrix}\right)\right\rangle ds.
\end{aligned}
\end{equation} and recall the definition of $\Gamma$, we have
\begin{equation}\nonumber
\begin{aligned}
&J(\xi,u_1(\cdot))\\
=&\frac{1}{2}\mathbb E\left\{\left\langle\left[\left(\begin{matrix} 0&0\\0&H_1\end{matrix}\right)+P_L(0)\right]\left(\begin{matrix} Y(0)\\X(0)\end{matrix}\right),\left(\begin{matrix} Y(0)\\X(0)\end{matrix}\right)\right\rangle\right.\\
&+\int_0^T\left\langle \left(\begin{matrix} Y\\X\end{matrix}\right),\mathbb Q\left(\begin{matrix} Y\\X\end{matrix}\right)+P_L\mathbb A\left(\begin{matrix} Y\\X\end{matrix}\right)+\Gamma\left(\begin{matrix} Y\\X\end{matrix}\right)+\Lambda\mathbb C\left(\begin{matrix} Y\\X\end{matrix}\right)+\mathbb C^\top\Lambda\left(\begin{matrix} Y\\X\end{matrix}\right)+\mathbb A^\top P_L\left(\begin{matrix} Y\\X\end{matrix}\right)+\mathbb C^\top P_L\mathbb C\left(\begin{matrix} Y\\X\end{matrix}\right)\right\rangle ds\\
&-\int_0^T\left\langle\left(\begin{matrix} Y\\X\end{matrix}\right),(\mathbb B^\top P_L+\mathbb D^\top\Lambda+\mathbb D^\top P_L\mathbb C)^\top(\mathbb R+\mathbb D^\top P_L\mathbb D)^{-1}(\mathbb B^\top P_L+\mathbb D^\top\Lambda+\mathbb D^\top P_L\mathbb C)\left(\begin{matrix} Y\\X\end{matrix}\right)\right\rangle ds\\
&+\int_0^T\left\langle(\mathbb R+\mathbb D^\top P_L\mathbb D)\left[\left(\begin{matrix} u_1\\Z\end{matrix}\right)+(\mathbb R+\mathbb D^\top P_L\mathbb D)^{-1}(\mathbb B^\top P_L+\mathbb D^\top\Lambda+\mathbb D^\top P_L\mathbb C)\left(\begin{matrix} Y\\X\end{matrix}\right)\right]\right.,\\
&\left.\left.\qquad\qquad\left(\begin{matrix} u_1\\Z\end{matrix}\right)+(\mathbb R+\mathbb D^\top P_L\mathbb D)^{-1}(\mathbb B^\top P_L+\mathbb D^\top\Lambda+\mathbb D^\top P_L\mathbb C)\left(\begin{matrix} Y\\X\end{matrix}\right)\right\rangle ds\right\}\\
=&\frac{1}{2}\mathbb E\left\langle\left[\left(\begin{matrix} 0&0\\0&H_1\end{matrix}\right)+P_L(0)\right]\left(\begin{matrix} Y(0)\\X(0)\end{matrix}\right),\left(\begin{matrix} Y(0)\\X(0)\end{matrix}\right)\right\rangle\\
&+\frac{1}{2}\mathbb E\int_0^T\left\langle(\mathbb R+\mathbb D^\top P_L\mathbb D)\left[\left(\begin{matrix} u_1\\Z\end{matrix}\right)+(\mathbb R+\mathbb D^\top P_L\mathbb D)^{-1}(\mathbb B^\top P_L+\mathbb D^\top\Lambda+\mathbb D^\top P_L\mathbb C)\left(\begin{matrix} Y\\X\end{matrix}\right)\right]\right.,\\
&\left.\qquad\qquad\left(\begin{matrix} u_1\\Z\end{matrix}\right)+(\mathbb R+\mathbb D^\top P_L\mathbb D)^{-1}(\mathbb B^\top P_L+\mathbb D^\top\Lambda+\mathbb D^\top P_L\mathbb C)\left(\begin{matrix} Y\\X\end{matrix}\right)\right\rangle ds.\qquad\Box
\end{aligned}
\end{equation}

\subsection{Proof of Theorem \ref{th A 2}:}\label{sec7.3}

First, we will give two lemmas. Note that for a given $(X(\cdot),Z(\cdot))\times X(0)\in L^2_{\mathbb F}(0,T;\mathbb R^m)\times L^2_{\mathbb F}(0,T;\mathbb R^m)\times L^2_{\mathcal F_0}(\Omega;\mathbb R^m)$, where $X(0)$ is the value of process $X(\cdot)$ at initial time, the forward equation in the BFSDEs \eqref{eA2 1} has a unique solution $Y(\cdot)\in L^2_{\mathbb F}(0,T;\mathbb R^n)$, thus we introduce a map $\mathbb M_1:L^2_{\mathbb F}(0,T;\mathbb R^m)\times L^2_{\mathbb F}(0,T;\mathbb R^m)\times L^2_{\mathcal F_0}(\Omega;\mathbb R^m)\rightarrow L^2_{\mathbb F}(0,T;\mathbb R^n)$, through
\begin{equation}\label{eA2 X}
Y(t)=h(X(0))+\int^t_0b(s,Y,X,Z,\mathbb EZ)ds+\int^t_0\sigma(s,Y,X,Z)dW(s).
\end{equation}
Therefore, $\mathbb E\sup_{t\in[0,T]}|Y(t)|^2<\infty$. For any $\rho\in\mathbb R$, define $\|X\|_\rho\triangleq\left(E\int_0^Te^{-\rho t}|X(t)|^2dt\right)^\frac{1}{2}$.
%Note that $\|X\|_\rho$ is equivalent to the norm $\Big(\mathbb E\int_0^T|X(t)|^2dt\Big)^\frac{1}{2}$.
\begin{lemma}\label{Lemma A1}
Let $Y_i(\cdot)$ be the solution of \eqref{eA2 X} corresponding to $(X_i(\cdot),Z_i(\cdot))\in L^2_{\mathbb F}(0,T;\mathbb R^m)\times L^2_{\mathbb F}(0,T;\mathbb R^m), i=1,2$. Then for all $\rho\in\mathbb R$, $c_1,c_2,c_3>0$, we have
\begin{equation}\label{eA2 2}
\begin{aligned}
&e^{-\rho t}\mathbb E|\widehat Y(t)|^2+\bar \rho_1\int^t_0e^{-\rho s}\mathbb E|\widehat Y(s)|^2ds\\
\leq& k^2_{10}\mathbb E|\widehat X(0)|^2
+(k_1c_1+k^2_8)\int^t_0e^{-\rho s}\mathbb E|\widehat X(s)|^2ds+(k_2c_2+k_3c_3+k^2_9)\int^t_0e^{-\rho s}\mathbb E|\widehat Z(s)|^2ds,
\end{aligned}
\end{equation}
\begin{equation}\label{eA2 3}
\begin{aligned}
e^{-\rho t}\mathbb E|\widehat Y(t)|^2
\leq& k^2_{10}e^{-\bar\rho_1t}\mathbb E|\widehat X(0)|^2
+(k_1c_1+k^2_8)\int^t_0e^{-\bar\rho_1(t-s)-\rho s}\mathbb E|\widehat X(s)|^2ds\\
&+(k_2c_2+k_3c_3+k^2_9)\int^t_0e^{-\bar\rho_1(t-s)-\rho s}\mathbb E|\widehat Z(s)|^2ds,
\end{aligned}
\end{equation}
where $\bar\rho_1=\rho-2\rho_1-k_1c_1^{-1}-k_2c_2^{-1}-k_3c_3^{-1}-k^2_7$ and $\widehat \varphi=\varphi_1-\varphi_2,\varphi=Y,X,Z$.
Moreover, we have
\begin{equation}\label{eA2 3 1}
||\widehat Y(\cdot)||_\rho^2\leq \frac{1-e^{-\bar\rho_1T}}{\bar\rho_1}\left[k^2_{10}\mathbb E|\widehat X(0)|^2+(k_1c_1+k^2_8)||\widehat X(\cdot)||_\rho^2+(k_2c_2+k_3c_3+k^2_9)||\widehat Z(\cdot)||_\rho^2\right],\end{equation}
\begin{equation}\label{eA2 3 2}
e^{-\rho T}\mathbb E|\widehat Y(T)|^2\leq \max\{1,e^{-\bar\rho_1 T}\}\left[k^2_{10}\mathbb E|\widehat X(0)|^2+(k_1c_1+k^2_8)||\widehat X(\cdot)||_\rho^2+(k_2c_2+k_3c_3+k^2_9)||\widehat Z(\cdot)||_\rho^2\right].\end{equation}
In particular, if $\bar\rho_1>0$, we have
\begin{equation*}\label{eA2 3 3}
e^{-\rho T}\mathbb E|\widehat Y(T)|^2\leq k^2_{10}\mathbb E|\widehat X(0)|^2+(k_1c_1+k^2_8)||\widehat X(s)||_\rho^2+(k_2c_2+k_3c_3+k^2_9)||\widehat Z(\cdot)||_\rho^2.\end{equation*}
\end{lemma}
\proof
Under (H5), applying It\^{o}'s formula to $e^{-\rho s}|\widehat Y(s)|^2$ and taking expectation, we obtain \eqref{eA2 2}. Furthermore, applying It\^{o}'s formula again to $e^{-\bar\rho_1(t-s)-\rho s}|\widehat Y(s)|^2$ for $s\in[0,t]$ and taking expectation, we get \eqref{eA2 3}. Integrating both sides of \eqref{eA2 3} on $[0,T]$ and noting $\frac{1-e^{-\bar\rho_1(T-s)}}{\bar\rho_1}\leq \frac{1-e^{-\bar\rho_1T}}{\bar\rho_1},\forall s\in[0,T]$, we have \eqref{eA2 3 1}. Letting $t=T$ in \eqref{eA2 3} and noticing that $e^{-\bar\rho_1(T-s)}\leq \max\{1,e^{-\bar\rho_1 T}\}$, we obtain \eqref{eA2 3 2}.\qquad $\Box$\\
Similarly, for given $Y(\cdot)\in L^2_{\mathbb F}(0,T;\mathbb R^n)$, the backward equation in the BFSDEs \eqref{eA2 1} has a unique solution $(X(\cdot),Z(\cdot))\in L^2_{\mathbb F}(0,T;\mathbb R^m)\times L^2_{\mathbb F}(0,T;\mathbb R^m)$, and the corresponding initial value of $X(\cdot)$ is denoted by $X(0)\in L^2_{\mathcal F_0}(\Omega;\mathbb R^m)$. Thus, we can introduce another map $\mathbb M_2:L^2_{\mathbb F}(0,T;\mathbb R^n)\rightarrow L^2_{\mathbb F}(0,T;\mathbb R^m)\times L^2_{\mathbb F}(0,T;\mathbb R^m)\times L^2_{\mathcal F_0}(\Omega;\mathbb R^m)$, through
\begin{equation}\label{eA2 Y}
X(t)=g(Y(T),\mathbb EY(T))+\int^T_0f(s,Y,X,Z,\mathbb EZ)ds-\int^T_0ZdW(s),
\end{equation}
which satisfies $\mathbb E\sup\limits_{t\in[0,T]}|X(t)|^2+\mathbb E\int_0^T|Z(t)|^2dt<\infty$.
Similar to Lemma \ref{Lemma A1}, we have
\begin{lemma}\label{Lemma A2}Let $(X_i(\cdot),Z_i(\cdot))$ be the solution of \eqref{eA2 Y} corresponding to $Y_i(\cdot)\in L^2_{\mathbb F}(0,T;\mathbb R^n), i=1,2$. Then for all $\rho\in\mathbb R$, $c_4,c_5,c_6>0$, we have
\begin{equation*}\label{eA2 4}
\begin{aligned}
&e^{-\rho t}\mathbb E|\widehat X(t)|^2+\bar \rho_2\int^T_0e^{-\rho s}\mathbb E|\widehat X(s)|^2ds+(1-k_5c_5-k_6c_6)\int^T_0e^{-\rho s}\mathbb E|\widehat Z(s)|^2ds\\
\leq& 2(k^2_9+k_{10}^2)\mathbb E|\widehat Y(T)|^2+k_4c_4\int^T_0e^{-\rho s}\mathbb E|\widehat Y(s)|^2ds,
\end{aligned}\end{equation*}
\begin{equation*}\label{eA2 5}
\begin{aligned}
&e^{-\rho t}\mathbb E|\widehat X(t)|^2+(1-k_5c_5-k_6c_6)\int^T_0e^{-\bar\rho_2(s-t)-\rho s}\mathbb E|\widehat Z(s)|^2ds\\
\leq& 2(k^2_9+k_{10}^2)e^{-\bar\rho_2(T-t)-\rho T}\mathbb E|\widehat Y(T)|^2+k_4c_4\int^T_0e^{-\bar\rho_2(s-t)-\rho s}\mathbb E|\widehat Y(s)|^2ds,
\end{aligned}
\end{equation*}
where $\bar\rho_2=-\rho-2\rho_2-k_4c_4^{-1}-k_5c_5^{-1}-k_6c_6^{-1}$ and $\widehat \varphi=\varphi_1-\varphi_2,\varphi=Y,X,Z$.
Moreover, choosing $c_4\in(0,k_4^{-1})$, we have
\begin{equation*}\label{eA2 5 1}
||\widehat X(\cdot)||_\rho^2\leq \frac{1-e^{-\bar\rho_2T}}{\bar\rho_2}\left[2(k^2_9+k_{10}^2)e^{-\rho T}\mathbb E|\widehat Y(T)|^2+k_4c_4||\widehat Y(\cdot)||^2_\rho\right],\end{equation*}
\begin{equation*}\label{eA2 5 2}
||\widehat Z(\cdot)||_\rho^2\leq \frac{2(k^2_9+k_{10}^2)e^{-(\bar\rho_2+\rho) T}\mathbb E|\widehat Y(T)|^2+k_4c_4\max\{1,e^{-\bar\rho_2T}\}||\widehat Y(\cdot)||^2_\rho}{(1-k_5c_5-k_6c_6)\min\{1,e^{-\bar\rho_2T}\}},\end{equation*}
\begin{equation*}\label{eA2 5 3}
\mathbb E|\widehat X(0)|^2\leq \max\{1,e^{-\bar\rho_2T}\}\left[2(k^2_9+k_{10}^2)e^{-\rho T}\mathbb E|\widehat Y(T)|^2+k_4c_4||\widehat Y(\cdot)||^2_\rho\right].\end{equation*}
In particular, if $\bar\rho_2>0$, we have
\begin{equation*}\label{eA2 5 4}
||\widehat Z(\cdot)||_\rho^2\leq \frac{2(k^2_9+k_{10}^2)\mathbb E|\widehat Y(T)|^2+k_4c_4||\widehat Y(\cdot)||^2_\rho}{1-k_5c_5-k_6c_6}.\end{equation*}
\end{lemma}
Now we will give the proof of Theorem \ref{th A 2}.
Consider the map $\mathbb M\triangleq \mathbb M_2\circ\mathbb M_1$. It suffices to show that $\mathbb M$ is a contraction mapping under $||\cdot||_\rho$. In fact, for $(X_i(\cdot),Z_i(\cdot))\times X_i(0)\in L^2_{\mathbb F}(0,T;\mathbb R^m)\times L^2_{\mathbb F}(0,T;\mathbb R^m)\times L^2_{\mathcal F_0}(\Omega;\mathbb R^m),i=1,2$, let $Y_i\triangleq\mathbb M_1(X_i(\cdot),Z_i(\cdot),X_i(0))$ and $(\bar X_i(\cdot),\bar Z_i(\cdot),\bar X_i(0))\triangleq\mathbb M((X_i(\cdot),Z_i(\cdot),X_i(0)))$, by Lemmas \ref{Lemma A1} and \ref{Lemma A2}, we have
\begin{equation}\nonumber\begin{aligned}
&\mathbb E|\bar X_1(0)-\bar X_2(0)|^2+||\bar X_1(\cdot)-\bar X_2(\cdot)||_\rho^2+||\bar Z_1(\cdot)-\bar Z_2(\cdot)||_\rho^2\\
\leq & \left[\frac{1-e^{-\bar\rho_2T}}{\bar\rho_2}+\frac{\max\{1,e^{-\bar\rho_2T}\}}{(1-k_5c_5-k_6c_6)\min\{1,e^{-\bar\rho_2T}\}}+
\max\{1,e^{-\bar\rho_2T}\}\right]\left[2(k^2_9+k_{10}^2)e^{-\rho T}\mathbb E|\widehat Y(T)|^2+k_4c_4||\widehat Y(\cdot)||^2_\rho\right]\\
\leq & \left[\frac{1-e^{-\bar\rho_2T}}{\bar\rho_2}+\frac{\max\{1,e^{-\bar\rho_2T}\}}{(1-k_5c_5-k_6c_6)\min\{1,e^{-\bar\rho_2T}\}}+
\max\{1,e^{-\bar\rho_2T}\}\right]\left[2(k^2_9+k_{10}^2)\max\{1,e^{-\bar\rho_1 T}\}+k_4c_4\frac{1-e^{-\bar\rho_1T}}{\bar\rho_1}\right]\\
&\times\left[k^2_{10}\mathbb E|\widehat X(0)|^2+(k_1c_1+k^2_8)||\widehat X(\cdot)||_\rho^2+(k_2c_2+k_3c_3+k^2_9)||\widehat Z(\cdot)||_\rho^2\right].
\end{aligned}\end{equation}
Recalling that $\bar\rho_1=\rho-2\rho_1-k_1c_1^{-1}-k_2c_2^{-1}-k_3c_3^{-1}-k^2_7$ and  $\bar\rho_2=-\rho-2\rho_2-k_4c_4^{-1}-k_5c_5^{-1}-k_6c_6^{-1}$. Then by choosing suitable $\rho$, the first assertion is immediate. For the second assertion, since $2(\rho_1 + \rho_2) < -2k^2_5-2k^2_6-k_7^2$, we can choose a $\rho\in\mathbb R$, $0<c_5<\frac{1}{2} k^{-1}_5$, $0<c_6<\frac{1}{2} k^{-1}_6$ and sufficient large $c_1,c_2,c_3,c_4$ such that $\bar\rho_1>0,\quad\bar\rho_2>0,\quad 1-k_5c_5-k_6c_6>0.$
Then, using a similar method, we  get
\begin{equation}\nonumber\begin{aligned}
&\mathbb E|\bar X_1(0)-\bar X_2(0)|^2+||\bar X_1(\cdot)-\bar X_2(\cdot)||_\rho^2+||\bar Z_1(\cdot)-\bar Z_2(\cdot)||_\rho^2\\
\leq & \left[\frac{1}{\bar\rho_2}+\frac{1}{1-k_5c_5-k_6c_6}+1\right]\left[2k^2_9 +2k_{10}^2+\frac{k_4c_4}{\bar\rho_1}\right]\\
&\left[k^2_{10}\mathbb E|\widehat X(0)|^2+(k_1c_1+k^2_8)||\widehat X(\cdot)||_\rho^2+(k_2c_2+k_3c_3+k^2_9)||\widehat Z(\cdot)||_\rho^2\right].\quad\Box
\end{aligned}\end{equation}

\end{document}